\documentclass{amsart}
\usepackage{amssymb}
\usepackage{mathrsfs}
\usepackage{stmaryrd}
\usepackage{bbm}
\usepackage{oldgerm}
\usepackage[english]{babel}
\usepackage[T1]{fontenc}
\usepackage[latin1]{inputenc}
\usepackage[all]{xy}
\usepackage{mathabx}
\usepackage{graphicx}
\usepackage[pagebackref, colorlinks,linkcolor=red,anchorcolor=red,citecolor=blue]{hyperref}

\renewcommand*\backref[1]{\ifx#1\relax \else Cited on page(s) #1.\fi}

\newtheorem{theorem}[subsection]{Theorem}
\newtheorem{proposition}[subsection]{Proposition}
\newtheorem{corollary}[subsection]{Corollary}
\newtheorem{lemma}[subsection]{Lemma}
\newtheorem{conjecture}[subsection]{Conjecture}

\theoremstyle{definition}

\newtheorem{definition}[subsection]{Definition}
\newtheorem{remark}[subsection]{Remark}

\numberwithin{equation}{subsection}

\newcommand{\ol}{\overline}
\newcommand{\ul}{\underline}
\newcommand{\wt}{\widetilde}

\newcommand{\spec}{\mathrm{Spec}}

\newcommand{\alg}{\mathrm{alg}}
\newcommand{\gal}{\mathrm{Gal}}
\newcommand{\ch}{\mathrm{char}}

\newcommand{\fm}{\mathfrak m}

\newcommand{\Hom}{\mathrm{Hom}}
\newcommand{\pr}{\mathrm{pr}}
\newcommand{\id}{\mathrm{id}}

\newcommand{\wh}{\widehat}
\newcommand{\dt}{\mathrm{dt}}
\newcommand{\DT}{\mathrm{DT}}
\newcommand{\rk}{\mathrm{rk}}
\newcommand{\iso}{\xrightarrow{\sim}}
\newcommand{\sw}{\mathrm{sw}}

\newcommand{\rt}{\mathrm{t}}

\newcommand{\SW}{\mathrm{SW}}

\newcommand{\ord}{\mathrm{ord}}

\newcommand{\rR}{\mathrm{R}}

\newcommand{\sO}{\mathscr{O}}

\newcommand{\bZ}{\mathbb Z}
\newcommand{\sF}{\mathscr F}
\newcommand{\bA}{\mathbb A}
\newcommand{\bP}{\mathbb P}
\newcommand{\bQ}{\mathbb Q}
\newcommand{\bT}{\mathbb T}
\newcommand{\bG}{\mathbb G}

\newcommand{\bF}{\mathbb F}
\newcommand{\sG}{\mathscr G}
\newcommand{\sH}{\mathscr H}

\newcommand{\rh}{\mathrm{h}}
\newcommand{\sN}{\mathscr{N}}
\newcommand{\sL}{\mathscr{L}}
\newcommand{\sM}{\mathscr{M}}

\newcommand{\fp}{\mathfrak{p}}
\newcommand{\sK}{\mathscr K}
\newcommand{\cS}{\mathcal S}
\newcommand{\cI}{\mathcal I}
\newcommand{\cQ}{\mathcal Q}
\newcommand{\cX}{\mathcal X}
\newcommand{\cV}{\mathcal V}
\newcommand{\cU}{\mathcal U}

\newcommand{\cT}{\mathcal T}

\newcommand{\rc}{\mathrm c}
\newcommand{\rC}{\mathrm C}
\newcommand{\rlc}{\mathrm{lc}}
\newcommand{\rLC}{\mathrm{LC}}
\newcommand{\rNP}{\mathrm{NP}}
\newcommand{\rLNP}{\mathrm{LNP}}
\newcommand{\cZ}{\mathcal{Z}}

\newcommand{\rS}{\mathrm S}
\newcommand{\rLS}{\mathrm{LS}}

\newcommand{\Gr}{\mathrm{Gr}}

\begin{document}
\title[Semi-continuity of conductors, and ramification bound of nearby cycles]{Semi-continuity of conductors \\ and \\ ramification bound of nearby cycles}
\author{Haoyu Hu}
\address{Department of Mathematics, Nanjing University, Nanjing 210093, CHINA}
%\thanks{two funds}

\email{huhaoyu@nju.edu.cn, huhaoyu1987@163.com}

\begin{abstract}
For a constructible \'etale sheaf on a smooth variety of positive characteristic ramified along an effective divisor, the largest slope in Abbes and Saito's ramification theory of the sheaf gives a divisor with rational coefficients called the {\it conductor divisor}.
In this article, we prove decreasing properties of the conductor divisor after pull-backs. The main ingredient behind is the construction of \'etale sheaves with pure ramifications.

As applications, we first prove a lower semi-continuity property for conductors of \'etale sheaves on relative curves in the equal characteristic case, which supplement Deligne and Laumon's lower semi-continuity property of Swan conductors (\cite{lau}) and is also an $\ell$-adic analogue of Andr\'e's semi-continuity result of Poincar\'e-Katz ranks for meromorphic connections on complex relative curves (\cite{andre}). Secondly, we give a ramification bound for the nearby cycle complex of an \'etale sheaf ramified along the special fiber of a regular scheme semi-stable over an equal characteristic henselian trait, which extends a main result in a joint work with Teyssier (\cite{HT18}) and answers a conjecture of Leal (\cite{Leal}) in a geometric situation.

\end{abstract}

\maketitle
\tableofcontents
\section{Introduction}

\subsection{}\label{settingintro}
Let $\kappa$ be a perfect field of characteristic $p>0$, let $X$ be a smooth $\kappa$-scheme of finite type, let $D$ be a reduced Cartier divisor of $X$, let $U$ be the complement of $D$ in $X$ and let $j:U\to X$ be the canonical injection. We denote by $\{D_i\}_{1\leq i\leq m}$ the set of irreducible components of $D$. Let $\Lambda$ be a finite field of characteristic $\ell\neq p$ and let $\sF$ be a locally constant and constructible \'etale sheaf of $\Lambda$-modules on $U$.

\subsection{}
When $\dim_{\kappa}X\geq 2$, the ramification of local fields with imperfect residue fields has been for a long time an obstacle to study the ramification of $\sF$ along $D$ and to generalize the famous Grothendieck-Ogg-Shafarevich formula (\cite{SGA5}). In 1970's, Deligne began a program to measure the ramification of $\sF$ by restricting to smooth curves and to generalize Grothendieck-Ogg-Shafarevich formula to higher dimensions. It was firstly developed in \cite{lau2}. The idea of studying the ramification by restricting to curves has further progress (\cite{Wd1, Wd2, EK12} etc.).
From 2002, Abbes and Saito introduced a geometric method to study the ramification of local fields with imperfect residue fields that enables to investigate the ramification of an $\ell$-adic sheaf along a divisor (\cite{as1,as2,rc,mla, as3}). Different approaches have also been made in the same topic  (\cite{katochar, Ked07, Ked11, xl1,xl2} etc.). In 2016, an important breakthrough is achieved by Beilinson (\cite{bei}). Using Radon transform initiated by Brylinski, he defined the singular support for an $\ell$-adic sheaf on a smooth algebraic variety.  After, Saito constructed the characteristic cycle of an $\ell$-adic sheaf, as a cycle supported on the singular support, and proved an index formula that computes the Euler-Poincar\'e characteristic of an $\ell$-adic sheaf on a projective smooth variety that generalizing Grothendieck-Ogg-Shafarevich formula (\cite{cc}). In this article, we firstly focus on the connection between the ramification of $\sF$ along $D$ by restricting to curves and that by Abbes and Saito's approach. Our research tools are mainly the singular support and the characteristic cycle of an $\ell$-adic sheaf.

\subsection{}
 We denote by $\xi_i$ the generic point of $D_i$, by $K_i$ the fraction field of the complete discrete valuation ring $\wh\sO_{X,\xi_i}$, by $\ol K_i$ a separable closure of $K_i$ and by $G_{K_i}$ the Galois group of $\ol K_i/K_i$.
For each $1\leq i\leq m$, Abbes and Saito's ramification theory gives two decreasing filtrations $\{G^r_{K_i}\}_{r\in\mathbb Q_{\geq 1}}$ and $\{G^r_{K_i,\log}\}_{r\in\mathbb Q_{\geq 0}}$ of $G_{K_i}$ by closed normal subgroups, called {\it the ramification filtration} and {\it the logarithmic ramification filtration}, respectively. They generalize the classical upper numbering filtration (cf. subsection \ref{propfiltration}). The restriction $\sF|_{\spec(K_i)}$ corresponds to a finitely dimensional $\Lambda$-vector space $M_i$ with a continuous $G_{K_i}$-action. The two ramification filtrations give two slope decompositions
\begin{equation*}
M_i=\bigoplus_{r\in\mathbb Q_{\geq 1}}M_i^{(r)}\ \ \ \textrm{and}\ \ \ M_{i}=\bigoplus_{r\in\mathbb Q_{\geq 0}}M_{i,\log}^{(r)},
\end{equation*}
similar to the classical upper numbering slope decomposition. The largest rational number such that $M_i^{(r)}\neq 0$ (resp. $M_{i,\log}^{(r)}\neq 0$) is called {\it the conductor} (resp. {\it logarithmic conductor}) of $\sF$ at $\xi_i$ and is denoted by $\rc_{D_i}(\sF)$ (resp. $\rlc_{D_i}(\sF)$). We define two divisors with rational coefficients relative to $j_!\sF$ (Definition \ref{defCDTNP})
\begin{equation*}
  \rC_X(j_!\sF)=\sum_{i=1}^m\rc_{D_i}(\sF)\cdot D_i\ \ \ \textrm{and}\ \ \ \rLC_X(j_!\sF)=\sum_{i=1}^m\rlc_{D_i}(\sF)\cdot D_i,
\end{equation*}
which are called {\it the conductor divisor} and {\it the logarithmic conductor divisor}, respectively. First two main results of this article are the following:

\begin{theorem}[{Theorem \ref{Ctheorem}}]\label{Cintroduction}
 Let $f:Y\to X$ be a morphism of smooth $\kappa$-schemes of finite type. We assume that $f^*D=D\times_XY$ is a Cartier divisor of $Y$.
Then, we have
\begin{equation*}
f^*(\rC_X(j_!\sF))\geq \rC_Y(f^*j_!\sF).
\end{equation*}

\end{theorem}

\begin{theorem}[{Theorem \ref{LCtheorem}}]\label{LCintroduction}
We assume that $D$ is a divisor with simple normal crossings. Let $f:Y\to X$ be a morphism of smooth $\kappa$-schemes of finite type. We assume that $f^*D=D\times_XY$ is a Cartier divisor of $Y$.
\begin{itemize}
\item[(1)]
Then, we have
\begin{equation*}
f^*(\rLC_X(j_!\sF))\geq \rLC_Y(f^*j_!\sF).
\end{equation*}
\item[(2)]
We further assume that $D$ is irreducible. Let $\mathcal I(X, D)$ be the set of triples $(S,h:S\to X,x)$ where $h:S\to X$ is an immersion from a smooth $\kappa$-curve $S$ to $X$ such that $x=S\bigcap D$ is a closed point of $X$. Then, we have
\begin{equation*}
\rlc_D(\sF)=\sup_{\mathcal I(X, D)}\frac{\mathrm{lc}_x(\sF|_{S-\{x\}})}{m_x(h^*D)}.
\end{equation*}
\end{itemize}
\end{theorem}

\subsection{}
The inequality for total dimension divisors (Definition \ref{defCDTNP}) similar to that in Theorem \ref{Cintroduction} was firstly studied in \cite{wr} and extended by \cite[Theorem 4.2]{HY17} (cf. Theorem \ref{DTtheorem}). When $D$ is a smooth divisor, Theorem \ref{Cintroduction} was partially verified in \cite[\S 3]{HT18}, which is a main ingredient for proving \cite[Theorem 5.7]{HT18} (a special case of Theorem \ref{ramboundintroduction} below). In \cite{EK12}, Esnault and Kerz proved that the Swan conductor of $\sF$ after restricting to a curve is bounded by a Cartier divisor of $X$ supported on $D$ (there $X$ is assumed to be normal). In {\it loc. cit.}, they further predicted that Abbes and Saito's logarithmic ramification theory gives the sharp lower bound. Barrientos formulated Esnault and Kerz's conjecture in \cite{bar} and proved it when $\rk_{\Lambda}(\sF)=1$ and the base scheme is smooth. His result in {\it loc. cit.} is a special case of Theorem \ref{LCintroduction}.  Two formulas for Swan divisors similar to those in Theorem \ref{LCintroduction} (cf. Theorem \ref{SWtheorem}) was proved in \cite{Hu19IMRN}, which verified the conjecture in the case where the base scheme is smooth. Theorem \ref{LCintroduction} shows that the classical upper numbering conductor of $\sF$ after restricting to a curve is sharply bounded by the logarithmic conductor divisor $\rLC_X(j_!\sF)$. It enriches Esnault and Kerz's conjecture.

\subsection{}\label{strategyofCthmintro}
The strategy of proving Theorem \ref{Cintroduction} is the following. We work Zariski locally on $X$. Firstly, for an admissible $m$-tuple $(r_1,\ldots, r_m)\in\bQ_{> 1}$ and a constant non-zero vector $\omega$ in the trivial bundle $\bT^*X$, we construct a locally constant and constructible sheaf of $\Lambda$-modules $\sG$ on $U$ satisfying the following three nice ramification conditions (Theorem \ref{pureramifiction any D} and Proposition \ref{curiso=geniso+pure}):
\begin{itemize}
\item[(i)]
For each $1\leq i\leq m$,  the $G_{K_i}$-representation $\sG|_{\spec(K_i)}$ has only one slope with $c_{D_i}(\sG)=r_i$;
\item[(ii)] For each closed point $x\in D$, the singular support of $j_!\sG$ satisfies $SS(j_!\sG)\times_Xx=\langle\omega \rangle$ (we call the ramification of $\sG$ along $D$ is pure (Definition \ref{definitionpure}));
\item[(iii)]
For any $SS(j_!\sG)$-transversal immersion $h:Z\to X$ from a $\kappa$-curve $Z$ meeting $D$ at a closed point $z$, the slope decomposition of $h^*j_!\sG$ at $z$ has only one slope  with $\rC_Z(h^*j_!\sG)=h^*(\rC_X(j_!\sG))$.
\end{itemize}
Secondly, we reduce Theorem \ref{Cintroduction} to the case where $Y$ is a curve and $f:Y\to X$ is an immersion by d\'evissage (i.e., Proposition \ref{Ctheorem curve}). Finally, we find $\sG$ satisfying three conditions above such that each $\rc_{D_i}(\sG)$ is larger than and sufficiently close to $\rc_{D_i}(\sF)$ and that $f:Y\to X$ is $SS(j_!\sG)$-transversal. Then, applying \cite[Theorem 4.2]{HY17} to $\sF\otimes_{\Lambda}\sG$, we obtain Theorem \ref{Cintroduction}. The proof of Theorem \ref{LCintroduction} follows a similar strategy of demonstrating  \cite[Theorems 6.5 and 6.7]{Hu19IMRN}, suggested by Saito. It relies on Theorem \ref{Cintroduction} and asymptotic properties of (logarithmic) conductor divisors by tame covers.

\subsection{}
It is well known that algebraic $D$-modules are analogues of $\ell$-adic sheaves. Two local invariants of algebraic $D$-modules on smooth complex curves called the irregularity and the Poincar\'e-Katz rank are analogues of the Swan conductor and the upper numbering conductor, respectively. Andr\'e studied decreasing properties for both irregularities and Poincar\'e-Katz ranks after restricting to curves of meromorphic connections on complex surfaces \cite[Th\'eor\`emes 7.1.1 and 6.1.1]{andre}. Using the two results, he immediately obtained the lower semi-continuity for both irregularities and Poincar\'e-Katz ranks of meromorphic connections on complex relative curves (\cite[Corollaires 7.1.2 and 6.1.3]{andre}). The former is a conjecture of Malgrange motivated from Deligne and Laumon's lower semi-continuity of Swan conductors (\cite{lau}). From the observation above, we may ask: whether there is a lower semi-continuity result for upper numbering conductors of $\ell$-adic sheaves on relative curves (cf. \cite[Remark 21.1.4]{AFM})? Applying Theorem \ref{Cintroduction}, we prove the following result that answers the question in equal characteristic.

\begin{theorem}[{Theorem \ref{theoremlowsemicontchi}}]\label{theoremlowsemicontintro}
Let $T$ be an excellent Noetherian scheme of characteristic $p>0$, let $g:Y\to T$ be a separated and smooth morphism of finite type of relative dimension $1$, let $E$ be a reduced closed subscheme of $Y$ such that the restriction $f|_E:E\to T$ is quasi-finite and flat and let $V$ be the complement of $E$ in $Y$. Let $\Lambda$ be a finite field of characteristic $\ell$ $(\ell\neq p)$ and let $\sG$ be a locally constant and constructible sheaf of $\Lambda$-modules on $V$. For a point $t\in T$, we denote by $\ol t\to T$ an algebraic geometric point above $t$. Then, the function (Definition \ref{cDlcD})
\begin{align*}
\chi: T\to \mathbb Q,&\ \ \ t\mapsto \sum_{x\in E_{\ol t}}\rc_x(\sG|_{V_{\ol t}})
\end{align*}
is constructible and lower semi-continuous.
\end{theorem}

\subsection{}
Towards complements and generalizations of Deligne and Laumon's semi-continuity of Swan conductors for $\ell$-adic sheaves, there have been several works. Saito proved a semi-continuity property for total dimensions of stalks of vanishing cycles complexes, that was used to construct the characteristic cycle of an $\ell$-adic sheaf  (\cite{cc}). Two lower semi-continuity properties for total dimension divisors and Swan divisors of $\ell$-adic sheaves on a smooth fibration were proved in \cite{HY17,Hu19IMRN}. A semi-continuity property for singular supports and characteristic cycles of $\ell$-adic sheaves is obtained in \cite{HYselecta}. Recently, Takeuchi proved a semi-continuity result for local epsilon factors (\cite{Tak1}).

\subsection{}
The monodromy action on nearby cycles complex, vanishing cycles complex as well as \'etale cohomological groups of varieties over local fields is an important theme in \'etale cohomology. The study started in SGA $7$ (\cite{SGA7}), including  local monodromy theorem,  Milnor formula and semi-stable reduction criteria for abelian varieties. Under the semi-stability condition, the tameness of the monodromy action on the nearby cycles complex of the constant sheaf was studied in \cite{RZ} and \cite{Ill04}.  For arbitrary schemes over a trait, the wild ramification of the monodromy action is involved. Bloch's conductor formula  interpreted the alternating sum of the Swan conductors of the cohomology groups of sheaves in terms of intersection numbers (\cite{Bl}). In this direction, many efforts have been made (\cite{abbes, KS05, KS13, SaitoDirectimage} etc.).  Concerning the wild ramification for the monodromy action on each cohomology group of an arbitrary $\ell$-adic sheaf, Leal firstly formulated a conjecture below.

\subsection{}
Let $\cS$ be a henselian trait with a perfect residue field of characteristic $p>0$, $s$ its closed point, $\eta$ its generic point, $\ol\eta$  a geometric generic point and $G=\gal(\ol\eta/\eta)$. Let $h:\cX\to\cS$ be a morphism of finite type, $\cZ$ a reduced closed subscheme of $\cX$, $\cU$ the complement of $\cZ$ in $\cX$ and $\jmath:\cU\to\cX$ the canonical injection. We assume that $(\cX,\cZ)$ is a semi-stable pair over $\cS$ (see Definition \ref{semistablepair}). Let $\sH$ be a locally constant and constructible sheaf of $\Lambda$-modules on $\cU$ which is tamely ramified along those irreducible components of $\cZ$ flat over $\cS$ and let $\wt c$ be the maximum of the set of logarithmic conductors of $\sH$ at generic points of the special fiber $\cX_s$.

\begin{conjecture}[{Leal,  \cite{Leal}}]\label{Lealconjintroduction}
Assume that $h:\cX\to\cS$ is proper. Then, the upper numbering ramification subgroup $G^{\wt c+}_{\mathrm{up}}$ of $G$ acts trivially on each $H^i_c(\cU_{\ol\eta},\sH|_{\cU_{\ol\eta}})$.
\end{conjecture}

In the last part of this article, we aim at Leal's conjecture by giving a ramification bound of the nearby cycle complexe $R\Psi(\jmath_!\sH,h)$.

\begin{theorem}[{Theorem \ref{maintheoremnearbycycle} and Corollary \ref{Lealcoro}}]\label{ramboundintroduction}
Assume that $\cS$ is a henselization of a smooth $\kappa$-curve at a closed point. Then, the upper numbering ramification subgroup $G^{\wt c+}_{\mathrm{up}}$ of $G$ acts trivially on each $R^i\Psi(\jmath_!\sH,h)$. In particular, Conjecture \ref{Lealconjintroduction} is true when $\cS$ is a henselization of a smooth $\kappa$-curve at a closed point.
\end{theorem}

\subsection{}
Leal proved her conjecture when $S$ is equal characteristic, $f:\cX\to\cS$ is a relative curve and $\rk_{\Lambda}(\sH)=1$ in \cite{Leal}. Her proof relies on a conductor formula due to Kato and Saito (\cite{KS13}). In a joint work with Teyssier \cite{HT18}, we proved Theorem \ref{ramboundintroduction} under an extra condition that $f:\cX\to\cS$ is smooth. Our method there is purely local by applying the singular support and the characteristic cycle.  Now, by studying $\ell$-adic sheaves with pure ramification along arbitrary divisors, the following acyclic property for the operator $Rj_*$ is obtained and Leal's conjecture in a geometric setting can be fully solved. An analogue result of Theorem \ref{ramboundintroduction} in $D$-modules theory was proved by Teyssier in \cite[Proposition 7.1.2]{T19}.

\begin{theorem}[{Theorem \ref{keysteprambound}}]\label{keystepramboundintro}
We keep the notation and assumption in subsections \ref{settingintro}. Let $S$ be a smooth $\kappa$-curve, $s$ a closed point of $S$, $V$ the complement of $s$ in $S$ and $f:X\to S$ a $\kappa$-morphism. We further assume that $D=f^*(s)$ is a divisor with simple normal crossings and that the restriction $f_0:U=X-D\to V$ of $f:X\to S$ on $U$  is smooth. Let $\sN$ be a locally constant and constructible sheaf of $\Lambda$-modules on $V$ having only one logarithmic slope at $s$ with
\begin{equation*}
\rlc_s(\sN)>\max_{1\leq i\leq m}\{\rlc_{D_i}(\sF)\}.
\end{equation*}
Then, we have
\begin{equation*}
Rj_*(\sF\otimes_{\Lambda}f_0^*\sN)=j_!(\sF\otimes_{\Lambda}f_0^*\sN).
\end{equation*}
\end{theorem}

\subsection{}
In general, the ramification of $\sF\otimes_{\Lambda}f_0^*\sN$ is not pure along $D$. Hence, we cannot directly apply a base change property (Corollary \ref{stalkpullbackcurves}) to show the acyclicity. To solve the difficulty, we work Zariski locally on $X$. We may reduce to the case where the ramification of $\sN$ at $s$ is "irreducible" (Proposition \ref{propRj*=j!}). Then, after choosing a nice finite, surjective and radicial cover $\pi:X'\to X$, we obtain that $\pi^*f_0^*\sN$ has pure ramification along $D'=(\pi^*D)_{\mathrm{red}}$ (cf. Theorem \ref{pureMtheorem}).
The conductor of $\pi^*f_0^*\sN$ at generic points of $D'$ dominate that of $\pi^*\sF$. By verifying the purity of the ramification  of $\pi^*(f_0^*\sN\otimes_{\Lambda}\sF)$ along $D'$ (see Proposition \ref{sGCisoclinic+CXjG>CXjF=>Rj*sGsF=j!sGsF}), we obtain Proposition \ref{propRj*=j!} and Theorem \ref{keystepramboundintro}.

% Thanks to Theorem \ref{Cintroduction}, we can show that $\pi^*(\sF\otimes_{\Lambda}f_0^*\sN)$ also has pure ramification along $D'$. Finally, applying Corollary \ref{stalkpullbackcurves} to $\pi^*Rj_*(\sF\otimes_{\Lambda}f_0^*\sN)$, we prove Proposition \ref{keystepramboundintro}.

\subsection{}
The article is organized as follows. In \S \ref{ASramsection}, \S \ref{SSCCsection} and \S\ref{CDTLCSWsection}, we recall necessary results in Abbes and Saito's ramification theory, the theory of singular support and characteristic cycle of $\ell$-sheaves and introduce the notation of several divisors with ramification invariant coefficients. In \S \ref{pureramificationsection1}, we study \'etale sheaves with pure ramifications along divisors and we construct sheaves with nice ramification conditions stated in subsection \ref{strategyofCthmintro}. Relying on results in \S \ref{pureramificationsection1}, we prove Theorem \ref{Cintroduction} and Theorem \ref{LCintroduction} in \S \ref{proofCLC}. We focus on the proof of Theorem \ref{theoremlowsemicontintro} in \S \ref{semi-contcondsection}. In \S \ref{sectionpureram2}, we construct an \'etale sheave pulling-back from a curves with pure ramifications along a simple normal crossing divisor and prove Theorem \ref{keystepramboundintro}. We finally prove Theorem \ref{ramboundintroduction} in \S \ref{ramificationboundsection}.

\subsection*{Notation}
In this article, $k$ denotes a field. All smooth $k$-schemes are assumed to be separated and of finite type over $\spec(k)$ and all morphism between smooth $k$-schemes are assumed to be separated $k$-morphisms of finite type.

Let $X$ be a smooth $k$-scheme. We denote by $\bT^*X$ and $\bT X$ the cotangent bundle and the tangent bundle of $X$, respectively.  We say a closed subset $C$ of $\bT^*X$ is conical if it is invariant under the canonical $\bG_m$-action on $\bT^*X$. We always endow a reduced induced closed subscheme structure on $C$. We denote by $B(C)$ the image of $C$ in $X$ by the canonical projection $\pi:\bT^*X\to X$ and call it the base of $C$. Let $x$ be a point of the smooth $k$-scheme $X$.  We denote by $\bT^*_xX=\bT^*X\times_Xx$ (resp. $\bT_xX=\bT X\times_Xx$ and resp. $C_{x}=C\times_Xx$)  fibers of $\bT^*X$ (resp. $\bT^*X$ and resp. $C$) at $x$. Let $\ol x\to X$ be a geometry point above $x\in X$.
We denote by $\bT^*_{\ol x}X=\bT^*X\times_X\ol x$ (resp. $\bT_{\ol x}X=\bT X\times_X\ol x$ and resp. $C_{\ol x}=C\times_X\ol x$) the fibers of $\bT^*X$  (resp. $\bT_xX=\bT X\times_Xx$ and resp. $C$) at $\ol x$. 

Let $X$ be a smooth $k$-scheme with an \'etale map $\pi:X\to\bA^n_k$, let $0$ be the origin of $\bA^n_k$ and let $\omega$ be a non-zero vector of $\bT^*_0\bA^n_k$. For any closed point $x\in X$, we denote by $\langle\omega\rangle$ the $1$-dimensional $k(x)$-vector subspace of $\bT^*_xX$ spanned by $\omega$. For a closed subscheme $Z$ of $X$, we denote by $Z\cdot\langle\omega\rangle$ the unique closed conical subset of $\bT^*X$ with the fiber $\langle\omega\rangle$ at each closed point of $Z$.

A geometric point $\ol y\to Y$ above a point $y$ of a scheme $Y$ is called {\it algebraic geometric} if $k(\ol y)$ is an algebraic closure of $k(y)$.

For a scheme $Y$, $\mathrm{Div}(Y)$ denotes the group of Cartier divisors of $Y$. We call an element $E$ in $\mathrm{Div}(Y)\otimes_{\bZ}\bQ$ a Cartier divisor with rational coefficients of $Y$. For two Cartier divisors with rational coefficients $D$ and $E$ of $Y$, we write $D\geq E$ if $D-E$ is effective. For a Cartier divisor $D$ with rational coefficients and a point $\xi$ of height $1$ on a normal scheme $Y$, we denote by $m_\xi(D)$ the multiplicity of $D$ at $\xi$.

\section{Abbes and Saito's ramification theory}\label{ASramsection}

\subsection{}
In this section, $K$ denotes a complete discrete valuation field, $\sO_K$ its integer ring, $\fm_K$ the maximal ideal of $\sO_K$ and $F$ the residue field of $\sO_K$. We assume that $\ch(F)=p>0$. Let $\ol K$ be a separable closure of $K$, $G_K$ the Galois group of $\overline K/K$. If $F$ is perfect, we have a classical upper numbering ramification filtration $G^r_{K,\mathrm{up}}$ $ (r\in \bQ_{\geq 0})$ of $G_K$ (\cite{lf}). In general, Abbes and Saito defines two decreasing filtrations $G^r_K$ and $G^r_{K,\log}$ $(r\in\bQ_{\geq 0})$ of $G_K$ by closed normal subgroups. These two filtrations are called the ramification filtration and the logarithmic ramification filtration, respectively (\cite{as1}). They generalize the classical upper numbering filtration.

\subsection{}\label{propfiltration}
For any $r\in \bQ_{\geq 0}$, we put
\begin{equation*}
G^{r+}_K=\overline{\bigcup_{s\in \bQ_{>r}}G^s_K}\ \ \ \textrm{and}\ \ \ G^{r+}_{K,\log}=\overline{\bigcup_{s\in \bQ_{>r}}G^s_{K,\log}}.
\end{equation*}
We denote by $G_K^0$ the group $G_K$. For any $0<r\leq 1$, the subgroup $G^r_K=G_{K,\log}^0$ is the inertia subgroup $I_K$ of $G_K$ and $G^{1+}_K=G^{0+}_{K,\log}$ is the wild inertia subgroup $P_K$ of $G_K$. For any $r\in \bQ_{\geq 0}$, we have
$G^{r+1}_K\subseteq G^r_{K,\log}\subseteq G^{r}_K.$ If $F$ is perfect, we have $G^{r+1}_K=G^r_{K,\log}=G^{r}_{K,\mathrm{up}}$ for any $r\in \bQ_{\geq 0}$ (cf. \cite{as1}).

For any $r\in\bQ_{>1}$, the graded piece $\Gr^rG_K=G^r_K/G^{r+}_K$ is abelian, killed by $p$ and contained in the center of $P_K/G^{r+}_{K}$ (\cite[2.15]{as2}, \cite[Corollary 2.28]{wr} and \cite{gqsaito}). For any $r\in \bQ_{>0}$, the graded piece $\Gr_{\log}^rG_K=G^r_{K,\log}/G^{r+}_{K,\log}$ is abelian, killed by $p$ and contained in the center of $P_K/G^{r+}_{K,\log}$ (\cite[5.12]{as2}, \cite[Theorem 1.24]{saito cc}, \cite{as3}, \cite{xl1} and \cite{xl2}).

\subsection{}\label{slopedecom}
In the following of this section, $\Lambda$ denotes a finite field of characteristic $\ell$ $(\ell\neq p)$ and $M$ a finite dimensional $\Lambda$-module with a continuous $P_K$-action. The module $M$ has two decompositions
\begin{equation*}
M=\bigoplus_{r\in \mathbb Q_{\geq 1}}M^{(r)}\ \ \ \textrm{and}\ \ \ M=\bigoplus_{r\in \mathbb Q_{\geq 0}}M^{(r)}_{\log}
\end{equation*}
into $P_K$-stable submodules,  such that $M^{(1)}=M^{(0)}_{\log}=M^{P_K}$ and that, for any $r\in \mathbb Q_{>0}$,
\begin{align*}
(M^{(r+1)})^{G^{(r+1)+}_K}=M^{(r+1)}\ \ \ &\textrm{and}\ \ \ (M^{(r+1)})^{G^{r+1}_K}=\{0\};\\
(M^{(r)}_{\log})^{G^{r+}_{K,\log}}=M^{(r)}_{\log}\ \ \ &\textrm{and}\ \ \ (M^{(r)}_{\log})^{G^{r}_{K,\log}}=\{0\}.
\end{align*}
They are called the {\it slope decomposition} and the {\it logarithmic slope decompositions}, respectively.
Both of them are finite direct sums. We say that $M$ is {\it tame} if the action of $P_K$ on $M$ is trivial. We say that $M$ is {\it totally wild} if $M^{P_K}=\{0\}$.

We call a rational number $r\geq 1$ for which $M^{(r)}\neq 0$ a {\it slope} of $M$.  We denote by $\rS_K{(M)}$ the set of slopes of $M$. We denote by $\rc_K(M)$ the largest slope of $M$ and we call it the {\it conductor} of $M$. We say that $M$ is {\it isoclinic} if $\rS_K{(M)}$ has only one element. The {\it total dimension} of $M$ is defined by
\begin{equation*}
\dt_K(M)=\sum_{r\geq 1}r\cdot \dim_{\Lambda}M^{(r)}.
\end{equation*}
Let $r_1> r_2>\cdots >r_n\geq 1$ be all slopes of $M$. The {\it Newton polygon} of the slope decomposition of $M$ is defined to be the {\it concave} piecewise linear function $\rNP_K(M): [0,\dim_\Lambda M]\to\mathbb R$ connecting the following points:
\begin{equation*}
(0,0)\ \ \ \textrm{and}\ \ \   \left(\sum_{i=1}^d\dim_{\Lambda}M^{(r_i)}, \sum_{i=1}^d r_i\cdot\dim_{\Lambda}M^{(r_i)}\right),\ \ \ (d=1,2,\ldots, n).
\end{equation*}

We call a rational number $r\geq 0$ for which $M^{(r)}_{\log}\neq 0$ a {\it logarithmic slopes} of $M$. We denote by ${\rLS}_K{(M)}$ the set of logarithmic slopes of $M$. We denote by ${\rlc}_K(M)$ the largest logarithmic slope of $M$ and we call it the {\it logarithmic conductor} of $M$. We say that $M$ is {\it logarithmic isoclinic} if ${\rLS}_K{(M)}$ has only one element. The {\it Swan conductor} of $M$ is defined by
\begin{equation*}
\sw_K(M)=\sum_{r\geq 0}r\cdot \dim_{\Lambda}M_{\log}^{(r)}.
\end{equation*}
 Similarly, we have the {\it logarithmic Newton polygon} of $M$ and we denote it by $\rLNP_K(M)$.

We have (cf. \ref{propfiltration})
\begin{align*}
{\rlc}_{K}(M)+1&\geq \rc_K(M)\geq {\rlc}_K(M)\\
\sw_K(M)+\dim_{\Lambda}M&\geq  \dt_K(M)\geq \sw_K(M).
\end{align*}
In the special case where $F$ is perfect, we have
\begin{equation*}
{\rlc}_K(M)+1=\rc_{K}(M),\ \ \ \sw_K(M)+\dim_{\Lambda}M=\dt_K(M),
\end{equation*}
and ${\rlc}_K(M)$ (resp. $\sw_K(M)$) equals the classical upper numbering conductor (resp. the classical Swan conductor) of $M$.

Let $N$ and $N'$ be two finitely generated $\Lambda$-modules with continuous $P_K$-actions. If  $N$ is isoclinic and $\rc_K(N)>\rc_K(N')$, then $N\otimes_{\Lambda}N'$ is isoclinic and we have
\begin{align*}
\rc_K(N\otimes_{\Lambda}N')=\rc_K(N)\ \ \ \textrm{and}\ \ \ \dt_K(N\otimes_{\Lambda}N')=\dim_{\Lambda}N\cdot \dim_{\Lambda}N'\cdot \rc_K(N).
\end{align*}
If $N$ is logarithmic isoclinic and $\rlc_K(N)>\rlc_K(N')$, then  $N\otimes_{\Lambda}N'$ is logarithmic isoclinic and
\begin{align*}
\rlc_K(N\otimes_{\Lambda}N')=\rlc_K(N)\ \ \ \textrm{and}\ \ \ \sw_K(N\otimes_{\Lambda}N')=\dim_{\Lambda}N\cdot \dim_{\Lambda}N'\cdot \rlc_K(N).
\end{align*}

\subsection{}\label{K'/K}
Let $K'$ be a finite separable extension of $K$ contained in $\ol K$ and $e$ the ramification index of $K'/K$.
We have the canonical inclusion $G_{K'}\subseteq G_K$. For any $r\geq 1$, we have $G^{er}_{K'}\subseteq G^r_{K}$. When $K'/K$ is unramified, the inclusion is an equality. For any $r\geq 0$, we have $G^{er}_{K',\log}\subseteq G^r_{K,\log}$. When $K'/K$ is tamely ramified, the inclusion is an equality (cf. \cite{as1}).
Consider $M$ as a $\Lambda$-module with the continuous $P_{K'}$-action by the canonical restriction. We have
\begin{align}
&e\cdot \rc_K(M)\geq \rc_{K'}(M),\ \ \ e\cdot\dt_K(M)\geq \dt_{K'}(M) \label{condineq};\\
&e\cdot{\rlc}_K(M)\geq {\rlc}_{K'}(M),\ \ \ e\cdot\sw_K(M)\geq \sw_{K'}(M). \label{dtineq}
\end{align}
Inequalities in \eqref{condineq} (resp. \eqref{dtineq}) become equalities when $K'/K$ is unramified  (resp. $K'/K$ is tamely ramified).

We assume that $\ch(K)=p>0$. Let $L$ be a finite and purely inseparable extension of $K$ contained in an algebraic closure $ K^{\alg}$ of $K$ with the ramification index $\beta=e(L/K)$ and $G_L$ the Galois group of $K^{\alg}$ over $L$. The extension $L/K$ induces the canonical isomorphism $\varphi:G_L\iso G_K$. For any $r\geq 1$ we have $\varphi(G^{\beta r}_L)\subseteq G^r_K$.  For any $r\geq 0$, we have $\varphi(G^{\beta r}_{L,\log})\subseteq G^r_{K,\log}$ (cf. \cite{Hu19}).
Consider $M$ as a $\Lambda$-module with the continuous $P_{L}$-action by the canonical restriction.  We have
\begin{align}
&\beta\cdot \rc_{K}(M)\geq \rc_{L}(M),\ \ \  \beta\cdot\dt_K(M)\geq \dt_{L}(M);\label{logcondineq}\\
&\beta\cdot \rlc_{K}(M)\geq \rlc_{L}(M),\ \ \  \beta\cdot\sw_K(M)\geq \sw_{L}(M).\label{swineq}
\end{align}
When the residue field $F$ of $\sO_K$ is perfect, we have $\varphi(G^{r}_L)=G^r_K$, for any $r\geq 1$, and
\begin{align}\label{purinsNP}
\rNP_{K}(M)=\rNP_{L}(M).
\end{align}

\subsection{}
We assume that $\ch(K)=p>0$. Then, the total dimension $\dt_K(M)$ and the Swan conductor $\sw_K(M)$ satisfies the Hasse-Arf theorem, i.e., $\dt_K(M)$ and $\sw_K(M)$ are non-negative integers. It is firstly proved for a $p$-adic representation of finite monodromy of $G_K$ in \cite[3.4.3]{xl1}. Confer \cite[Theorem 3.11]{Hu19IMRN} for the torsion coefficients.

\subsection{}\label{centerchar}
We assume that $\ch(K)=p>0$, that $M$ is isoclinic of conductor $c>1$ and that $\Lambda$-contains a primitive $p$-th root of unity. The graded piece $\Gr^c G_K=G^c_K/G^{c+}_K$ is an $\bF_p$-vector space and $M$ has a faithful $\Gr^c G_K$-action. The module $M$ has a unique decomposition (\cite[Lemma 6.7]{rc})
\begin{equation*}
M=\bigoplus_{\chi\in X(c)} M_{\chi}
\end{equation*}
into $P_K$-stable submodules  $M_{\chi}$ $(n_\chi\in \mathbb Z_{\geq 0})$, where $X(c)$ is the set of isomorphic classes of  non-trivial characters $\chi: \Gr^c G_K\to \Lambda^{\times}$ and the restriction of $M_{\chi}$ on $\Gr^c G_K$ is a finite direct sum of $\chi$. We call it {\it the central character decomposition} of $M$. We say the central character decomposition of $M$ is isotypic if $M=M_{\chi}$ for some $\chi\in X(c)$.

We further assume  that $F$ is of finite type over a perfect field. We denote by $\ord:\ol K\to \bQ\bigcup\{\infty\}$ a valuation normalized by $\ord(K^{\times})=\bZ$. For any rational number $r$, we put
\begin{equation*}
\fm^r_{\overline K}=\big\{x\in \ol K^{\times}\,;\,\ord(x)\geq r\big\}\ \ \ \textrm{and}\ \ \ \fm^{r+}_{\overline K}=\big\{x\in \ol K^{\times}\,;\,\ord(x)> r\big\}.
\end{equation*}
Notice that $\fm^0_{\overline K}$ is the integral closure of $\sO_K$ in $\ol K$.  The residue field
$\ol F=\fm^0_{\overline K}/\fm^{0+}_{\overline K}$ is an algebraic closure of $F$. For each $r\in\mathbb Q$, the quotient $\fm^r_{\overline K}/\fm^{r+}_{\overline K}$ is a $1$-dimensional vector space over $\ol F$. For any rational number $r>1$, there exists an injective homomorphism called the characteristic form (cf.~\cite{wr})
\begin{equation}\label{charform}
\ch:\Hom_{\bF_p}(\Gr^r G_K,\bF_p)\to \Hom_{\overline F}(\fm^r_{\overline K}/\fm^{r+}_{\overline K},\Omega^1_{\sO_K}\otimes_{\sO_K}\ol F).
\end{equation}
We fix a non-trivial character $\psi_0:\bF_p\to \Lambda^{\times}$. Each $\chi\in X(c)$ uniquely factors as $\Gr^cG_K\to \bF_p\xrightarrow{\psi_0} \Lambda^{\times}$. We also denote by $\chi:\Gr^cG_K\to \bF_p$ the induced character and by $\ch(\chi)$ its characteristic form.

\subsection{}
Let $X$ be a connected regular scheme, $D\subseteq X$ an integral closed subscheme of codimension $1$, $U$ the complement of $D$ in $X$. Let $\xi$ the generic point of $D$, $X_{(\xi)}$ the henselization of $X$ at $\xi$, $\eta$ the generic point of $X_{(\xi)}$, $L$ the fraction field of $X_{(\xi)}$, $\ol L$ a separable closure of $L$ and $G_L$ the Galois group of $\ol L/L$. We assume that $\ch(k(\xi))=p>0$ and that $\ell$ is invertible in $X$. Let $\sF$ be a locally constant and constructible sheaf of $\Lambda$-modules on $U$.

\begin{definition}\label{cDlcD}
The restriction $\sF|_{\eta}$ associates to a finitely generated $\Lambda$-module with a continuous $G_L$-action. We call it {\it the ramification of $\sF$ at the generic point of $D$}. We denote by $\rS_{D}(\sF)$ the set of slopes of $\sF|_{\eta}$, by $\rc_{D}(\sF)$ the conductor $\rc_{L}(\sF|_{\eta})$, by $\dt_{D}(\sF)$ the the total dimension $\dt_{L}(\sF|_{\eta})$ and by $\rNP_{D}(\sF)$ the Newton Polygon $\rNP_{L}(\sF|_{\eta})$. We denote by $\rLS_{D}(\sF)$ the set of logarithmic conductors of $\sF|_{\eta}$,
by $\rlc_{D}(\sF)$ the logarithmic conductor $\rlc_{L}(\sF|_{\eta})$, by $\sw_{D}(\sF)$ the Swan conductor $\sw_{L}(\sF|_{\eta})$ and by $\rLNP_{D}(\sF)$ the logarithmic Newton Polygon $\rLNP_{L}(\sF|_{\eta})$.
\end{definition}

\section{Recall of singular support and characteristic cycle}\label{SSCCsection}

\subsection{}
Let $X$ be a smooth $k$-scheme and $C$ a closed conical subset in $\bT^*X$.   Let $f:Y\rightarrow X$ be a morphism of smooth $k$-schemes and $\ol y\rightarrow Y$ a geometric point above a point $y$ of $Y$. We say that $f:Y\rightarrow X$ is $C$-{\it transversal at} $y$ if $\ker(df_{\ol y})\bigcap (C\times_X\ol y)\subseteq\{0\}\subseteq \bT^*_{f(\ol y)}X$, where $df_{\ol y}:\bT^*_{f(\ol y)}X\rightarrow\bT^*_{\ol y}Y$ is the canonical map. We say that $f:Y\rightarrow X$ is $C$-{\it transversal}  if it is $C$-transversal at every point of $Y$. If $f:Y\rightarrow X$ is $C$-transversal, we define $f^\circ C$ to be the scheme theoretic image of $Y\times_XC$ in $\bT^*Y$ by the canonical map $df: Y\times_X\bT^*X\rightarrow \bT^*Y$. Notice that $df:Y\times_XC\rightarrow f^\circ C$ is finite and that $f^\circ C$ is also a closed conical subset of $\bT^*Y$ (\cite[Lemma 1.2]{bei}).
Let $g:X\rightarrow Z$ be a morphism of smooth $k$-schemes, $x$ a point of $X$, and $\ol x\rightarrow X$ a geometric point above $x$. We say that $g:X\rightarrow Z$ is $C$-transversal at $x$ if $\mathrm{im}(dg_{\ol x})\bigcap (C\times_X\ol x)=\{0\}$, where $dg_{\ol x}:\bT^*_{g(\ol x)}Z\rightarrow \bT^*_{\ol x}X$ is the canonical map. We say that $g:X\rightarrow Z$ is $C$-{\it transversal} if it is $C$-transversal at every point of $X$. Let $(g,f):Z\leftarrow Y\rightarrow X$ be a pair of morphisms of smooth $k$-schemes. We say that $(g,f)$ is $C$-{\it transversal} if $f:Y\rightarrow X$ is $C$-transversal and $g:Y\rightarrow Z$ is $f^\circ C$-transversal.

\subsection{}\label{NotationSS}
In the following of this section, let $\Lambda$ be a local Artin finite $\mathbb Z_{\ell}$-algebra, where $\ell$ is invertible in $k$.
Let $X$ be a smooth $k$-scheme and $\sK$ an object of $D^b_c(X,\Lambda)$. We say that $\sK$ is {\it micro-supported} on a closed conical subset $C$ of $\bT^*X$ if, for any $C$-transversal pair of morphisms $(g,f):Z\leftarrow Y\rightarrow X$ of smooth $k$-schemes, the morphism $g:Y\rightarrow Z$ is locally acyclic with respect to $f^*\sK$. If there exists a smallest closed conical subset of $\bT^*X$ on which $\sK$ is micro-supported, we call it the {\it singular support} of $\sK$ and denote it by $SS(\sK)$. We often endow $SS(\sK)$ a reduced induced closed subscheme structure.

\begin{theorem}[{\cite[Theorem 1.3]{bei}}]\label{BeiSS}
Let $X$ be a smooth $k$-scheme and $\sK$ an object of $D^b_c(X,\Lambda)$. The singular support $SS(\sK)$ exists. Moreover, each irreducible component of $SS(\sK)$ has dimension $\dim_kX$ if $X$ is equidimensional.
\end{theorem}

\subsection{}\label{sectiondf}
Let $X$ be a smooth $k$-scheme of equidimesnion $n$ and $f:X\to\bA^1_k$ a $k$-morphism. A constant non-zero vector at each fiber of $\bT^*\bA^1_k$ gives a section $\theta_0:\bA^1_k\rightarrow \bT^*\bA^1_k$ of $\bT^*\bA^1_k$ and hence a section $\theta:X\rightarrow \bT^*\bA^1_k\times_{\bA^1_k}X$ of  $\bT^*\bA^1_k\times_{\bA^1_k}X$ by the base change $f:X\to\bA^1_k$. The composition $df\circ\theta:X\rightarrow \bT^*X$ is a section of the bundle $\bT^*X$ and we abusively denote this section by $df(X)$. Let $C$ be a conical closed subset of $\bT^*X$ of equidimension $n$ and $x\in X$ a closed point such that the morphism $f:X-\{x\}\to \bA^1_k$ is $C$-transversal. The closed subset $df(X)$ and $C$ intersect at a single point in the fiber $\bT^*_xX$. Let $A$ be a $n$-cycle of $\bT^*X$ supported on $C$. We abusively denote by $(A, [df(X)])_{x}$ the intersection number of $A$ and $df(X)$ at the fiber $\bT_{x}^*X$.  It is independent of the choice of the section $\theta_0:\bA^1_k\to\bT^*\bA^1_k$.

\subsection{} In the following of this section, we further assume that $k$ is a perfect field of characteristic $p>0$ and that $\Lambda$ is a finite field of characteristic $\ell\neq p$.

\begin{theorem}[{\cite[Theorem 7.6]{cc}}]\label{SaCC}
Let $X$ be a smooth $k$-scheme and $\sK$ an object of $D^b_c(X,\Lambda)$. Then, there exists an unique cycle $CC(\sK)$ supported on $SS(\sK)$ satisfying the following Milnor type formula:  For any \'etale morphism $g:V\rightarrow X$, any morphism $f:V\rightarrow \bA^1_k$ and any closed point $v\in V$ such that  $f:V-\{v\} \rightarrow\bA^1_k$ is $g^*(SS(\sK))$-transversal, we have
\begin{equation*}
\sum_i(-1)^i\dt(\rR^i\Phi_{\ol v}(g^*\sK,f))=-(g^*(CC(\sK)),[df(V)])_{\bT^*V, v},
\end{equation*}
where $\ol v\to V$ is a geometric point above $v$, the complex $\rR\Phi_{\ol v}(g^*\sK,f)$ is the stalk at $\ol v\to V$ of the vanishing cycle of $g^*\sK$ with respect to $f:V\to \bA^1_k$, and $\dt(\rR^i\Phi_{\ol v}(g^*\sK,f))$ is the total dimension of $\rR^i\Phi_{\ol v}(g^*\sK,f)$ with respect to the canonical action of the absolute Galois group of the local field $K(\sO^{\mathrm{h}}_{\bA^1_k,f(v)})$.
\end{theorem}

The cycle $CC(\sK)$ in Theorem \ref{SaCC} is called {\it the characteristic cycle} of $\sK$.

\subsection{}\label{propSSCC}

Let $X$ be a connected and smooth $k$-scheme and $\sK$ an object of $D^b_c(X,\Lambda)$. For a distinguished triangle $\sK_1\to\sK\to\sK_2\to$ of objects in $D^b_c(X,\Lambda)$, we have
\begin{align*}
SS(\sK)&\subseteq SS(\sK_1)\bigcup SS(\sK_2),\\
CC(\sK)&=CC(\sK_1)+CC(\sK_2).
\end{align*}

 If $\sK$ is a perverse sheaf on $X$, then the support of $CC(\sK)$ is $SS(\sK)$ and each coefficient of $CC(\sK)$ is a positive integer (\cite[Proposition 5.14]{cc}).

Let $U$ be an open subscheme of $X$, $j:X\to U$ the canonical injection and $\sG$ an object of $D^b_c(U,\Lambda)$. Then, we have $CC(j_!\sG)=CC(Rj_*\sG)$ (\cite[Lemma 5.13]{cc}). If $\sG$ is further a perverse sheaf on $U$, we have $SS(j_!\sG)=SS(Rj_*\sG)$ (\cite[Lemma 5.15]{cc}).

If each cohomology of $\sK$ is a locally constant sheaf on $X$, then
\begin{equation*}
CC(\sK)=(-1)^{\dim_kX}\sum_{i}\rk_{\Lambda}(\mathcal H^i(\sK))\cdot[\bT^*_XX]
\end{equation*}

If $X$ is a smooth $k$-curve and $\sK$ is a constructible sheaf of $\Lambda$-modules on $X$, then
\begin{equation*}
CC(\sF)=-\rk_{\Lambda}\sF\cdot [\bT^*_XX]-\sum_{x\in |X|}\left(\dt_x(\sF)-\dim_{\Lambda}(\sF|_{\ol x})\right)\cdot [\bT^*_xX]
\end{equation*}
where $\rk_{\Lambda}\sF$ denotes the generic rank of $\sF$.

\subsection{}\label{propCtrans}
Let $h:W\to X$ be a morphism between smooth equidimensional $k$-schemes, $C$ a closed conical subset of $\bT^*X$ of equidimensional $\dim_kX$. We say that $h:W\to X$ is {\it properly $C$-transversal} if it is $C$-transversal and if $C\times_XW$ is equidimensional $\dim_kW$. Consider the canonical morphisms
\begin{align*}
\bT^*X\xleftarrow{\pr_1} \bT^*X\times_XW\xrightarrow{dh} \bT^*W.
\end{align*}
We assume that $h:W\to X$ is properly $C$-transversal. Let $A$ be an $n$-cycle of $\bT^*X$ supported on $C$. We set
\begin{equation*}
h^!A=(-1)^{\dim_kX-\dim_kW}dh_*(\pr_1^*A).
\end{equation*}

Let $\sK$ be an object of $D^b_c(X,\Lambda)$ and $g:Z\to X$ a properly $SS(\sK)$-transversal morphism between smooth $k$-schemes. Then we have \cite[Theorem 7.6]{cc}
\begin{equation*}
CC(g^*\sK)=g^!(CC(\sK)).
\end{equation*}

\subsection{}
Let $X$ be a smooth $k$-scheme of equidimension $n\geq 1$, $D$ an integral Cartier divisor on $X$, $U$ the complement of $D$ in $X$, $j:U\rightarrow X$ the canonical injection. We denote by $\xi$ be the generic point of $D$, by $X_{(\xi)}=\spec(\sO_K)$ the henselization of $X$ at $\xi$, by $\eta=\spec(K)$ the generic point of $X_{(\xi)}$, by $\ol K$ a separable closure of $K$, by $G_K$ the Galois group of $\ol K$ over $K$ and by $\ol{k(\xi)}$ an algebraic closure of $k(\xi)$.

We assume that $\Lambda$ contains a primitive $p$-th root of unity. Let $\sG$ be a locally constant and constructible sheaf of $\Lambda$-modules on $U$. We denote by $M$ the finitely generated $\Lambda$-module with a continuous $G_K$-action corresponding to $\sG|_{\eta}$.  We have the slope decomposition and central character decompositions (subsections \ref{slopedecom} and \ref{centerchar})
\begin{equation*}
M=\bigoplus_{s\in\mathbb Q\geq 1}M^{(s)}\ \ \ \textrm{and}\ \ \ M^{(r)}=\bigoplus_{\chi\in X(r)}M^{(r)}_\chi\ \ \textrm{(for}\; r>1).
\end{equation*}
For each $r>1$ and for each $\chi\in X(r)$, we have the characteristic form $\ch(\chi)$ (cf. \eqref{charform})
\begin{equation*}
\mathrm{char}(\chi): \fm^r_{\overline K}/\fm^{r+}_{\overline K}\rightarrow \Omega^1_{\sO_{K}}\otimes_{\sO_{K}}\overline{\kappa(\xi)}
\end{equation*}
  Let $\kappa(\xi)_{\chi}$ be a field of definition of $\mathrm{char}(\chi)$ which is a finite extension of $\kappa(\xi)$ contained in $\ol{\kappa(\xi)}$. The characteristic form $\mathrm{char}(\chi)$ defines a line $L_\chi$ in $\bT^*X\times_X\spec(\kappa(\xi)_{\chi})$. Let $\overline L_{\chi}$ be the closure of the image of $L_{\chi}$ in $\bT^*X$. After removing a closed subscheme $Z\subseteq D$ of codimension 2 in $X$, we may assume that $D_0=D-Z$ is smooth over $\spec(k)$ and that the ramification of $\sG$ along $D_0$ is {\it non-degenerate} (\cite[Definition 3.1]{wr}). Roughly speaking, the ramification of an \'etale sheaf along a smooth divisor is non-degenerate if its ramification along the divisor is controlled by its ramification at generic points of the divisor. Using Abbes and Saito's ramification theory, the characteristic cycle of $j_!\sG$ on $X_0=X-Z$ can be computed explicitly as follows (\cite[Theorem 7.14]{cc})
  \begin{equation}\label{wrcbcc}
CC(j_!\sG|_{X_0})=(-1)^n\bigg(\rk_{\Lambda}\sG\cdot[\bT^*_{X_0}X_0]+\dim_{\Lambda}M^{(1)}\cdot[\bT^*_{D_0}X_0]+\sum_{r>1, \chi\in X(r)}\frac{r\cdot\dim_{\Lambda}M^{(r)}_{\chi}}{[\kappa(\xi)_{\chi}:\kappa(\xi)]}[\overline L_{\chi}]\bigg).
\end{equation}
The singular support $SS(j_!\sG|_{X_0})$ is the support of $CC(j_!\sG|_{X_0})$ (\cite[Proposition 3.15]{wr}).
When $\sG$ is not trivial, the non-degeneracy of the ramification of $\sG$ along $D_0$ implies that for every geometric point $\bar x\rightarrow D_0$, the fiber $SS(j_!\sG)\times_X \bar x$ is a finite union of $1$-dimensional vector spaces in $\bT^*_{\bar x}X$.

\vspace{0.2cm}
Motivated by the notion of non-degenerate ramification, we give a broader convention called {\it pure ramification} in the following. In \S \ref{pureramificationsection1} and \S \ref{sectionpureram2},  we construct \'etale sheaves with pure ramifications which play a crucial role in this article.

\begin{definition}\label{definitionpure}
Let $X$ be a smooth $k$-scheme, $D$ a reduced Cartier divisor on $X$, $U$ the complement of $D$ in $X$, $j:U\to X$ the canonical injection and $\sG$ a locally constant and constructible sheaf of $\Lambda$-modules on $U$. We say that the ramification of $\sG$ along $D$ is {\it pure at a point $x$ of} $D$ if, for a geometric point $\overline x\to X$ above $x$, fiber $SS(j_!\sG)\times_X \bar x$ is a finite union of $1$-dimensional $k(\ol x)$-vector spaces in $\bT^*_{\bar x}X$.  We say that the ramification of $\sG$ along $D$ is {\it pure} if it is pure at each point of $D$.
\end{definition}

\vspace{0.2cm}

Using Saito's notion of $\sF$-transversality (\cite[\S 8]
{cc}), the transversality condition ensures the following base change isomorphism for the direct image of open immersions.

\begin{proposition}[{cf. \cite[Corollary 2.13]{HT18}}]\label{TBC}
Let $h:Z\to X$ be a morphism of connected and smooth $k$-schemes, $D$ an effective Cartier divisor of $X$, $U$ the complement of $D$ in $X$ and $j:U\to X$ the canonical injection. We assume that $h^*Z=Z\times_XD$ is an effective Cartier divisor of $Z$. We consider the following Cartesian diagram
\begin{equation*}
\xymatrix{\relax
V\ar[d]_{j'}\ar[r]^{h'}\ar@{}|-{\Box}[rd]&U\ar[d]^j\\
Z\ar[r]_h&X}
\end{equation*}
Let $\sG$ be a locally constant and constructible sheaf of $\Lambda$-modules on $U$ such that $h:Z\to X$ is $SS(j_!\sG)$-transversal. Then, we have an isomorphism
\begin{equation}\label{TBCF}
h^*\rR j_*\sG\cong \rR j'_*{h'}^*\sG.
\end{equation}
\end{proposition}

The morphism $h:Z\to X$ is the composition of the graph $\Gamma: Z\to Z\times_kX$ of $f: Z\to X$ and the second projection $\pr_2:Z\times_kX\to X$. We have the following commutative diagram with Cartesian squares
\begin{equation*}
\xymatrix{\relax
V\ar[d]_{j'}\ar[r]^-(0.5){\Gamma'}\ar@{}|-{\Box}[rd]& Z\times_k U\ar[d]^{g}\ar[r]^-(0.5){\pr'_2}\ar@{}|-{\Box}[rd]&U\ar[d]^j\\
Z\ar[r]_-(0.5){\Gamma}&Z\times_kX\ar[r]_-(0.5){\pr_2}&X}
\end{equation*}
Since $\pr_2$ is smooth, applying the smooth base change theorem to the right square, we have
\begin{equation*}
h^*\rR j_*\sG=\Gamma^*\pr_2^*\rR j_*\sG\xrightarrow{\sim}\Gamma^*\rR g_*{\pr'_2}^*\sG.
\end{equation*}
To obtain the isomorphism \eqref{TBCF}, we are left to show
\begin{equation*}
\Gamma^*\rR g_*({\pr'_2}^*\sG)\cong \rR j'_*\Gamma'^*({\pr'_2}^*\sG)
\end{equation*}
 Since $h:Z\to X$ is $SS(j_!\sG)$-transversal, we see that the closed immersion $\Gamma:Z\to Z\times_kX$ is $\pr^{\circ}_2(SS(j_!\sG))$-transversal (\cite[Lemma 2.2]{bei}). Since $\pr_2$ is smooth, we see that $SS(\pr^*_2j_!\sG)=\pr^{\circ}_2(SS(j_!\sG))$ (\cite[Theorem 1.4]{bei}). Hence $\Gamma:Z\to Z\times_kX$ is $SS(g_!(\pr'^*_2\sG))$-transversal. Thus we are reduced to prove \eqref{TBCF} in the case where $h:Z\to X$ is a closed immersion, which is done in \cite[\S 2]{HT18}.

\begin{corollary}\label{stalkpullbackcurves}
Let $X$ be a smooth $k$-scheme, $D$ an effective Cartier divisor of $X$, $U$ the complement of $D$ in $X$ and $j:U\to X$ the canonical injection. Let $\sG$ be a locally constant and constructible sheaf of $\Lambda$-modules on $U$ such that the ramification of $\sG$ is pure at a closed point $x$ of $D$. Let $S$ be a smooth $k$-curve and $h:S\to X$ an $SS(j_!\sF)$-transversal quasi-finite morphism such that $s=(h^*D)_{\mathrm{red}}$ is a closed point of $S$ with $x=h(s)$, $V$ the complement of $s$ in $S$ and $\jmath: V\to S$ the canonical injection. Then, we have
\begin{equation}\label{bcfcurve}
R\jmath_*(\sG|_V)\cong h^*Rj_*\sG.
\end{equation}
In particular, if the ramification of $\sG|_V$ at $s$ is totally wild (see subsection \ref{slopedecom}), then the stalk $(\rR j_*\sG)|_{\ol x}$ at a geometric point $\ol x$ above $x$ is $0$.
\end{corollary}

Indeed, the isomorphism \eqref{bcfcurve} is from Proposition \ref{TBC}.
Let $\overline s\to S$ be a geometric point above $s\in S$, $\eta=\spec(K)$ the generic point of the strict henselization $S_{(\ol s)}$, $\overline K$ a separable closure of $K$, $I$ the Galois group of $\ol K/K$ and $P$ the wild ramification subgroup of $I$. We denote by $M$ the $\Lambda$-module with a continuous $I$-action associated to $\sG|_{\eta}$. Hence
we have
\begin{equation*}
(Rj_*\sG)|_{\ol x}\cong (h^*Rj_*\sG)_{\ol s}=(R\jmath_*(\sG|_V))_{\ol s}=\rR \Gamma(I,M)=R\Gamma(I/P,\rR \Gamma(P,M)).
\end{equation*}
Since the ramification of $\sG|_V$ is totally wild at $s\in S$, we have $\rR^0 \Gamma(P,M)=M^P=0$.
Since $P$ is a pro-$p$-group and $p$ is invertible in $\Lambda$, we have $\rR^i \Gamma(P,M)=0$ for $i\neq 0$. Hence $\rR \Gamma(P,M)=0$ and hence $(Rj_*\sG)|_{\ol x}=0$.

\begin{remark}\label{manytranscurve}
We assume that $k$ is algebraically closed. Let $X$ be a smooth $k$-scheme, $D$ a reduced Cartier divisor of $X$, $x$ a closed point of $D$ and $C$ a closed conical subset of $\bT^*_xX$ of dimension $1$. Then, we can find a smooth $k$-curve $S$ and an immersion $h:S\to X$ such that $x=S\bigcap D$ is a closed point of $D$ and that $h:S\to X$ is $C$-transversal at $x$. If $x$ is a smooth locus of $D$, we allow a further assumption that $S$ and $X$ meet transversally at $x$.

Indeed, we only need to treat the case where $\dim_k X\geq 2$. The conical set $C$ is a finite union of $1$-dimensional $k$-subspaces $L_1,\ldots, L_n$ of $\bT^*_xX$. We denote by $H_1,\ldots, H_n$ their dual hyperplanes in $\bT_xX$. We denote by $\wt X$ the blow-up of $X$ at $x$, by $E$ the exceptional divisor on $\wt X$ and by $\wt D$ the strict transform of $D$ in $\wt X$. The intersection $\wt D\bigcap E$ is of codimensional $1$ in $E$. By the canonical isomorphism $\varphi:E\xrightarrow{\sim} \bP(\bT_xX)$, we see that $\bigcup_i \bP(H_i)\bigcup \varphi(\wt D\bigcap E)$ is codimensional $1$ in $\bP(\bT_xX)$. We can find a closed point $z$ in $\bP(\bT_xX)$ which is not contained in $\bigcup_i \bP(H_i)\bigcup \varphi(\wt D\bigcap E)$. It give rise to a non-zero vector $l_z$ in $\bT_xX$ and we can find  an immersion $h:S\to X$ from a smooth $k$-curve to $X$ such that $x\in h(S)$ and that $\bT^*_xS=\langle l_z\rangle$. Since $z$ is not contained in $\bP(\bigcup_i H_i)$, any non-zero vector in $C=\bigcup_i L_i$ has non-zero image in $\bT^*_xS$ by the canonical surjection $dh_x:\bT^*_xX\to\bT^*_xS$. Hence $h:S\to X$ is $C$-transversal at $x$. Since $z\not\in \varphi(\wt D\bigcap E)$, the strict transform $\wt S$ of $S$ is not contained in $\wt D$. Hence $S\not\subseteq  D$. If $x$ is a smooth locus of $D$, we have $\varphi(\wt D\bigcap E)=\bP(\bT_xD)$. Hence $z\not\in \bP(\bT_xD)$ implies that $m_x(h^*D)=1$.
 \end{remark}

\section{Divisors from ramification invariants}\label{CDTLCSWsection}

\subsection{}\label{4conductors}

In this section, $k$ denotes a perfect field of characteristic $p>0$, $X$  a smooth $k$-scheme, $D$ a reduced Cartier divisor on $X$, $\{D_i\}_{i\in I}$ the set of irreducible components of $D$, $U$ the complement of $D$ in $X$ and $j:U\to X$ the canonical injection. Let $\Lambda$ be a finite field of characteristic $\ell\neq p$ and let $\sF$ be a locally constant and constructible sheaf of $\Lambda$-modules on $U$.

\begin{definition}\label{defCDTNP}
We keep the terminology of Definition \ref{cDlcD}.  We define the {\it conductor divisor} of $j_!\sF$ on $X$ and denote by $\rC_X(\sF)$ the Cartier divisor with rational coefficients
\begin{equation*}
\rC_X(j_!\sF)=\sum_{i\in I} \rc_{D_i}(\sF)\cdot D_i.
\end{equation*}
We define the {\it total dimension divisor} of $j_!\sF$ on $X$ and denote by $\DT_X(j_!\sF)$ the Cartier divisor
\begin{equation*}
\DT_X(j_!\sF)=\sum_{i\in I} \dt_{D_i}(\sF)\cdot D_i.
\end{equation*}
We define the {\it logarithmic conductor divisor} of $j_!\sF$ on $X$ and denote by $\rLC_X(j_!\sF)$ the Cartier divisor with rational coefficients
\begin{equation*}
\rLC_X(j_!\sF)=\sum_{i\in I} \rlc_{D_i}(\sF)\cdot D_i.
\end{equation*}
We define the {\it Swan divisor} of $j_!\sF$ on $X$ and denote by $\SW_D(j_!\sF)$ the Cartier divisor
\begin{equation*}
\SW_X(j_!\sF)=\sum_{i\in I} \sw_{D_i}(\sF)\cdot D_i.
\end{equation*}

\end{definition}

\begin{theorem}[{cf. \cite{wr, HY17}}]\label{DTtheorem}
Let $f:Y\to X$ be a morphism of smooth $k$-schemes. We assume that $f^*D=D\times_XY$ is a Cartier divisor of $Y$.
\begin{itemize}
\item[(1)] Then, we have {\rm(\cite[Theorem 4.2]{HY17})}
\begin{equation*}
f^*(\DT_X(j_!\sF))\geq \DT_Y(f^*j_!\sF).
\end{equation*}
\item[(2)]
Assume that $Y$ is a smooth $k$-curve, that $f:Y\to X$ is an immersion such that $D$ is smooth at the closed point $x=Y\bigcap D$ and that the ramification of $\sF$ at $x$ is non-degenerate. Then the following there conditions are equivalent \rm{(\cite[Propositions 2.22 and 3.8, Corollary 3.9]{wr})}:

(a) $f:Y\to X$ is $SS(j_!\sF)$-transversal at $x$;

(b) $\rNP_x(\sF|_{Y-\{x\}})=m_x(f^*D)\cdot\rNP_{D_i}(\sF)$, where $D_i$ is the unique irreducible component of $D$ containing $x$.

(c) $f^*(\DT_X(j_!\sF))= \DT_Y(f^*j_!\sF)$.
\end{itemize}
\end{theorem}

\begin{remark}
We are under assumptions of Theorem \ref{DTtheorem}. If three equivalent conditions in Part (2) are valid, we have
\begin{equation*}
f^*(\rC_X(j_!\sF))= \rC_Y(f^*j_!\sF)
\end{equation*}
by condition (b).
\end{remark}

\begin{theorem}[cf. {\cite{Hu19IMRN}}]\label{SWtheorem}
We assume that $D$ is a divisor with simple normal crossings. Let $f:Y\to X$ be a morphism of smooth $k$-schemes. We assume that $f^*D=D\times_XY$ is a Cartier divisor of $Y$.
\begin{itemize}
\item[(1)]
Then, we have {\rm (\cite[Theorem 6.6]{Hu19IMRN})}
\begin{equation*}
f^*(\SW_X(j_!\sF))\geq \SW_Y(f^*j_!\sF).
\end{equation*}
\item[(2)]
We further assume that $D$ is irreducible. Let $\mathcal I(X, D)$ be the set of triples $(S,h:S\to X,x)$ where $h:S\to X$ is an immersion from a smooth $k$-curve $S$ to $X$ such that $x=S\bigcap D$ is a closed point of $X$. Then, we have {\rm (\cite[Theorem 6.5]{Hu19IMRN})}
\begin{equation*}
\sw_D(\sF)=\sup_{\mathcal I(X, D)}\frac{\sw_x(\sF|_{S-\{x\}})}{m_x(h^*D)}.
\end{equation*}
\end{itemize}
\end{theorem}

In \S \ref{proofCLC}, we will state and prove analogous results of Theorem \ref{DTtheorem} (1) and Theorem \ref{SWtheorem} for conductor divisors and logarithmic conductor divisors, respectively (see Theorem \ref{Ctheorem} and Theorem \ref{LCtheorem}).

\section{\'Etale sheaves with pure ramifications (I)}\label{pureramificationsection1}

\subsection{}\label{IXQX}
In this section, $k$ denotes a perfect field of characteristic $p>0$, $X$ a connected and smooth $k$-scheme of dimension $n\geq 2$, $D$ a reduced effective Cartier divisor of $X$, $\{D_i\}_{1\leq i\leq d}$ the set of irreducible components of $D$, $U$ the complement of $D$ in $X$ and $j:U\to X$ the canonical injection. Let $Y=\spec(k[y_1,\ldots,y_n])$ be an affine space of dimension $n$.

We denote by  $\mathcal I(X, D)$ the set of triples $(S,h:S\to X,x)$ where $h:S\to X$ is an immersion from an affine connected smooth $k$-curve $S$ to $X$ such that $x=S\bigcap D$ is a closed point of $X$, and by $\mathcal Q(X, D)$ the set of triples $(S, h:S\to X, x)$ where $h:S\to X$ is a quasi-finite morphism from an affine connected smooth $k$-curve $S$ to $X$ such that $x=h(S)\bigcap D$ is a closed point of $X$ and that $h^{-1}(x)$ is a single closed point of $S$. Notice that $\cI(X, D)$ is a subset of $\cQ(X, D)$. For an element $(S, h:S\to X, x)$ of $\cI(X,D)$ or $\cQ(X,D)$, we put $S_0=S-\{h^{-1}(x)\}$.

\begin{proposition}\label{tranfibers}
Assume that $k$ is algebraically closed. Let $C$ be an irreducible closed conical subset of $\bT^*X$ such that $\dim_kC=n$ and that $B(C)\subseteq D$. Let $(x,\omega)\in C(k)$ be a smooth point of $C$ and let $L_x$ be a closed conical subset of $\bT^*_xX$ of dimension $1$. We assume that $\omega\not\in L_x$. Then,
\begin{itemize}
\item[(1)]
When $p\geq 3$, after replacing $X$ by a Zariski open neighborhood of $x$, there are two morphisms $f:X\to \bA^{n-1}_k$ and $g:\bA^{n-1}_k\to \bA^1_k$ satisfying:
\begin{itemize}
\item[a)]  they are smooth and $f|_{D_i}:D_i\to \bA^{n-1}_k$ is quasi-finite and flat, for each $D_i$ containing $x$;
\item[b)] the morphism $f:X\to \bA^{n-1}_k$ is $L_x$-transversal at $x$, the composition $g\circ f:X-\{x\}\to\bA^1_k$ is $C$-transversal and the section $d(g\circ f)(X)$ meets $C$ transversely at $(x,\omega)$ (cf. \ref{sectiondf}).
\end{itemize}
\item[(2)] When $p=2$, $(X,C,L_x, (x,\omega))$ may have the same geometric property as in (1). If not the case, we denote by $q:\bA^1_X\to X$ the canonical projection. After replacing $\bA^1_X$ by a Zariski open neighborhood of $(x,0)$, there are two morphisms $f:\bA^1_X\to \bA^{n}_k$ and $g:\bA^{n}_k\to \bA^1_k$ satisfying:
\begin{itemize}
\item[c)]  they are smooth and $f|_{\bA^1_{D_i}}:\bA^1_{D_i}\to \bA^{n}_k$ is quasi-finite and flat, for each $\bA^1_{D_i}$ containing $(x,0)$;
\item[d)] the morphism $f:\bA^1_X\to \bA^{n}_k$ is $( q^\circ L_x$-transversal at $(x,0)$, the composition $g\circ f:\bA^1_X-\{(x,0)\}\to\bA^1_k$ is $q^{\circ}C$-transversal and the section $d(g\circ f)(\bA^1_X)$ meets $q^{\circ}C$ transversely at $((x,0),q^\circ\omega)$, where we consider $q^\circ\omega$ and $q^\circ L_x$ as the images of $\omega$ and $L_x$ by canonical injection $dq:\bT^*X\times_X\bA^1_X\to\bT^*\bA^1_X$, respectively.
\end{itemize}

\end{itemize}
\end{proposition}

\begin{proof}
(1) This is a local question. After replacing $X$ by a Zariski neighborhood of $x$, we have an affine and \'etale morphism $\pi:X\to Y$.  We may assume that $y=\pi(x)$ is the origin of $Y$, that $\{D_i\}_{1\leq i\leq d}$ is the set of irreducible components of $D$ and that each $D_i$ contains $x$. Let $y_1,\ldots,y_n$ be the canonical local coordinate of $X$ at $x$ from the \'etale morphism $\pi:X\to Y$. Let $\theta_1,\theta_2,\ldots,\theta_r$ be non-zero vectors of $\bT^*_xX$ such that $L_x=\bigcup_{b=1}^r\langle \theta_b\rangle$. We have $k$-linear combinations
\begin{equation}
\omega=\sum_{i=1}^n \lambda_i dy_i\ \ \ \textrm{and}\ \ \ \theta_b=\sum_{i=1}^n \lambda'_{bi} dy_i\ \ (b=1,2,\ldots, r).
\end{equation}
 By the proof of \cite[Proposition 5.19]{cc}, after replacing $X$ by a Zariski open neighborhood of $x$, we have a smooth morphism $r_0:X\to \bA^1_k$ associating to a $k$-homomorphism
 \begin{equation}\label{saito's r}
 r_0^\sharp:k[t]\to\Gamma(X,\sO_X),\ \ \ t\mapsto \sum_{i=1}^n\lambda_i y_i+ R(y_1,\ldots, y_n),
 \end{equation}
where $R(y_1,\ldots, y_2)$ is a quadratic form with $k$-coefficients, such that the section $dr_0(X)$ meets $C$ transversely at $(x,\omega)\in \bT^*X$. Notice that $o=r_0(x)$ is the origin of $\bA^1_k$.

Let $\beta_1,\beta_2,\ldots,\beta_{d+1}$ be $d+1$ different elements in $k$. For each $\beta_u$,
we take a deformation $r_u:X\to \bA^1_k$ corresponding to
\begin{equation}\label{r(t)}
r_u^\sharp:k[t]\to\Gamma(X,\sO_X),\ \ \ t\mapsto \sum_{i=1}^n\lambda_i y_i+ R(y_1,\ldots, y_2)+y_1^{2p}+\beta_u y_2^{2p}.
\end{equation}
We have $r_u(x)=o\in\mathbb A^1_k$ for any $1\leq u\leq d+1$. Notice that each $Z_u=r_u^{-1}(o)$ is smooth at $x$. After replacing $X$ by a Zariski neighborhood of $x$, we may assume that each $Z_u$ is irreducible and smooth. Suppose that $Z_u=Z_{u'}$ for two different $u$ and $u'$. Observe that $r_{u'}^\sharp(t)-r_u^\sharp(t)=(\beta_{u'}-\beta_u)y_2^{2p}$. Hence, if $Z_u=Z_{u'}$ for $0\leq u<u'\leq d+1$, we have $Z_u=Z_{u'}=E_2$, where $E_2$ is a smooth Cartier divisor defined by the ideal $(y_2)\subset\Gamma(X,\sO_X)$. However $r_0^\sharp(t)$ is quadratic and $\deg y_1^{2p}=\deg y_2^{2p}=2p>2$, we see that $\ol{r_u^\sharp(t)}=\ol{r_0^\sharp(t)}+\ol{y_1}^{2p}+\beta_u\ol{y_2}^{2p}\neq 0$ in the integral domain $\Gamma(E_2,\sO_{E_2})=\Gamma(X, \sO_X)/(y_2)$, i.e., $E_2\neq Z_u$. Hence, the assumption $Z_u=Z_{u'}$ for $0\leq u<u'\leq d+1$ is not valid. Thus, $Z_1,\ldots, Z_{d+1}$ are $d+1$ different irreducible and smooth closed subscheme of $X$ containing $x$. Since $D$ has only $d$ irreducible components,  one $Z_u$ is not contained in $D$. We denote it by $Z$ and by $r:X\to\bA^1_k$ the $k$-morphism defining $Z$. Since $d(y_1^{2p}+\beta_uy_2^{2p})=2p(y_1^{2p-1}dy_1+\beta_uy_2^{2p-1}dy_2)=0$, the two sections  $dr(X)$ and $dr_0(X)$ of $\bT^*X$ are the same. Hence, $dr(X)$ meets $C$ transversely at $(x,\omega)\in \bT^*X$. Since $dr_x:\bT^*_o\bA^1_k\to \bT^*_xX$  is injective and maps $dt$ to $\omega$, where $\omega\not\in L_x$, the morphism $r:X\to\bA^1_k$ is smooth at $x$ and is $L_x$-transversal at $x$. After replacing $X$ by a Zariski open neighborhood of $x$, we may assume that $r:X\to \bA^1_k$ is smooth. For each $1\leq i\leq d$, $f|_{D_i}:D_i\to\bA^1_k$ is a dominant morphism, hence is flat.

When $n=2$, the morphism $f=r:X\to\bA^1_k$ and $g=\id:\bA^1_k\to\bA^1_k$ satisfies  a) and b). We  treat the case where $n\geq 3$ in the following. We denote by $\varphi:Z\to X$ the canonical injection.  For each $b=1,2,\ldots, r$, we denote by $D'_b$ the effective Cartier divisor of $X$ associated to the ideal $(\lambda'_{b1}y_1+\cdots+\lambda'_{bn}y_n)\subseteq\Gamma(X,\sO_X)$, which is smooth at $x\in X$. Since, for each $b=1,2,\ldots, r$, the two vectors $(\lambda_1,\ldots,\lambda_n)$ and $(\lambda'_{b1},\ldots,\lambda'_{bn})$ in $k^n$ are linearly independent, we see that $\lambda'_{b1}y_1+\cdots+\lambda'_{bn}y_n\in \fm_{Z,x}-\fm^2_{Z,x}\subseteq\sO_{Z,x}=\sO_{X,x}/(r^\sharp(t))$. Hence each $(D'_b)_Z=D'_b\times_XZ$ is an effective Cartier divisor of $Z$ smooth at $x$. We put $D'_Z=\sum_{b=1}^r (D'_b)_Z$. Since each $D_i$ is not contained in $Z$, $D_Z=D\times_XZ$ is an effective Cartier divisor of $Z$ containing $x$. Let $\wt Z$ be the blow-up of $Z$ at $x$, $E_x$ the exceptional divisor, $\wt{D}_Z$ the strict transform of $D_Z$, and $\wt {D'}_Z$ the strict transform of $D'_Z$. We have $\dim_k E_x=n-2 \geq 1$ and the intersections $E_x\bigcap\wt{D}_Z$ and $E_x\bigcap\wt{D'}_Z$ are of codimension $1$ in $E_x$. We take a closed point $s\in E_x-E_x\bigcap\wt{D}_Z-E_x\bigcap\wt{D'}_Z$.  Let $S$ be a smooth $k$-curve and $\iota:S\to Z$ an immersion  such that $x\in \iota(S)$ and that the image of $\partial\iota:\bT_xS\to\bT_xZ$ corresponds to the closed point $s\in E_x$, by the canonical isomorphism $\bP(\bT_xZ)\cong E_x$.  For each $b=1,2,\ldots, r$, the hyperplane $H_b\subseteq \bT_xZ$ associates to $E_x\bigcap(\wt{D'_b})_Z$ is the dual space of the image $d\varphi_x(\langle \theta_b\rangle)\in \bT^*_xZ $, where $d\varphi_x:\bT^*_xX\to\bT^*_xZ$ is the canonical surjection. Then, $s\not\in E_x\bigcap\wt{D'}_Z$ implies that $\mathrm{im}(\partial \iota))\not\subseteq \bigcup_{b=1}^r H_b$, hence, the composition $d\iota_x\circ d\varphi_x:\bT^*_xX\to\bT^*_xZ\to\bT^*_xS$ sends each $\theta_b$ to a non-zero vector in $\bT^*_xS$. In other words, $\iota\circ\varphi:S\to X$ is $\langle \theta_b\rangle$-transversal at $x$ for each $b=1,2,\ldots, r$. Hence, $\iota\circ\varphi:S\to X$ is $L_x$-transversal at $x$.  The strict transform $\wt S$ of $S$ in $\wt Z$ contains the closed point $s$. Hence $s\not\in E_x\bigcap \wt{D}_Z$ implies that $\wt S\not\subseteq \wt{D}_Z$ and moreover, that $S\not\subseteq D_Z$. After shrinking $X$, we may further assume that $\iota:S\to Z$ is a closed immersion and that $S\bigcap D=x$.  Consider the surjective local homomorphisms of regular local $k$-algebras $\varphi^\sharp:\sO_{X,x}\to\sO_{Z,x}$ and $\iota^{\sharp}:\sO_{Z,x}\to\sO_{S,x}$. Choose a regular system of parameters $u_1, u_2,\ldots,u_n$ of $\sO_{X,x}$ such that $u_1$ is the image of $r^{\sharp}(t)\in\Gamma(X,\sO_X)$ in $\sO_{X,x}$ \eqref{r(t)} and $(u_1,\ldots,u_{n-1})$ is the kernel of surjection $\iota^\sharp\circ\varphi^\sharp:\sO_{X,x}\to\sO_{S,x}$. We define two $k$-homomorphisms $g^\sharp:k[t]\to k[t_1,\ldots, t_{n-1}]$ and $f^\sharp:k[t_1,\ldots, t_{n-1}]\to\sO_{X,x}$ by
\begin{align*}
&g^\sharp:k[t]\to k[t_1,\ldots, t_{n-1}],\ \ \ t\mapsto t_1;\\
&f^\sharp:k[t_1,\ldots, t_{n-1}]\to\sO_{X,x},\ \ \ t_i\mapsto u_i.
\end{align*}
In a Zariski open neighborhood of $x$ in $X$, they induce two smooth morphisms among $k$-schemes $f:X\to \bA^{n-1}_k$ and $g:\bA^{n-1}_k\to \bA^1_k$ satisfying:
\begin{itemize}
\item[i)] Denoting by $O$ the origin of $\bA^{n-1}_k$, the morphism $f:X\to \bA^{n-1}_k$ is smooth at $x$ with $f^{-1}(O)=S$. Since $S\bigcap D=x$, the restriction $f|_{D_i}:D_i\to \bA^{n-1}_k$ is flat at $x$ for each $1\leq i\leq d$ (cf. \cite[5, I, Chapitre 0, 15.1.16]{EGA4}). Since each $D_i$ is of dimension $n-1$,  flat morphisms $f|_{D_i}:D_i\to \bA^{n-1}_k$ are quasi-finite in an open neighborhood of $x$.
\item[ii)]
The morphism $g:\bA^{n-1}_k\to \bA^1_k$ is the projection to the first coordinate, hence is smooth.
\item[iii)] The morphism $f:X\to \bA^{n-1}_k$ is $L_x$-transversal at $x$ since $\iota\circ \varphi:S\to X$ is $L_x$-transversal at $x$.
\item[iv)]
Notice that $g\circ f=r:X\to\bA^1_k$. Hence $g\circ f:X-\{x\}\to\bA^1_k$ is $C$-transversal and the section $d(g\circ f)(X)$ meets $C$ transversely at $(x,\omega)$.
\end{itemize}
Conditions i) - iv) implies that, after replacing $X$ by a Zariski open neighborhood of $x$, $f:X\to \bA^{n-1}_k$ and $g:\bA^{n-1}_k\to \bA^1_k$ satisfy a) and b).

(2) By the proof of \cite[Proposition 5.19]{cc}, we are in one of the following two situations when $p=2$:
\begin{itemize}
\item[i)]
We have a smooth morphism $r_0:X\to \bA^1_k$ defined by a quadratic polynomial in terms of a standard local coordinate of $X$ at $x$ as in \eqref{saito's r} such that the section $dr_0(X)$ meets $C$ transversely at $(x,\omega)\in \bT^*X$;
\item[ii)] We replace $X$, $C$ and $(x,\omega)$ by $\bA^1_X$, $q^{\circ}C$ and $((x,0),q^\circ\omega)$, respectively. We have a smooth morphism $r'_0:\bA^1_X\to \bA^1_k$ defined by a quadratic polynomial in terms of a standard local coordinate of $\bA^1_X$ at $(x,0)$ as in \eqref{saito's r} such that the section $dr'_0(\bA^1_X)$ meets $q^{\circ}C$ transversely at $((x,0),q^\circ\omega)\in \bT^*\bA^1_{X}$.
\end{itemize}
In any case, we adapt the same argument as in (1) to obtain (2).
\end{proof}

In the following of this section, let $\Lambda$ denote a finite field of characteristic $\ell$ $(\ell\neq p)$ containing a primitive $p$-th root of unity.

\begin{definition}\label{defCpure}
Let $C$ be a closed conical subset of $\bT^*X$ such that $B(C)=D$ and that $\dim_{k(\ol x)}C_{\ol x}=1$ for any geometric point $\ol x\to D$.
Let $\sF$ be a locally constant and constructible sheaf of $\Lambda$-modules on $U$.
We say the ramification of $\sF$ along $D$ is $C$-{\it pure by restricting to curves} if $\sF$ satisfies the following condition:

$(\star)$ For any element $(S, h:S\to X, x)$ in $\mathcal Q(X,D)$ such that $h:S\to X$ is $C$-transversal at $s=h^{-1}(x)$, the  equality
\begin{equation*}
\dt_s(\sF|_{S_0})=m_s(h^*\DT_X(j_!\sF))=\sum_{i=1}^m m_s(h^*D_i)\cdot \dt_{D_i}(\sF)
\end{equation*}
holds.
\end{definition}

\begin{lemma}\label{lemma_curve_pure}
Let $C$ be a closed conical subset of $\bT^*X$ such that $B(C)=D$ and that $\dim_{k(\ol x)}C_{\ol x}=1$ for any geometric point $\ol x\to D$. Then,

(1) Let $g:Z\to X$ be a smooth morphism between connected and smooth $k$-schemes. If an element $(S,h:S\to Z,z)$ of $\cQ(Z, g^*D)$ is $g^\circ C$-transversal at $s=h^{-1}(z)$, then $(S, g\circ h:S\to X, g(z))$ is an element of $\cQ(X,D)$ which is $C$-transversal at $s$.

(2) Let $\sF$ be a locally constant and constructible sheaf of $\Lambda$-modules on $U$.
If the ramification of $\sF$ along $D$ is $C$-pure by restricting to curves, then, for any smooth morphism $g:Z\to X$, that of $g_U^*\sF$ along $g^*D$ is $g^{\circ}C$-pure by restricting to curves (where $g_U:g^{-1}(U)\to U$). Conversely, if, for a smooth and surjective morphism $g:Z\to X$, the ramification of $g_U^*\sF$ along $g^*D$ is $g^{\circ}C$-pure by restricting to curves (where $g_U:g^{-1}(U)\to U$), then that of $\sF$ along $D$ is $C$-pure by restricting to curves.
\end{lemma}

\begin{proof}
 (1) We observe that $x=g(z)=(g\circ h)(S)\bigcap D$ is a single point of $X$ and $x\neq (g\circ h)(S)$. The composition $g\circ h:S\to X$ is quasi-finite, since the morphism $\ol{\pr}_2\circ h:S\to X$ is of finite type, $S$ is an irreducible smooth $k$-curve and $(\ol{\pr}_2\circ h)(S)\subseteq X$ is irreducible but not a single point. Observe further that $s=h^{-1}(z)=(g\circ h)^{-1}(x)$ is a closed point of $S$. Thus, $(S,g\circ h:S\to X,x)$ is an element of $\cQ(X,D)$. If $h:S\to Z$ is $g^{\circ}C$-transversal at $s$, then $g\circ h:S\to X$ is $C$-transversal at $s$, since $g:Z\to X$ is smooth.

(2) For a smooth morphism $g:Z\to X$ and an element $(S,h:S\to Z,z)$ of $\cQ(Z, g^*D)$ which is $g^\circ C$-transversal at $s=h^{-1}(z)$, we showed that $(S,g\circ h:S\to X,x)$ is an element of $\cQ(X,D)$ and that $g\circ h:S\to X$ is $C$-transversal at $s$, where $x=g(z)$. By the assumption that the ramification of $\sF$ along $D$ is $C$-pure by restricting to curves, we have
\begin{align}\label{dt=prh}
\dt_s((g_U^*\sF)|_{S_0})=\dt_s(\sF|_{S_0})=m_s((g\circ h)^*(\DT_X(j_!\sF))).
\end{align}
By \cite[Proposition 3.8]{wr}, we have
\begin{equation}\label{prdt=dt}
g^*(\DT_X(j_!\sF))=\DT_{Z}(j_{U!}g_U^*\sF),
\end{equation}
where $j_U:g^{-1}(U)\to Z$ is the canonical injection.
Combining \eqref{dt=prh} and \eqref{prdt=dt}, we obtain
\begin{align*}
\dt_s((g_U^*\sF)|_{S_0})&=m_s(h^*g^*(\DT_X(j_!\sF)))\\
&=m_s(h^*(\DT_Z(j_{U!}g_U^*\sF))).
\end{align*}
By definition, the ramification of $g_U^*\sF$ along $g^*D$ is $g^{\circ}C$-pure by restricting to curves.

Conversely, let $g:Z\to X$ be a surjective and smooth morphism and we assume that the ramification of $g_U^*\sF$ along $g^*D$ is $g^\circ D$-pure by restricting to curves. We are sufficient to show that, for any element $(S,h:S\to X,x)$ of $\cQ(X,D)$ such that $h:S\to X$ is $C$-transversal at $s=h^{-1}(x)$, we have
\begin{equation*}
\dt_{s}(\sF|_{S_0})=m_{s}(h^*(\DT_X(j_!\sF))).
\end{equation*}
We fix such an element $(S,h:S\to X,x)$.  After shrinking $Z$ by an open neighborhood of a closed point $z$ of $g^{-1}(x)$, we have an \'etale morphism $g':Z\to \bA^n_X$. Let $v$ be the image of $z$ in $\bA^n_k$ by the composition $g':Z\to \bA^n_X$ and the canonical projection $\pr:\bA^n_X\to\bA^n_k$. The quasi-finite morphism $h:S\to X$ lifts to $h':S\otimes_kk(v)\to \bA^n_X$ by the universal property of fiber products. Let $T$ be a connected component $(S\otimes_kk(v))\times_{\bA^n_X}Z$ which contains a closed point $t$ mapping to $z$ in $Z$ and mapping to $s$ in $S$. We denote by $h_T:T\to Z$ the canonical quasi-finite morphism. We observe that the lifting $h_T:T\to Z$ is $g^{\circ}C$-transversal at $t$ and that $(T,h_T:T\to Y,t)$ is an element of $\cQ(Y, g^*D)$. By the assumption that the ramification of $g_U^*\sF$ along $g^*D$ is $g^{\circ}C$-pure by restricting to curves and \eqref{prdt=dt}, we have
\begin{align*}
\dt_{s}(\sF|_{S_0})&=\dt_{t}(\sF|_{{T}-\{t\}})=\dt_{t}((g_U^*\sF)|_{{T}-\{t\}})\\
&=m_{t}(h_T^*(\DT_{Z}(g^*j_!\sF)))\\
&=m_{t}(h_T^*g^*(\DT_X(j_!\sF)))\\
&=m_{s}(h^*(\DT_X(j_!\sF))).
\end{align*}
It finishes the proof the converse part.
\end{proof}

\begin{proposition}\label{equivCPandP}
Let $C$ be a closed conical subset of $\bT^*X$ such that $B(C)=D$ and that $\dim_{k(\ol x)}C_{\ol x}=1$ for any geometric point $\ol x\to D$. Let $\sF$ be a locally constant and constructible sheaf of $\Lambda$-modules on $U$. Then, the following two conditions are equivalent:
\begin{itemize}
\item[(1)] The ramification of $\sF$ along $D$ is $C$-pure by restricting to curves.
\item[(2)] We have $SS(j_!\sF)\subseteq \bT^*_XX\bigcup C$. In particular, the ramification of $\sF$ along $D$ is pure (Definition \ref{definitionpure}).
\end{itemize}
\end{proposition}

\begin{proof}
We may assume that $k$ is algebraically closed.

(1)$\Rightarrow$(2). Suppose that there exists an irreducible component $C'$ of $SS(j_!\sF)$ which is not contained in $\bT^*_XX\bigcup C$. Notice that $B(C')\subseteq D$. Let $(x,\omega)\in C'(k)$ be a closed smooth point of $C'$ which is neither contained in other irreducible components of $SS(j_!\sF)$ nor contained in $C$. Let $b'$ be the coefficient of $[C']$ in $CC(j_!\sF)$. Since $j_!\sF[n]$ is a perverse sheaf on $X$, $b'$ is a non-zero integer.

We firstly consider the case where $p\geq 3$. After replacing $X$ by a Zariski open neighborhood of $x$, we have two morphisms $f:X\to \bA^{n-1}_k$ and $g:\bA^{n-1}_k\to \bA^1_k$ satisfying (Proposition \ref{tranfibers}):
\begin{itemize}
\item[a)]  they are smooth and $f|_{D_i}:D_i\to \bA^{n-1}_k$ is quasi-finite and flat, for each $D_i$ containing $x$;
\item[b)] the morphism $f:X\to \bA^{n-1}_k$ is $C$-transversal, the composition $g\circ f:X-\{x\}\to\bA^1_k$ is $C'$-transversal and the section $d(g\circ f)(X)$ meets $C'$ transversely at $(x,\omega)$ (cf. \ref{sectiondf}).
\end{itemize}
 Since $x$ is an isolated point in $f^{-1}(f(x))\bigcap D$, there exists an affine and connected \'etale neighborhood $A'$ of $f(x)\in\bA^{n-1}_k$, a closed point $x'\in X'=X\times_{\bA^n_k}A'$ above $x$ and an open and closed subscheme $D'$ of $D\times_{\bA^n_k}A'$ such that $D'$ is finite and flat over $A'$ and that $f'^{-1}({f'(x')})\bigcap D'=\{x'\}$, where $f':X'\to A'$ is the base change of $f:X\to\bA^{n-1}_k$ (\cite[IV, 18.12.1]{EGA4}).
We have the following commutative diagram
\begin{equation}\label{squareDL1}
\xymatrix{\relax
D'\ar[r]^{\iota'}\ar[d]_{\psi_D}&X'\ar[d]_{\psi_X}\ar[r]^{f'}\ar@{}|-{\Box}[rd]&A'\ar[d]^{\psi}\\
D\ar[r]_{\iota}&X\ar[r]_-(0.5)f&\bA^{n-1}_k}
\end{equation}
Let $I$ be the subset of $\{1,\ldots,d\}$ consisting of all $i$ for $x\in D_i$. We put $D'_i=D'\bigcap(D_i\times_{\bA^{n-1}_k}A')$ and we see that  $D'_i$ is finite and flat over $A'$ for each $i\in I$.  Since $f:X\to\bA^{n-1}_k$ is $C$-transversal and the right square of \eqref{squareDL1} is Cartesian, the composition of $\iota_{a'}:X'_{a'}\to X'$ and $\psi_X:X'\to X$ is a closed immersion and is $C$-transversal for any closed point $a'\in A'$, where $X'_{a'}=f'^{-1}(a')=X\times_{\bA^{n-1}}a'$. We put $U'=X'\times_XU$, put $U'_{a'}=U'\times_{A'}a'$, put $D'_{a'}=D'\times_{A'}a'$ and put $D_{i,a'}=D'_i\times_{A'}a'$. Since the ramification of $\sF$ along $D$ is $C$-pure by restricting to curves, we have
 \begin{align*}
 \sum_{t\in D'_{a'}}\dt_t(\sF|_{U'_{a'}})&=\sum_{i\in I}\sum_{t\in D'_{i,a'}}m_t((\psi_X\circ\iota_{a'})^*D_i)\cdot\dt_{D_i}(\sF)\\
 &=\sum_{i\in I}\sum_{t\in D'_{i,a'}}m_t(\iota_{a'}^*D'_i)\cdot\dt_{D_i}(\sF)\\
 &=\sum_{i\in I}m_{x'}(\iota_{f'(x')}^*D'_i)\cdot\dt_{D_i}(\sF)\\
 &=\dt_{x'}(\sF|_{U'_{f'(x')}})
 \end{align*}
 for any $a'\in A'(k)$. Hence, the following function
 \begin{equation*}
 \varphi: A'(k)\to \bZ,\ \ \ a'\mapsto \sum_{t\in D'_{a'}}\dt_t(\sF|_{U'_{a'}})
 \end{equation*}
is constant. By \cite[Th\'eor\`eme 2.1.1]{lau}, the morphism $f':X'\to A'$ is universally locally acyclic with respect to $\psi^*_Xj_!\sF$ in a Zariski open neighborhood of $D'$ in $X'$. Therefore, $f:X\to\bA^{n-1}_k$ is universally locally acyclic with respect to $j_!\sF$ in a Zariski open neighborhood of $x$ in $X$. Since $g:\bA^{n-1}_k\to \bA^1_k$ is smooth, it is universally locally acyclic with respect to the constant sheaf $\Lambda$ on $\bA^{n-1}_k$. We conclude that the composition $g\circ f:X\to \bA^1_k$ is universally locally acyclic with respect to $j_!\sF$ in a Zariski open neighborhood of $x\in X$ (\cite[Lemma 7.7.6]{fu}). Hence $\rR\Phi_x(j_!\sF,g\circ f)=0$. Notice that $d(g\circ f)(X)$ meets $C'$ transversely at $(x,\omega)\in \bT^*X$ and $(x,\omega)$ is not contained in other irreducible components of $SS(j_!\sF)$, we have ({\cite[Theorem 7.6]{cc}})
\begin{align}
b'=(b'[C'],[d(g\circ f)(X)])_{\bT^*X,x}&=(CC(j_!\sF),[d(g\circ f)(X)])_{\bT^*X,x}\\
&=-\sum_i(-1)^i\dt(\rR^i\Phi_{x}(j_!\sF,g\circ f))=0\nonumber
\end{align}
Hence, when $p\geq 3$, we obtain $b'=0$. It contradicts to our assumption that $C'$ is an irreducible component of $SS(j_!\sF)$. Thus, (1)$\Rightarrow$(2) is valid when $p\geq 3$.

When $p= 2$, we are left to consider the exceptional case by Proposition \ref{tranfibers}. We suppose that there exists an irreducible component $C'$ of $SS(j_!\sF)$ which is not contained in $\bT^*_XX\bigcup C$. Notice that $B(C')\subseteq D$. Let $(x,\omega)\in C'(k)$ be a closed smooth point of $C'$ which is neither contained in other irreducible components of $SS(j_!\sF)$ nor contained in $C$. Let $b'$ be the coefficient of $[C']$ in $CC(j_!\sF)$. Since $j_!\sF[n]$ is a perverse sheaf on $X$, $b'$ is a non-zero integer.
 Let $q:\bA^1_X\to X$ be the canonical projection. By \cite[Theorem 7.6]{cc}, the multiplicity of the cycle $[q^\circ C']$ in $CC(q^*j_!\sF)$ is $-b'$ and $((x,0),q^\circ\omega)$ is a smooth locus of $q^\circ C'$ and is neither contained in $q^\circ C$ nor contained in other irreducible components of $SS(q^*j_!\sF)$. After replacing $\bA^1_X$ by a Zariski open neighborhood of $(x,0)$, we have two morphisms $f:\bA^1_X\to \bA^{n}_k$ and $g:\bA^{n}_k\to \bA^1_k$ satisfying (Proposition \ref{tranfibers}):
\begin{itemize}
\item[a)]  they are smooth and $f|_{\bA^1_{D_i}}:\bA^1_{D_i}\to \bA^{n}_k$ is quasi-finite and flat, for each $D_i$ containing $x$;
\item[b)] the morphism $f:\bA^1_X\to \bA^n_k$ is $q^\circ C$-transversal, the composition $g\circ f:\bA^1_X-\{(x,0)\}\to\bA^1_k$ is $q^{\circ} {C'}$-transversal and the section $d(g\circ f)(\bA^1_X)$ meets $q^\circ C'$ transversely at $((x,0),q^\circ\omega)$ where we consider $q^\circ\omega$ as a vector of $\bT^*_{(x,0)}\bA^1_X$ by the canonical injection $dq:\bT^*X\times_X\bA^1_X\to\bT^*\bA^1_X$.
\end{itemize}
By Lemma \ref{lemma_curve_pure}, we see that the ramification of $\sF|_{\bA^1_U}$ along $\bA^1_D$ is $q^\circ C$-pure by restricting to curves. Applying the same strategy as in the $p\geq 3$ case to the sheaf $\sF|_{\bA^1_U}$ ramified along $\bA^1_D$, we can prove $b'=0$. It contradicts to our assumption that $C'$ is an irreducible component of $SS(j_!\sF)$. Hence, we also prove that (1)$\Rightarrow$(2) is valid when $p=2$.

(2)$\Rightarrow$(1). Let $(S, h:S\to X, x)$ be an element of $\cQ(X,D)$ such that $h:S\to X$ is $C$-transversal at $s=h^{-1}(x)$. The morphism $h:S\to X$ is the composition of the graph $\Gamma_h:S\to X\times_kS$ and the projection $\pr_1:X\times_kS\to X$. Notice that $\Gamma_h:S\to X\times_kS$ is $\pr_1^\circ(\bT^*_XX\bigcup C)$-transversal. By \cite[Theorem 7.6]{cc}, we have $SS(\pr_1^*j_!\sF)\subseteq \pr_1^\circ(\bT^*_XX\bigcup C)$. Since $\pr_1$ is smooth, we have $\pr_1^*(\DT_X(j_!\sF))=\DT_{X\times_kS}(\pr_1^*j_!\sF)$. Hence, to show (1), it is sufficient to show
\begin{equation*}
  \dt_s(\sF|_{S_0})=m_s(\Gamma_h^*(\DT_{X\times_kS}(\pr_1^*j_!\sF))).
\end{equation*}
Thus, we are reduced to verify the condition $(\star)$ in Definition \ref{defCpure} for all elements in $\cI(X,D)$.

Let $(S, h:S\to X, x)$ be an element of $\cI(X,D)$ such that $h:S\to X$ is $C$-transversal at $x$.
After replacing $X$ by an open neighborhood of $x$, we may assume that there is a smooth relative curve $f:X\to\bA^{n-1}_k$ with $f^{-1}(f(x))=S$. Since $h:S\to X$ is $C$-transversal at $x$, the morphism $f:X\to\bA^{n-1}_k$ is $C$-transversal at $x$. After replacing $X$ by a Zariski neighborhood of $x$, we may assume that $f:X\to\bA^{n-1}_k$ is $C$-transversal.
 Since $SS(j_!\sF)\subseteq \bT^*_XX\bigcup C$, the smooth morphism $f:X\to\bA^{n-1}_k$ is $SS(j_!\sF)$-transversal.
Hence $f:X\to\bA^{n-1}_k$ is universally locally acyclic with respect to $j_!\sF$.  Since $x$ is an isolated point in $f^{-1}(f(x))\bigcap D$, there exists an affine and connected \'etale neighborhood $A'$ of $f(x)\in\bA^{n-1}_k$, a closed point $x'\in X'=X\times_{\bA^n_k}A'$ above $x$ and an open and closed subscheme $D'$ of $D\times_{\bA^n_k}A'$ such that $D'$ is finite and flat over $A'$ and that $f'^{-1}({f'(x')})\bigcap D'=\{x'\}$, where $f':X'\to A'$ is the base change of $f:X\to\bA^{n-1}_k$ (\cite[IV, 18.12.1]{EGA4}).
We have the following commutative diagram
\begin{equation}\label{squareDL}
\xymatrix{\relax
D'\ar[r]^{\iota'}\ar[d]_{\psi_D}&X'\ar[d]_{\psi_X}\ar[r]^{f'}\ar@{}|-{\Box}[rd]&A'\ar[d]^{\psi}\\
D\ar[r]_{\iota}&X\ar[r]_-(0.5)f&\bA^{n-1}_k}
\end{equation}
Let $I$ be the subset of $\{1,\ldots,d\}$ consisting of all $i$ for $x\in D_i$. We put $D'_i=D'\bigcap(D_i\times_{\bA^{n-1}_k}A')$ and we see that  $D'_i$ is finite and flat over $A'$ for each $i\in I$.  The morphism $f':X'\to A'$ is universally locally acyclic with respect to $\psi^*_Xj_!\sF$ in a Zariski open neighborhood of $D'$ in $X'$. By \cite[Th\'eor\`eme 2.1.1]{lau}, the function
 \begin{equation*}
 \varphi: A'(k)\to \bZ,\ \ \ a'\mapsto \sum_{t\in D'_{a'}}\dt_t(\sF|_{U'_{a'}})
 \end{equation*}
is constant, where $U'=U\times_{X} X'$, $U'_{a'}=U'\times_{A'}a'$ and $D'_{a'}=D'\times_{A'}a'$. Let $Z'$ be a closed subscheme of $D'$ of codimension $1$ such that $D'-Z'$ is smooth and that the ramification of $\sF|_{U'}$ along $D'-Z'$ is non-degenerate. Since $D'$ is finite and flat over $A'$, the image of $Z'$ in $A'$ is of codimension $1$ in $A'$. Let $a$ be a closed point in $A'-f(Z')$. The composition of  $\iota'_a:X'_a=f'^{-1}(a)\to X'$ and $\psi_X:X'\to X$ is a closed immersion and is $C$-transversal (thus $SS(j_!\sF)$-transversal). By \cite[Corollary 3.9]{wr}, we have
\begin{align*}
\dt_x(\sF|_{S_0})&=\sum_{t\in D'_{a}}\dt_t(\sF|_{U'_{a}})=\sum_{i\in I}\sum_{t\in D'_{i,a}}\dt_t(\sF|_{U'_{a}})\\
&=\sum_{i\in I}\sum_{t\in D'_{i,a}}m_t((\psi_X\circ\iota'_a)^* D_i)\cdot\dt_{D_i}(\sF)\\
&=\sum_{i\in I}\sum_{t\in D'_{i,a}}m_t(\iota'^*_a(D'_i))\cdot\dt_{D_i}(\sF)\\
&=\sum_{i\in I}m_x(h^*D_i)\cdot\dt_{D_i}(\sF)=m_x(h^*(\DT_X(j_!\sF))),
\end{align*}
Hence, the ramification of $\sF$ along $D$ is $C$-pure by restricting to curves.
\end{proof}

\begin{corollary}\label{criteriaCP}
Let $C$ be a closed conical subset of $\bT^*X$ such that $B(C)=D$, that $\dim_{k(\ol x)}C_{\ol x}=1$ for any geometric point $\ol x\to D$ and that $\{C\times_XD_i\}_{1\leq i\leq d}$ is the set of irreducible components of $C$. Let $\sF$ be a non-zero locally constant and constructible sheaf of $\Lambda$-modules on $U$.  Then, the ramification of $\sF$ along $D$ is $C$-pure by restricting to curves if and only if
\begin{align}
SS(j_!\sF)&=\bT^*_XX\bigcup C,\label{calSS}
\end{align}
\end{corollary}

 The if part is from Proposition \ref{equivCPandP} (2)$\Rightarrow$(1). Since $\sF$ is ramified along $D$, the singular support $SS(j_!\sF))$ contains irreducible components supported on each $D_i$.
We obtain the only if part from Proposition \ref{equivCPandP} (1)$\Rightarrow$(2).

\begin{definition}\label{defCisoclinic}
Let $C$ be a closed conical subset of $\bT^*X$ such that $B(C)=D$ and that $\dim_{k(\ol x)}C_{\ol x}=1$ for any geometric point $\ol x\to D$.
Let $\sG$ be a locally constant and constructible sheaf of $\Lambda$-modules on $U$.
We say the ramification of $\sG$ along $D$ is $C$-{\it isoclinic by restricting to curves} if $\sG$ satisfies the following condition:

$(\ast)$ For any element $(S, h:S\to X, x)$ in $\mathcal Q(X,D)$ such that $h:S\to X$ is $C$-transversal at $s=h^{-1}(x)$,
the ramification of $\sG|_{S_0}$ at $s$ is isoclinic and the equality
\begin{equation*}
\rc_s(\sG|_{S_0})=m_s(h^*\rC_X(j_!\sG))=\sum_{i=1}^m m_s(h^*D_i)\cdot \rC_{D_i}(\sG)
\end{equation*}
holds.
\end{definition}

\begin{lemma}\label{lemma_curve_isoclinic}
Let $C$ be a closed conical subset of $\bT^*X$ such that $B(C)=D$ and that $\dim_{k(\ol x)}C_{\ol x}=1$ for any geometric point $\ol x\to D$. Let $\sF$ be a locally constant and constructible sheaf of $\Lambda$-modules on $U$.
If the ramification of $\sF$ along $D$ is $C$-isoclinic by restricting to curves, then, for any smooth morphism $g:Z\to X$, that of $g_U^*\sF$ along $g^*D$ is $g^{\circ}C$-isoclinic by restricting to curves (where $g_U:g^{-1}(U)\to U$). Conversely, if, for a smooth and surjective morphism $g:Z\to X$, the ramification of $g_U^*\sF$ along $g^*D$ is $g^{\circ}C$-isoclinic by restricting to curves (where $g_U:g^{-1}(U)\to U$), then that of $\sF$ along $D$ is $C$-isoclinic by restricting to curves.
\end{lemma}

The proof of the lemma above is similar to that of Lemma \ref{lemma_curve_pure} (2).

\begin{lemma}\label{lemmaisocliniccurve}
Let $C$ be a closed conical subset of $\bT^*X$ such that $B(C)=D$ and that $\dim_{k(\ol x)}C_{\ol x}=1$ for any geometric point $\ol x\to D$. Let $\sG$ be a locally constant and constructible sheaf of $\Lambda$-modules on $U$ whose ramification along $D$ is $C$-isoclinic by restricting to curves. Then
\begin{itemize}
\item[(1)]
The ramification of $\sG$ at each generic point of $D$ is isoclinic.
\item[(2)]
The ramification $\sG$ along $D$ is $C$-pure by restricting to curves.
\item[(3)]
Suppose that $\{C\times_XD_i\}_{1\leq i\leq d}$ is the set of irreducible components of $C$. We have
\begin{align}
SS(j_!\sG)&=\bT^*_XX\bigcup C\end{align}
\end{itemize}
\end{lemma}
\begin{proof}
(1). We may assume that $k$ is algebraically closed. We fix $1\leq i\leq d$ and we choose a closed point $u$ of $D_i$ such that $u$ is not contained each $D_{i'}\ \ (i'\neq i)$, that $u$ is a smooth locus of $D_i$ and that the ramification of $\sG$ along $D_i$ is non-degenerate at $u$. In this case, we have $\dim_k(SS(j_!\sG)\times_Xu)=1$. Then, we can find an element $(S,h:S\to X,u)$ of $\cI(X,D)$ such that $h:S\to X$ is $(SS(j_!\sG)\bigcup C)$-transversal at $s=h^{-1}(u)$ and that $m_s(h^*D_i)=1$ (Remark \ref{manytranscurve}). On one hand, by applying \cite[Proposition 2.22 and 3.8]{wr} (cf. part (2) of Theorem \ref{DTtheorem}) to $h:S\to X$ and $\sG$, we have $\rNP_{D_i}(\sG)=\rNP_{s}(\sG|_{S_0})$. On the other hand, since the ramification of $\sG$ along $D$ is $C$-isoclinic by restricting to curves, the ramification of $\sG|_{S_0}$ at $s$ is isoclinic. Hence, the ramification of $\sG$ at the generic point of $D_i$ is isoclinic with $\rc_{D_i}(\sG)=\rc_{s}(\sG|_{S_0})$. We obtain (1).  

(2). By (1) and the assumption on the ramification of $\sG$, we obtain that, for any element $(S, h:S\to X, x)$ in $\mathcal Q(X,D)$ such that $h:S\to X$ is $C$-transversal at $s=h^{-1}(x)$,
\begin{equation}
\dt_s(\sG|_{S_0})=\rk_{\Lambda}\sG\cdot \rc_s(\sG|_{S_0})=\rk_{\Lambda}\sG\cdot m_s(h^*\rC_X(j_!\sG))=m_s(h^*\DT_X(j_!\sG)).
\end{equation}
We get (2). 

Part (3) is from (2) and Corollary \ref{criteriaCP}.
\end{proof}

\begin{lemma}\label{section5lemmaRj*Rj}
Let $C$ be a closed conical subset of $\bT^*X$ such that $B(C)=D$ and that $\dim_{k(\ol x)}C_{\ol x}=1$ for any geometric point $\ol x\to D$. Let $\sG$ be a locally constant and constructible sheaf of $\Lambda$-modules on $U$. If the ramification of $\sG$ along $D$ is $C$-isoclinic by restricting to curves and $\rc_{D_i}(\sG)>1$ for each $1\leq i\leq d$, then we have $Rj_*\sG=j_!\sG$.
\end{lemma}
\begin{proof}
We may assume that $k$ is algebraically closed. By Proposition \ref{equivCPandP} and Lemma \ref{lemmaisocliniccurve}, we see that the ramification of $\sG$ along $D$ is pure with $SS(j_!\sG)\subseteq \bT^*_XX\bigcup C$. Let $x$ be an arbitrary closed point of $D$. By Remark \ref{manytranscurve}, we can find an element $(S,h:S\to X,x)$ of $\cI(X,D)$ such that $h:S\to X$ is $C$-transversal at $x$. Since the ramification of $\sG$ along $D$ is $C$-isoclinic by restricting to curves and $\rc_{D_i}(\sG)>1$ for each $1\leq i\leq d$, the ramification of $\sG|_{S_0}$ is isoclinic with
\begin{equation}
\rc_x(\sG|_{S_0})=m_x(h^*(\rC_X(\sG)))=\sum_{i=1}^d m_x(h^*D_i)\cdot \rc_{D_i}(\sG)> 1.
\end{equation}
Hence, the ramification of $\sG|_{S_0}$ at $x$ is totally wild.
By Corollary \ref{stalkpullbackcurves}, we see that $(Rj_*\sG)_{\ol x}=0$, where $\ol x\to X$ denotes a geometric point above $x\in D$. Hence $Rj_*\sG=j_!\sG$. 
\end{proof}

%\begin{lemma}
%Let $C$ be a linear closed conical subset of $\bT^*X$ such that $B(C)=D$. Let $\sG$ be a locally constant and constructible sheaf of $\Lambda$-modules on $U$. Then, the following two conditions are equivalent: 
%\begin{itemize}
%\item[(1)]
%The ramification of $\sG$ along $D$ is $C$-isoclinic by restricting to curves;
%\item[(2)]
%For any $C$-transversal morphism $g:Z\to X$ of smooth $k$-schemes such that $g^*D=D\times_XZ$ is a Cartier divisor of $Z$, the ramification of $g_0^*\sG$ along $T=(g^*D)_{\mathrm{red}}$ is $g^{\circ}C$-isoclinic by restricting to curves.
%\end{itemize}
%\end{lemma}
%\begin{proof}
%(2)$\Rightarrow$(1) is obvious by taking $g=\id_X:X \to X$.

%(1)$\Rightarrow$(2). We may assume that $k$ is algebraically closed. 
%\end{proof}

\subsection{}\label{r and L}
Let $E_i$ $(1\leq i\leq m)$ be the divisor of $Y=\spec(k[y_1,\cdots, y_n])$ corresponding to the ideal $(y_i)$ and $E=\sum_{i=1}^m E_i$. We assume that $m\leq n-1$. Let $V$ be the complement of $E$ in $Y$ and $\jmath:V\to Y$ the canonical injection.
 Let $r$ be a positive integer co-prime to $p$ and $[r]:V\to V$ a finite \'etale covering tamely ramified along $E$ given by
 \begin{align*}
  [r]^\sharp:k[y_1^{\pm 1},\ldots,y_m^{\pm 1},y_{m+1},\ldots,y_n]&\to k[y_1^{\pm 1},\ldots,y_m^{\pm 1},y_{m+1},\ldots,y_n],\\
  y_i&\mapsto y_i^r\ \ \ (1\leq i\leq m),\\
  y_i&\mapsto y_i\ \ \ (m+1\leq i\leq n).
\end{align*}
 Let $(p_1,\ldots,p_m)$ be an $m$-tuples of positive integers and $V'$ an Artin-Schreier covering of $V$ defined by
 \begin{equation*}
  V'=\spec\left(\Gamma(V, \sO_V)[T]/(T^p-T-y_n/\left(y_1^{p_1}\cdots y_m^{p_m})\right)\right).
\end{equation*}
We have $\gal(V'/V)\cong \bZ/p\bZ$. We fix a non-trivial character $\chi:\gal(V'/V)\to\Lambda^{\times}$ and we denote by $\sL$ the locally constant and constructible sheaf of $\Lambda$-modules on $V$ of rank $1$ associated to $\chi:\gal(V'/V)\to\Lambda^{\times}$.

\begin{theorem}\label{all of rL}
We take the notation and assumptions of subsection \ref{r and L}. We assume that $p_1,\ldots, p_m$ and $r$ satisfy
\begin{itemize}
\item{}
$p_i>p$, for each $1\leq i\leq m$;
\item{}
 $\{p_i=p^{a}b_i\}_{1\leq i\leq m}$ for positive integers $a, b_1,\ldots, b_m$ where $(p,b_i)=1$;
 \item{}
 $r<\min_{1\leq i\leq m}\{p_i-b_i\}$ and $(r,p)=1$.
\end{itemize}
Then, the ramification of $[r]_*\sL$ along $E$ is $(E\cdot\langle dy_n\rangle)$-isoclinic by restricting to curves with
\begin{align}\label{CYjrsL}
\rC_Y(\jmath_![r]_*\sL)=\sum_{i=1}^m \frac{p_i}{r}\cdot E_i.
\end{align}
In particular, the ramification of $[r]_*\sL$ at each generic point of $E$ is isoclinic and we have
\begin{align}
SS(\jmath_![r]_*\sL)&=\bT^*_YY\bigcup E\cdot\langle dy_n\rangle,\label{SSjrsL}\\
CC(\jmath_![r]_*\sL)&=(-1)^{n}r^m\left([\bT^*_YY]+ \sum_{i=1}^m \frac{p_i}{r}[E_i\cdot\langle dy_n\rangle]\right).\label{CCjrsL}
\end{align}
Moreover, for each $1\leq i\leq m$, the ramification of $[r]_*\sL$ at the generic point of $E_i$ is logarithmic isoclinic of logarithmic conductor $p_i/r$.
\end{theorem}

\begin{proof}
We firstly show that, for any element $(S,h:S\to Y,y)$ of $\cQ(Y,E)$ such that $h:S\to Y$ is $\langle dy_n\rangle$-transversal at $s=h^{-1}(y)$, the ramification of $([r]_*\sL)|_{S_0}$ at $s$ is isoclinic with (cf. subsection \ref{slopedecom})
\begin{equation}\label{rLcurve}
\rc_s(([r]_*\sL)|_{S_0})=\sum_{i=1}^m m_s(h^*E_i)\cdot \frac{p_i}{r}
\end{equation}
We may assume that $k$ is algebraically closed. By Lemma \ref{lemma_curve_isoclinic} and \cite[Proposition 2.22 and 3.8]{wr}, it is sufficient to show the special case where $m=n-1$. Let $(\lambda_1,\ldots,\lambda_{n-1},\mu)\in k^{n}$ be the coordinates of $y$ in $Y$. Since $y$ is a point in $E$, there are zeros in $\lambda_1,\ldots,\lambda_{n-1}$. Without loss of generality, we may assume that $\lambda_1=\cdots=\lambda_e=0$ and $\lambda_i\neq 0$ for $e+1\leq i\leq n-1$. We choose a uniformizer $t$ of the henselization $\sO^{\rh}_{S,s}$ such that $h:S\to Y$ gives a homomorphism of $k$-algebras
\begin{align*}
h^{\sharp}:k[y_1,\ldots,y_n]&\to \sO^{\rh}_{S,s},\\
y_i&\mapsto u_i t^{\alpha_i}\ \ \ (1\leq i\leq e),\\
y_i&\mapsto \lambda_i+u_i t^{\alpha_i}\ \ \ (e+1\leq i\leq n-1),\\
y_n&\mapsto \mu+t.
\end{align*}
where each $\alpha_i=m_s(h^*E_i)$ is a positive integer and each $u_i$ is a unit of $\sO^{\rh}_{S,s}$. Notice that $\lambda_i+u_i t^{\alpha_i}$ $ (e+1\leq i\leq n-1)$ are units of $\sO^{\rh}_{S,s}$. Since $\sO^{\rh}_{S,s}$ is strictly henselian and $(r,p)=1$,  all roots of the following equations
\begin{align}
T^r-u_i&=0\ \ \ (1\leq i\leq e),\label{allroots1}\\
 T^r-(\lambda_i+u_i t^{\alpha_i})&=0 \ \ \ (e+1\leq i\leq n-1)\label{allroots2}
\end{align}
 are contained in $\sO^{\rh}_{S,s}$. For each $1\leq i\leq e$, we choose a root $v_i$ of $T^r-u_i=0$ in $\sO^{\rh}_{S,s}$. For each $e+1\leq i\leq n-1$, we choose a root $\tau_i+w_i t^{\alpha_i}$ of $T^r-(\lambda_i+u_i t^{\alpha_i})=0$ in $\sO^{\rh}_{S,s}$, where $\tau_i\in k^{\times}$ and $w_i$ is a unit in $\sO^{\rh}_{S,s}$. To show \eqref{rLcurve}, we may replace $S$ by an affine \'etale neighborhood of $s\in S$ such that the image of $\Gamma(S, \sO_S)$ in $\sO^{\rh}_{S,s}$ contains roots of \eqref{allroots1} and \eqref{allroots2}.  Let $S'=\spec(\Gamma(S, \sO_S)[t']/(t'^r-t))$ be a finite cover of $S$ tamely ramified at $s$ and we denote by $\gamma:S'\to S$ the canonical projection. Notice that $S'$ is an affine connected $k$-curve smooth at the closed point $s'=\gamma^{-1}(s)$. Let $h':S'\to Y$ be a $k$-morphism given by
\begin{align}
h'^{\sharp}:k[y_1,\ldots,y_n]&\to \Gamma(S, \sO_S)[t']/(t'^r-t)\label{keyh'}\\
y_i&\mapsto v_i t'^{\alpha_i}\ \ \ (1\leq i\leq e),\nonumber\\
y_i&\mapsto \tau_i+w_it'^{\alpha_ir}\ \ \ (e+1\leq i\leq n-1),\nonumber\\
y_n&\mapsto \mu+t'^r.\nonumber
\end{align}
 We set $S'_0=S'-\{s'\}$. We have a commutative diagram
\begin{equation*}
\xymatrix{\relax
S'_0\ar[r]^-(0.5){\gamma}\ar[d]_-(0.5){h'}&S_0\ar[d]^-(0.5)h\\
V\ar[r]_{[r]}&V}
\end{equation*}
We have
\begin{align*}
\gamma^*(([r]_*\sL)|_{S_0})\cong h'^*[r]^*[r]_*\sL.
\end{align*}
Since $\gamma:S'\to S$ is tamely ramified at $s$ of degree $r$, to show the ramification of $([r]_*\sL)|_{S_0}$ at $s$ is isoclinic of the conductor \eqref{rLcurve}, it is sufficient to show $h'^*[r]^*[r]_*\sL$ is isoclinic with (cf. subsection \ref{K'/K})
\begin{equation}
\rc_{s'}(h'^*[r]^*[r]_*\sL)=\sum_{i=1}^e \alpha_i p_i-r+1.
\end{equation}
Let $\Pi$ be the set of $(n-1)$-tuples $\ul\varpi=(\varpi_1,\ldots, \varpi_{n-1})$ of $r$-th roots of unity in $k$. For each $\ul\varpi=(\varpi_1,\ldots, \varpi_{n-1})$ in $\Pi$, we have an isomorphism $\cT_{\ul\varpi}:V\to V$ given by
\begin{align*}
\cT_{\ul\varpi}^{\sharp}:k[y_1^{\pm 1},\ldots,y_{n-1}^{\pm 1},y_n]&\to k[y_1^{\pm 1},\ldots,y_{n-1}^{\pm 1},y_n],\\
 y_i&\mapsto \varpi_i y_i\ \  (1\leq i\leq n-1), \\
  y_n&\mapsto y_n.
\end{align*}
By the proper base change theorem, we have
\begin{equation*}
[r]^*[r]_*\sL\cong \bigoplus_{\ul\varpi\in\Pi}\cT_{\ul\varpi}^*\sL.
\end{equation*}
We are left to show, for each $\ul\varpi\in\Pi$,  the rank $1$ sheaf $h'^*\cT_{\ul\varpi}^*\sL$ has
\begin{equation}\label{rank1S'}
\rc_{s'}(h'^*\cT_{\ul\varpi}^*\sL)=\sum_{i=1}^e \alpha_i p_i-r+1.
\end{equation}
We observe that each $\cT_{\ul\varpi}^*\sL$ associates to the following Artin-Schreier covering of $V$
\begin{equation*}
V'_{\ul\varpi}=\spec(\Gamma(V, \sO_V)[T]/(T^p-T-\delta y_n/ \left(y_1^{p_1}\cdots y_{n-1}^{p_{n-1}})\right))
\end{equation*}
where $\delta$ is a $r$-th root of unity in $k$ relative to $\ul\varpi$. We fix a uniformizer $t'$ of $\sO_{S',s'}$ with $t'^r=t\in\sO_{S,s}$ and fix an isomorphism $k[[t']]\cong \wh\sO_{S',s'}$ of complete discrete valuation rings. We denote by $\wh\eta$ the generic point $\spec(\wh\sO_{S',s'})$.  Then, by \eqref{keyh'}, the restriction $(h'^*\cT_{\ul\varpi}^*\sL)|_{\wh\eta}$ corresponds to the following Artin-Schreier extension
\begin{equation}\label{firstfieldL}
L=k((t'))[T]\Big/\left(T^p-T-\delta(\mu+t'^r)\big/(ut'^\alpha )\right),
\end{equation}
of $K=k((t'))$, where $\alpha=\sum_{i=1}^e \alpha_ip_i$ and $u=\prod_{i=1}^ev_i^{ p_i}\cdot\prod_{i=e+1}^{n-1}(\tau_i+w_i t'^{\alpha_i r})^{p_i}$ is a unit of $\sO_K=k[[t']]$. In the case where $\mu=0$, we have
\begin{equation*}
L=K[T]\Big/\left(T^p-T-\delta\big/(ut'^{\alpha-r})\right),
\end{equation*}
 Since $e\geq 1$, $(r,p)=1$ and $\alpha_i\geq 1$, $p_i>r$ and $p|p_i$ for each $1\leq i\leq e$, we have $\alpha-r\geq 1$ and $(\alpha-r,p)=1$. Let $v$ be an $(\alpha-r)$'th root of $u/\delta$ in $\sO_K$ and we replace the uniformizer $t'$ by $t''=vt'$. The field extension $L/K$ \eqref{firstfieldL} is regarded as $L=K[T]/\left(T^p-T-1/t''^{\alpha-r}\right)$ over $K=k((t''))$. By the computation in \cite[1.1.7]{lau}, we obtain
\begin{equation}\label{casemu=0}
\rc_{s'}(h'^*\cT^*_{\ul\varpi}\sL)=\sw_K((h'^*\cT_{\ul\varpi}^*\sL)|_{\wh\eta})+1=\alpha-r+1,
\end{equation}
which verifies \eqref{rank1S'} for the $\mu=0$ situation.
In the case where $\mu\neq 0$, we firstly notice that $b_i=p_i/p^a$ is co-prime to $p$ for each $1\leq i\leq n-1$ and that the unit
\begin{equation*}
w=\prod_{i=1}^ev_i^{b_i}\cdot\prod_{i=e+1}^m(\tau_i+w_i t'^{\alpha_i r})^{b_i}
\end{equation*}
 in $\sO_K$ is a $p^a$-th root of $u=\prod_{i=1}^ev_i^{ p_i}\cdot\prod_{i=e+1}^{n-1}(\tau_i+w_i t'^{\alpha_i r})^{p_i}$. We take
 \begin{equation*}
 T'=T-\sum^{a}_{j=1}\frac{(\delta\mu)^{1/p^j}}{w^{p^{a-j}}t'^{\alpha/p^j}}.
 \end{equation*}
 Then, we have
\begin{align*}
T^p-T-\frac{\delta(\mu+t'^r)}{ut'^\alpha}&=\left(T^p-T-\frac{\delta\mu}{ut'^\alpha}\right)-\frac{\delta}{ut'^{\alpha-r}}\\
&=\left(T'^p-T'-\frac{(\delta\mu)^{1/p^a}}{wt'^{\alpha/p^a}}\right)-\frac{\delta }{ut'^{\alpha-r}}\\
&=T'^p-T'-\frac{(\delta\mu)^{1/p^a}w^{p^a-1}t'^{\alpha(p^a-1)/p^a-r}+\delta }{ut'^{\alpha-r}}
\end{align*}
Observe that
\begin{equation*}
\frac{\alpha(p^a-1)}{p^a}-r=\sum^e_{i=1}\alpha_i(p_i-b_i)-r\geq \alpha_1(p_1-b_1)-r\geq p_1-b_1-r\geq 1
\end{equation*}
Hence, $v=(\delta\mu)^{1/p^a}w^{p^a-1}t'^{\alpha(p^a-1)/p^a-r}+\delta $ is a unit in $\sO_K$. Let $z$ be an $(\alpha-r)$'th root of $uv^{-1}$ in $\sO_K$ and we replace the uniformizer $t'$ by $\wt t=zt'$. The field extension $L/K$ \eqref{firstfieldL} is regarded as $L=K[T']/\left(T'^p-T'-1/\wt t^{\alpha-r}\right)$ over $K=k((\wt t))$. By the computation in \cite[1.1.7]{lau} again, we obtain
\begin{equation*}
\rc_{s'}(h'^*\cT^*_{\ul\varpi}\sL)=\sw_K((h'^*\cT_{\ul\varpi}^*\sL)|_{\wh\eta})+1=\alpha-r+1,
\end{equation*}
which verifies \eqref{rank1S'} for the $\mu\neq 0$ situation. Hence, we obtain that the ramification of $([r]_*\sL)|_{S_0}$ at $s$ is isoclinic with \eqref{rLcurve}. By definition, the ramification of $[r]_*\sL$ along $E$ is $(E\cdot\langle dy_n\rangle)$-isoclinic by restricting to curves with \eqref{CYjrsL}. By Lemma \ref{lemmaisocliniccurve}, the ramification of $\sG$ at each generic point of $E$ is isoclinic and we have \eqref{SSjrsL}. By \eqref{SSjrsL} and \cite[Theorem 7.6]{cc}, we obtain \eqref{CCjrsL}.

For any $1\leq i\leq m$, we have $\bT^*_{E_i}Y=E_i\cdot\langle dy_i\rangle\neq E_i\cdot\langle dy_n\rangle$. Hence, by \cite[Proposition 7.2]{Hu19IMRN}, the Newton polygon and the logarithmic Newton polygon of the ramification of $[r]_*\sL$ at the generic point of $E_i$ are the same. Together with the fact that the ramification of $\sF$ at the generic point of each $E_i$ is isoclinic and \eqref{CYjrsL}, we see that the ramification of $[r]_*\sL$ at the generic point of each $E_i$ is logarithmic isoclinic of logarithmic conductor $p_i/r$.
 \end{proof}

\subsection{}
We say $X$ is {\it local} if there is an affine and \'etale map $\pi:X\to Y$ and each $D_i$ associates to a principal ideal of $\Gamma(X,\sO_X)$.

\begin{theorem}\label{pureramifiction any D}
Assume that $X$ is local with a structure morphism $\pi:X\to Y$ and that each irreducible component of $D$ contains a common closed point $x\in X$. Let $\omega$ be a non-zero vector in the $k$-vector space $\bT^*_0Y$. Let $p_1,\ldots, p_d,r$ be positive integers satisfying
\begin{itemize}
\item{} $p_i>p$ for each $1\leq i\leq d$;
\item{} $\{p_i=p^ab_i\}_{1\leq i\leq d}$ for positive integers $a, b_1,\ldots, b_d$ such that $(b_i,p)=1$;
\item{} $r< \min_{1\leq i\leq d}\{p_i-b_i\}$ and $(r,p)=1$.
\end{itemize}
 Then we can find a locally constant and constructible sheaf of $\Lambda$-modules $\sG$ on $U$ of rank $r^d$ such that
 the ramification of $\sG$ along $D$ is $(D\cdot\langle\omega\rangle)$-isoclinic by restricting to curves with 
 \begin{equation}\label{CXjsG}
\rC_X(j_!\sG)=\sum_{i=1}^d \frac{p_i}{r}\cdot D_i. 
\end{equation}
 In particular, the ramification of $\sG$ at each generic point of $D$ is isoclinic and we have 
 \begin{align*}
SS(j_!\sG)&=\bT^*_XX\bigcup D\cdot\langle \omega\rangle, \\
CC(j_!\sG)&=(-1)^{n}\bigg(r^d[\bT^*_XX]+ \sum_{i=1}^d r^{d-1}p_i[D_i\cdot\langle \omega\rangle]\bigg).
\end{align*}
\end{theorem}
\begin{proof}
We only need to find a sheaf $\sG$ satisfying, for any closed point $u$ of $D$ and any element $(S,h:S\to X,u)$ of $\cQ(X,D)$ such that $h:S\to X$ is $\langle \omega\rangle$-transversal at $s=h^{-1}(u)$, the ramification of $\sG|_{S_0}$ at $s$ is isoclinic with
\begin{equation}
\rc_s(\sG|_{S_0})=\sum_{i=1}^d m_s(h^*D_i)\cdot \frac{p_i}{r}.
\end{equation}
The rest is due to Lemma \ref{lemmaisocliniccurve} and \cite[Theorem 7.6]{cc}.

Let $Z=\spec(k[z_1,\ldots,z_{n+d}])$ be an affine space of dimension $n+d$,  $H_i$ its Cartier divisor associated to the ideal $(z_i)$ ($1\leq i\leq d$), $H$ the sum of all $H_i$, $W$ the complement of $H$ in $Z$ and $j_W:W\to Z$ the canonical injection.  Let $(f_i)$ be a principal ideal of $\Gamma(X,\sO_X)$ associating to $D_i$ $(1\leq i\leq d)$ and  $\omega=a_1dy_1+\cdots+a_{n}dy_n$ for a non-zero vector $(a_1,\ldots,a_{n})$ in $k^{n}$. We find an index $1\leq \beta\leq n$, such that $a_\beta\neq 0$. We may assume that $y=\pi(x)$ is the origin of $Y=\spec(k[y_1,\ldots, y_n])$. Let $y_1,\ldots,y_n$ be the canonical local coordinate of $X$ at $x$ from coordinates of $Y$.  Let $\iota:X\to Z$ be a $k$-morphism given by
\begin{align*}
\iota^\sharp: k[z_1,\ldots,z_{n+d}]&\to \Gamma(X,\sO_X)\\
z_i&\mapsto f_i,\ \ \ (1\leq i\leq d),\\
z_i&\mapsto y_{i-d},\ \ \ (d+1\leq i\leq n+d,\ \ i\neq \beta+d),\\
z_{\beta+d}&\mapsto a_1y_{1}+\cdots +a_ny_{n}.
\end{align*}
Since each $D_i=(f_i)$ contains $x\in X$, we note that $\iota(x)$ is the origin of $Z$. We have $\iota^{-1}(W)=U$ and $D_i=\iota^*H_i$, for each $1\leq i\leq d$.  The morphism $\iota:X\to Z$ is properly $(H\cdot\langle dz_{\beta+d}\rangle)$-transversal (subsection \ref{propCtrans}) and we have $\iota^\circ(H\cdot\langle dz_{\beta+d}\rangle)=D\cdot\langle \omega\rangle$.

By Theorem \ref{all of rL}, we have a locally constant and constructible sheaf of $\Lambda$-modules $\sM$ of rank $r^d$ on $W$ whose ramification along $H$ is $(H\cdot\langle dz_{\beta+d}\rangle)$-isoclinic by restricting to curves with
\begin{align}\label{CYjrsM}
\rC_Z(\jmath_!\sM)=\sum_{i=1}^d \frac{p_i}{r}\cdot H_i.
\end{align}
We put $\sG=\sM|_U.$  Let $(S,h:S\to X,u)$ be an element of $\cQ(X,D)$ such that $h:S\to X$ is $\langle \omega\rangle$-transversal at $s=h^{-1}(u)$. We note that $(S,\iota\circ h:S\to Z,\iota(u))$ is an element of $\cQ(Z,H)$ and $\iota\circ h:S\to Z$ is $\langle dz_{\beta+d}\rangle$-transversal at $s$. By the isoclinic condition of $\sM$, we obtains that the ramification of $\sG|_{S_0}=\sM|_{S_0}$ at $s$ is isoclinic with the conductor
\begin{align*}
\rc_s(\sG|_{S_0})=\rc_s(\sM|_{S_0})=\sum_{i=1}^d m_s((\iota\circ h)^*H_i))\cdot \frac{p_i}{r}=\sum_{i=1}^d m_s(h^*D_i)\cdot \frac{p_i}{r}.
\end{align*}
We obtain that the ramification of $\sG$ along $D$ is $(D\cdot\langle\omega\rangle)$-isoclinic by restricting to curves with \eqref{CXjsG}.
\end{proof}

\begin{lemma}\label{p_i and r}
Let $c_1,\ldots, c_d\geq 1$ and $\varepsilon>0$ be real numbers. Then, we can find positive integers $a, r, b_1,\ldots, b_d$ satisfying:
\begin{itemize}
\item[(1)]  $p<p^ab_i$ and $(b_i,p)=1$ for any $1\leq i\leq d$;
\item[(2)] $(r,p)=1$ and $r<\min_{1\leq i\leq d}\{p^ab_i-b_i\}$;
\item[(3)]
 $c_i<p^ab_i/r<c_i+\varepsilon$ for every $1\leq i\leq d$.
\end{itemize}
\end{lemma}
\begin{proof}
We firstly take a positive integer $a>\log_p(1+2/\varepsilon)$. Hence we have $\varepsilon>2/(p^a-1)$. Secondly, we take a positive integer $r>\max\{4p^a/\varepsilon,p\}$ which is co-prime to $p$. Then, for every $1\leq i\leq d$, the length of the open interval
\begin{equation}\label{openinterval}
\left(\frac{r}{p^a}\left(c_i+\frac{\varepsilon}{2}\right), \frac{r}{p^a}(c_i+\varepsilon)\right)
\end{equation}
is $r\varepsilon/(2p^a)>2$. Hence, for each $1\leq i\leq m$, the interval \eqref{openinterval} contains two adjacent positive integers and let $b_i$ be one of them which is co-prime to $p$.  Thus, for each $1\leq i\leq d$, we have
\begin{equation}\label{cplusepsilon}
c_i<c_i+\frac{\varepsilon}{2}<\frac{p^ab_i}{r}<c_i+\varepsilon.
\end{equation}
 Finally,  from \eqref{cplusepsilon},  we obtain
\begin{equation*}
r<\frac{p^ab_i}{c_i+\frac{\varepsilon}{2}}\leq \frac{p^ab_i}{1+\frac{\varepsilon}{2}}<\frac{p^ab_i}{1+\frac{1}{p^a-1}}=p^ab_i-b_i,
\end{equation*}
for each $1\leq i\leq d$. In particular, $p<r< p^ab_i$ for any $1\leq i\leq d$. In summary, the integers $a,r,b_1,\ldots,b_d$ satisfy conditions (1), (2) and (3).
\end{proof}

\begin{proposition}\label{Ctheorem curve}
Let $\sF$ be a locally constant and constructible sheaf of $\Lambda$-modules on $U$. For any element $(S,h:S\to X,x)$ of $\cQ(X,D)$, we have
\begin{equation}\label{Ccutbycurve}
m_s(h^*(\rC_X(j_!\sF)))\geq \rc_s(\sF|_{S_0})
\end{equation}
\end{proposition}
\begin{proof}
Let  $(S,h:S\to X,x)$ be an element of $\cQ(X,D)$.
The morphism $h:S\to X$ is a composition of the graph $\Gamma_h:S\to X\times_kS$ and the canonical projection $\pr_1:X\times_kS\to X$. Since $\pr_1:X\times_kS\to X$ is smooth, we have (\cite[Proposition 2.22 and 3.8]{wr})
\begin{equation*}
\pr_1^*(\rC_X(j_!\sF))=\rC_{(X\times_kS)}(\pr_1^*j_!\sF)).
\end{equation*}
Hence, to show \eqref{Ccutbycurve}, it is sufficient to show
\begin{equation*}
m_s(\Gamma_h^*(\rC_{(X\times_kS)}(\pr_1^*j_!\sF)))\geq \rc_s(\Gamma^*_h(\pr_1^*j_!\sF)).
\end{equation*}
Thus, after replacing $X$ by $X\times_kS$ and replacing $j_!\sF$ by $\pr_1^*j_!\sF$, we may assume that $h:S\to X$ is a closed immersion and $\dim_k X=n\geq 2$. The question is local. We may further assume that $x=S\bigcap D$ is a closed point of $X$.

 Suppose that there is an element $(S,h:S\to X,x)$ of $\cI(X,D)$ such that \eqref{Ccutbycurve} is not valid, i.e.,
\begin{equation}\label{assumption c>C}
\rc_x\left(\sF|_{S_0}\right)>\sum_{i=1}^d m_x(h^*D_i)\cdot\rc_{D_i}(\sF).
\end{equation}
 We proceed by contradiction. We may assume that $k$ is algebraically closed. After replacing $X$ by an affine Zariski neighborhood of $x$, we may assume that $X$ is local with a structure morphism $\pi:X\to Y$. Since $n\geq 2$, we can find a non-zero vector $\omega\in\bT^*_0Y$ such that $\omega\not\in \ker(dh)$, where $dh:\bT^*_xX\to\bT^*_xS$ is the canonical map. We see that $h:S\to X$ is $\langle \omega\rangle$-transversal at $x$. We simply put
\begin{equation*}
m_i=m_x(h^*D_i) \ \ \ \textrm{and}\ \ \ c_i=\rc_{D_i}(\sF),
\end{equation*}
for each $1\leq i\leq d$. Let $\delta$ be a positive real number such that
\begin{equation*}
\rc_x\left(\sF|_{S_0}\right)>\sum^d_{i=1}m_ic_i+\delta\ \ \ \textrm{and}\ \ \ \rS_x(\sF|_{S_0})\bigcap \bigg(\sum_{i=1}^d m_ic_i, \sum_{i=1}^d m_ic_i+\delta\bigg)=\emptyset
\end{equation*}
Put $\varepsilon=\delta/(\sum_{i=1}^d m_i)$. By Lemma \ref{p_i and r}, we can find positive integers $a, r, b_1,\ldots, b_m$ satisfying:
\begin{itemize}
\item[(1)]  $p<p^ab_i$  and $(b_i,p)=1$ for any $1\leq i\leq d$;
\item[(2)]$(r,p)=1$ and $r<\min_{1\leq i\leq d}\{p^ab_i-b_i\}$,
\item[(3)]
 $c_i<p^ab_i/r<c_i+\varepsilon$ for every $1\leq i\leq d$.
\end{itemize}
We put $p_i=p^ab_i$ for each $1\leq i\leq d$. By Theorem \ref{pureramifiction any D}, we have a locally constant and constructible sheaf of $\Lambda$-modules $\sG$ of rank $r^d$ on $U$ whose ramification along $D$ is $(D\cdot\langle\omega\rangle)$-isoclinic by restricting to curves with
\begin{equation}\label{rCjsG=sumpi/rDi}
\rC_X(j_!\sG)=\sum_{i=1}^d \frac{p_i}{r}\cdot D_i.
\end{equation}
Hence, the ramification of $\sG|_{S_0}$ at $x$ is isoclinic with the conductor
\begin{equation*}
\rc_x(\sG|_{S_0})=\sum_{i=1}^d \frac{m_ip_i}{r}.
\end{equation*}
We consider the locally constant and constructible sheaf $\sF\otimes_{\Lambda}\sG$ on $U$.
Since the ramification of $\sG$ at the generic point of $D_i$ is isoclinic and 
\begin{equation*}
c_{D_i}(\sG)=\frac{p_i}{r}>c_i=c_{D_i}(\sF),
\end{equation*}
we obtain that the ramification of $\sF\otimes_{\Lambda}\sG$ at the generic point of $D_i$ is isoclinic with
\begin{equation*}
c_{D_i}(\sF\otimes_{\Lambda}\sG)=c_{D_i}(\sG)=\frac{p_i}{r},
\end{equation*}
for any $1\leq i\leq d$. Hence, we have
\begin{equation*}
\DT_X(j_!(\sF\otimes_{\Lambda}\sG))=\sum^d_{i=1}\rk_{\Lambda}\sF\cdot r^{d-1}p_i\cdot D_i.
\end{equation*}
Then, we have
\begin{equation}\label{mxDTSG}
m_x(h^*(\DT_X(j_!(\sF\otimes_{\Lambda}\sG))))=\sum^d_{i=1}\rk_{\Lambda}\sF\cdot r^{d-1}p_i\cdot m_x(h^*D_i)=\sum^d_{i=1}\rk_{\Lambda}\sF\cdot r^{d-1}m_ip_i
\end{equation}
By \eqref{rCjsG=sumpi/rDi}, we obtain that
\begin{equation*}
\sum^d_{i=1}m_ic_i<\rc_x(\sG|_{S_0})=\sum_{i=1}^d \frac{m_ip_i}{r}<\sum_{i=1}^dm_i(c_i+\varepsilon)=\sum^d_{i=1}m_ic_i+\delta.
\end{equation*}
Hence $\rc_x(\sG|_{S_0})\not\in \rS_x(\sF|_{S_0})$ and $\rc_x(\sG|_{S_0})<\rc_x(\sF|_{S_0})$. Then, we have  (subsection \ref{slopedecom})
\begin{equation}\label{rkFGcGdtFG}
\sum^d_{i=1}\rk_{\Lambda}\sF\cdot r^{d-1}m_ip_i=\rk_{\Lambda}(\sF\otimes_{\Lambda}\sG)\cdot \rc_x(\sG|_{S_0})<\dt_x((\sF\otimes_{\Lambda}\sG)|_{S_0}).
\end{equation}
By \eqref{mxDTSG} and \eqref{rkFGcGdtFG}, we obtain
\begin{align*}
m_x(h^*(\DT_X(j_!(\sF\otimes_{\Lambda}\sG))))&<\dt_x((\sF\otimes_{\Lambda}\sG)|_{S_0}).
\end{align*}
It contradicts to Theorem \ref{DTtheorem} (1). Hence, our assumption \eqref{assumption c>C} does not hold. We obtain the proposition.
\end{proof}

\begin{proposition}\label{curiso=geniso+pure}
Let $C$ be a closed conical subset of $\bT^*X$ such that $B(C)=D$ and that $\dim_{k(\ol x)}C_{\ol x}=1$ for any geometric point $\ol x\to D$. Let $\sG$ be a locally constant and constructible sheaf of $\Lambda$-modules on $U$. Then, the following two conditions are equivalent:
\begin{itemize}
\item[(1)]
The ramification of $\sG$ along $D$ is $C$-isoclinic by restricting to curves.
\item[(2)]
The ramification of $\sG$ at each generic point of $D$ is isoclinic and $SS(j_!\sG)\subseteq \bT^*_XX\bigcup C$.
\end{itemize}
\end{proposition}
\begin{proof}
(1)$\Rightarrow$(2) is from Proposition \ref{equivCPandP} and Lemma \ref{lemmaisocliniccurve}.

(2)$\Rightarrow$(1). We may assume that $k$ is algebraically closed.  Let $(S,h:S\to X,x)$ be an element $(S,h:S\to X,x)$ $\cQ(X, D)$ such that $h:S\to X$ is $C$-transversal at $s=h^{-1}(x)$. On one hand, by Proposition \ref{equivCPandP} and the condition in (2), we have 
\begin{equation}\label{dtsGS0=mshrCsG}
\dt_s(\sG|_{S_0})=m_s(h^*(\DT_X(j_!\sG)))=\rk_{\Lambda}\sG\cdot m_s(h^*(\rC_X(j_!\sG))).
\end{equation}
On the other hand, by Proposition \ref{Ctheorem curve}, we have
\begin{equation}\label{dtsGS0<=mshrCsG}
\dt_s(\sG|_{S_0}) \leq \rk_{\Lambda}\sG\cdot\rc_s(\sG|_{S_0})\leq \rk_{\Lambda}\sG\cdot m_s(h^*(\rC_X(j_!\sG))).
\end{equation}
By \eqref{dtsGS0=mshrCsG} and \eqref{dtsGS0<=mshrCsG}, we see that the ramification of $\sG|_{S_0}$ at $s$ is isoclinic with 
$$\rc_s(\sG|_{S_0})= m_s(h^*(\rC_X(j_!\sG))).$$
Hence, we obtain (1).
\end{proof}

\begin{proposition}\label{sGCisoclinic+CXjG>CXjF=>Rj*sGsF=j!sGsF}
Let $C$ be a closed conical subset of $\bT^*X$ such that $B(C)=D$ and that $\dim_{k(\ol x)}C_{\ol x}=1$ for any geometric point $\ol x\to D$. Let $\sF$ and $\sG$ be locally constant and constructible sheaves of $\Lambda$-modules on $U$ satisfying:
\begin{itemize}
\item[1)]
The ramification of $\sG$ along $D$ is $C$-isoclinic by restricting to curves;
\item[2)]
$\rc_{D_i}(\sG)>\rc_{D_i}(\sF)$ for any $1\leq i\leq d$.
\end{itemize}
Then, the ramification of $\sG\otimes_{\Lambda}\sF$ along $D$ is $C$-isoclinic by restricting to curves with 
\begin{align}
 SS(j_!(\sG\otimes_{\Lambda}\sF))&=SS(j_!\sG)\subseteq \bT^*_XX\bigcup C, \label{SSjsGotsF=SSjsGinT*XXC}\\
  CC(j_!(\sG\otimes_{\Lambda}\sF))&=\rk_{\Lambda}\sF\cdot CC(j_!\sG). \label{CCjsGotsF=rksFCCjsG}
\end{align}
In particular, we have 
\begin{equation}\label{formulaRj*sGsF=j!sGsF}
 Rj_*(\sG\otimes_{\Lambda}\sF)=j_!(\sG\otimes_{\Lambda}\sF).
\end{equation}
\end{proposition}

\begin{proof}
 By Lemma \ref{lemmaisocliniccurve}, the ramification of $\sG$ at each generic point of $D$ is isoclinic and that $SS(j_!\sG)\subseteq \bT^*_XX\bigcup C$. We put $C'=SS(j_!\sG)\times_XD\subseteq C$. By Proposition \ref{curiso=geniso+pure}, we obtain that the ramification of $\sG$ along $D$ is $C'$-isoclinic by restricting to curves.  Let $(S,h:S\to X,x)$ be an element of $\cQ(X,D)$ such that $h:S\to X$ is $C'$-transversal at $x$. By conditions on $\sG$,  the ramification of $\sG|_{S_0}$ at $s$ is isoclinic with
\begin{equation}
\rc_s(\sG|_{S_0})=m_s(h^*(\rC_X(j_!\sG)))>m_s(h^*(\rC_X(j_!\sF)))\geq \rc_s(\sF|_{S_0}),
\end{equation}
and we have $$\rC_X(j_!(\sG\otimes_{\Lambda}\sF))=\rC_X(j_!\sG).$$
Hence, the ramification of $(\sG\otimes_{\Lambda}\sF)|_{S_0}$ at $s$ is isoclinic with 
\begin{align}
\rc_s((\sG\otimes_{\Lambda}\sF)|_{S_0})=\rc_s(\sG|_{S_0})=m_s(h^*(\rC_X(j_!\sG)))=m_s(h^*(\rC_X(j_!(\sG\otimes_{\Lambda}\sF)))).
\end{align}
Thus, the ramification of $\sG\otimes_{\Lambda}\sF$ along $D$ is $C'$-isoclinic by restricting to curves. In particular, the ramification of $\sG\otimes_{\Lambda}\sF$ along $D$ is $C$-isoclinic by restricting to curves. Applying Proposition \ref{curiso=geniso+pure} to $\sG\otimes_{\Lambda}\sF$, we obtain that 
\begin{equation}\label{SSj!sGsFsubseteqSSj!sG}
SS(j_!(\sG\otimes_{\Lambda}\sF))\subseteq \bT^*_XX\bigcup C'=SS(j_!\sG).
\end{equation}
We put $\sF^{\vee}=\mathcal{H}om_{\Lambda}(\sF,\Lambda)$. For any $1\leq i\leq d$, we have 
\begin{equation*}
\rc_{D_i}(j_!(\sG\otimes_{\Lambda}\sF))=\rc_{D_i}(j_!\sG)>\rc_{D_i}(j_!\sF)=\rc_{D_i}(j_!\sF^{\vee}).
\end{equation*}
Replacing $\sG$ by $\sG\otimes_{\Lambda}\sF$ and replacing $\sF$ by $\sF^{\vee}$, we run the proof above again and get 
\begin{equation}\label{SSj!sGsFsFvsubseteqSSj!sGsF}
SS(j_!(\sG\otimes_{\Lambda}\sF\otimes_{\Lambda}\sF^{\vee}))\subseteq SS(j_!(\sG\otimes_{\Lambda}\sF)).
\end{equation}
We have a short exact sequence of locally constant and constructible sheaves on $U$:
$$0\to \mathscr H\to \sF\otimes_{\Lambda}\sF^{\vee}\xrightarrow{\mathrm{Ev}} \Lambda\to 0.$$
Hence, by \cite[Theorem 1.4]{bei}, we have
\begin{equation}\label{SSj!sGsubseteqSSj!sGsFsFv}
SS(j_!\sG)=SS(j_!(\sG\otimes_\Lambda\Lambda))\subseteq SS(j_!(\sG\otimes_{\Lambda}\sF\otimes_{\Lambda}\sF^{\vee})).
\end{equation}
By \eqref{SSj!sGsFsubseteqSSj!sG}, \eqref{SSj!sGsFsFvsubseteqSSj!sGsF} and \eqref{SSj!sGsubseteqSSj!sGsFsFv},  we get \eqref{SSjsGotsF=SSjsGinT*XXC}. By \cite[Theorem 7.14]{cc} and \eqref{SSjsGotsF=SSjsGinT*XXC}, we get \eqref{CCjsGotsF=rksFCCjsG}. The equality \eqref{formulaRj*sGsF=j!sGsF} is from condition 2) for $\sG$, \eqref{SSjsGotsF=SSjsGinT*XXC} and Lemma \ref{section5lemmaRj*Rj}.
\end{proof}

\section{Decreasing of conductor divisors after pull-backs}\label{proofCLC}

\subsection{}

In this section, $k$ denotes a perfect field of characteristic $p>0$, $X$  a smooth $k$-scheme of dimension $n\geq 1$, $D$ a reduced Cartier divisor on $X$, $\{D_i\}_{ i\in I}$ the set irreducible components of $D$, $U$ the complement of $D$ in $X$ and $j:U\to X$ the canonical injection. Let $\sF$ be a non-trivial locally constant and constructible sheaf of $\Lambda$-modules on $U$. We take the notation in Definition \ref{defCDTNP}.

\begin{theorem}\label{Ctheorem}
 Let $f:Y\to X$ be a morphism of smooth $k$-schemes. We assume that $f^*D=D\times_XY$ is a Cartier divisor of $Y$.
Then, we have
\begin{equation}\label{firstinequalityCtheorem}
f^*(\rC_X(j_!\sF))\geq \rC_Y(f^*j_!\sF).
\end{equation}
\end{theorem}
\begin{proof}  
We may assume that $k$ is algebraically closed. It is a Zariski local question. Hence, after replacing $Y$ by Zariski open neighborhoods of generic points of $f^*D$,  we may assume that $E=(f^*D)_{\mathrm{red}}$ is an irreducible and smooth divisor of $Y$. We denote by $V$ the complement of $E$ in $Y$. Let $y$ be a closed point of $E$ such that the ramification of $\sF|_V$ along $E$ is non-degenerate at $y$, let $h:S\to Y$ an immersion from a connected smooth $k$-curve to $Y$ such that $y=S\bigcap E$, that $m_y(h^*E)=1$ and that $h:S\to Y$ is $SS(f^*j_!\sF)$-transversal at $y$. We put $S_0=S-\{0\}$. By Theorem \ref{DTtheorem} (2), we have
\begin{equation*}
\rc_y(\sF|_{S_0})=\rc_E(\sF|_V).
\end{equation*}
Applying Proposition \ref{Ctheorem curve} to the composition $f\circ h:S\to X$ and the sheaf $\sF$, we have
\begin{equation*}
h^*(f^*(\rC_X(j_!\sF)))\geq \rC_y((f\circ h)^*j_!\sF)=\rc_y(\sF|_{S_0})\cdot y.
\end{equation*}
Hence, we obtain
\begin{equation*}
f^*(\rC_X(j_!\sF))\geq \rc_y(\sF|_{S_0})\cdot E=\rc_E(\sF|_V)\cdot E=\rC_Y(f^*j_!\sF).
\end{equation*}
We get \eqref{firstinequalityCtheorem}.
\end{proof}

\begin{remark}
Using \cite[Proposition 2.22]{wr} and a cutting-by-curve method in \cite[Corollary 3.9]{wr}, the inequality in Theorem \ref{Ctheorem} becomes an equality if we add more ramification and geometric conditions. More precisely,  let $f:Y\to X$ be a morphism of smooth $k$-schemes such that that $f^*D=D\times_XY$ is a Cartier divisor of $Y$. We assume that the image of each generic point of $f^*D$ by $f:Y\to X$ in $D$ is contained in the non-degenerate and smooth locus of the ramification of $\sF$ and that $f:Y\to X$ is $SS(j_!\sF)$-transversal. Then, we have 
\begin{equation}\label{secondequalityCtheorem}
f^*(\rC_X(j_!\sF))=\rC_Y(f^*j_!\sF).
\end{equation}

\end{remark}
%(2). This part is a corollary of Theorem \ref{DTtheorem} (2). Since the image of the generic point of $E$ is contained in the non-degenerate locus of the ramification of $\sF$ on $D$, we can find an open dense subset $E_0$ of $E$ such that the ramification of $\sF|_V$ along $E_0$ is non-degenerate and that $f(E_0)$ is contained in the non-degenerate locus of the ramification of $\sF$ along $D$. Let $y$ be a closed point of $E_0$,  $h:S\to Y$ be an immersion from a connected smooth $k$-curve to $Y$ such that $y=S\bigcap E$, that $m_y(h^*E)=1$ and that $h:S\to Y$ is $SS(f^*j_!\sF)$-transversal at $y$. We put $S_0=S-\{0\}$. By Theorem \ref{DTtheorem} (2), we have
%\begin{equation*}
%\rc_y(\sF|_{S_0})=\rc_E(\sF|_V).
%\end{equation*}
%Since $f:Y\to X$ is $SS(j_!\sF)$-transversal, $SS(f^*j_!\sF)$ and $f^\circ(SS(j_!\sF))$ have the same fiber at $y$. Hence the composition $f\circ h:S\to X$ is $SS(j_!\sF)$-transversal. Applying Theorem \ref{DTtheorem} (2) to $f\circ h:S\to X$ and the sheaf $\sF$, we have 
%\begin{equation*}
%h^*(f^*(\rC_X(j_!\sF)))=\rC_y((f\circ h)^*j_!\sF)=\rc_y(\sF|_{S_0})\cdot y.
%\end{equation*}
%Hence, we obtain
%\begin{equation*}
%f^*(\rC_X(j_!\sF))= \rc_y(\sF|_{S_0})\cdot E=\rc_E(\sF|_V)\cdot E=\rC_Y(f^*j_!\sF).
%\end{equation*}
%We obtain (2).

\begin{proposition}\label{LCtheorem curve}
We assume that $D$ is a divisor with simple normal crossings.
Let $S$ be a smooth $k$-curve and $h:S\to X$ a quasi-finite morphism such that $h^*D$ is a Cartier divisor of $S$. Then, we have
\begin{equation}\label{LCcutbycurve}
h^*(\rLC_X(j_!\sF))\geq \rLC_S(h^*j_!\sF)).
\end{equation}
\end{proposition}

\begin{proof}
The morphism $h:S\to X$ is a composition of the graph $\Gamma_h:S\to X\times_kS$ and the canonical projection $\pr_1:X\times_kS\to X$. Since $\pr_1:X\times_kS\to X$ is smooth, we have (\cite[Lemma 1.22]{saito cc})
\begin{equation*}
\pr_1^*(\rLC_X(j_!\sF))=\rLC_{(X\times_kS)}(\pr_1^*j_!\sF)).
\end{equation*}
Hence, to show \eqref{LCcutbycurve}, it is sufficient to show
\begin{equation*}
\Gamma_h^*(\rLC_{(X\times_kS)}(\pr_1^*j_!\sF))\geq \rLC_S(\Gamma^*_h(\pr_1^*j_!\sF)).
\end{equation*}
Thus, after replacing $X$ by $X\times_kS$ and replacing $j_!\sF$ by $\pr_1^*j_!\sF$, we may assume that $h:S\to X$ is a closed immersion with $h(S)\not\subseteq D$.

 Suppose that there is an immersion $h:S\to X$ from a connected and smooth $k$-curve to $X$ with $x=S\bigcap D$ a closed point of $X$ such that \eqref{LCcutbycurve} is not valid, i.e.,
\begin{equation}\label{assumption lc>LC}
\rlc_x\left(\sF|_{W}\right)=\sum_{i=1}^m m_x(h^*D_i)\cdot\rlc_{D_i}(\sF)+\varepsilon
\end{equation}
for a real number $\varepsilon>0$, where $W=S-\{0\}$ and $\{D_i\}_{1\leq i\leq m}$ are irreducible components of $D$ containing $x$. We proceed by contradiction.  This is a local question. After replacing $X$ by an affine Zariski open neighborhood of $x$, divisors $D_1,\ldots, D_m$ are defined by $f_1,\ldots,f_m\in\Gamma(X,\sO_X)$, respectively.
 Let $t\in\Gamma(S, \sO_S)$ be a local coordinate of $S$ at $x$.  For simplicity, we put  $d_i=m_x(h^*D_i)$ for each $1\leq i\leq m$. After shrinking $X$ again, we may normalize each $f_i$ by $h^\sharp(f_i)=t^{d_i}$ by the canonical map $h^\sharp:\Gamma(X,\sO_X)\to \Gamma(S,\sO_S)$. Let $r$ be a positive integer co-prime to $p$ and greater than $(\sum^m_{i=1} d_i)/\varepsilon$,  let
\begin{equation*}
X'=\spec(\Gamma(X,\sO_X)[T_1,\ldots, T_m]/(T_1^r-f_1,\ldots, T_m^r-f_m))
\end{equation*}
be a finite cover of $X$ tamely ramified along each $D_i$ and let
\begin{equation*}
S'=\spec(\Gamma(S,\sO_S)[T]/(T^r-t))
\end{equation*}
be a finite cover of $S$ tamely ramified at $x$. We denote by $g_X:X'\to X$ and $g_S:S'\to S$ the canonical projection and we put $W'=W\times_SS'$ and put $U'=U\times_XX'$.
Let $h':S'\to X'$ be a $k$-morphism given by
\begin{equation*}
h'^\sharp:\Gamma(X,\sO_X)[T_1,\ldots, T_m]/(T_1^r-f_1,\ldots, T_m^r-f_m)\to \Gamma(S,\sO_S)[T]/(T^r-t),\ \ \ T_i\to T^{d_i}.
\end{equation*}
We have a commutative diagram
\begin{equation*}
\xymatrix{\relax
S'\ar[d]_-(0.5){h'}\ar[r]^-(0.5){g_S}&S\ar[d]^h\\
X'\ar[r]_-(0.5){g_X}&X}
\end{equation*}
Let $D'_i$ be the Cartier divisor of $X'$ associated to the ideal $(T_i)\subseteq \Gamma(X',\sO_{X'})$ for each $1\leq i\leq m$ and $x'$ the closed point of $S'$ associated to the ideal $(T)\subseteq \Gamma(S',\sO_{S'})$.
Using fundamental properties of Abbes and Saito's ramification theory (subsections \ref{slopedecom} and \ref{K'/K}), we have
\begin{align*}
\rc_{x'}(\sF|_{W'})\geq \rlc_{x'}(\sF|_{W'}) \ \ \  &\textrm{and}\ \ \ \rlc_{x'}(\sF|_{W'})=r\cdot \rlc_{x}(\sF|_W),\\
\rlc_{D'_i}(\sF|_{U'})=r\cdot\rlc_{D_i}(\sF)\ \ \ &\textrm{and}\ \ \ \rlc_{D'_i}(\sF|_{U'})+1\geq\rc_{D'_i}(\sF|_{U'}).
\end{align*}
Hence, noting that $r>(\sum^m_{i=1}d_i)/\varepsilon$, we have
\begin{align}\label{computeclc}
\rc_{x'}(\sF|_{W'})&\geq \rlc_{x'}(\sF|_{W'})=r\cdot \rlc_{x}(\sF|_W)=r\cdot\left(\sum_{i=1}^m d_i\cdot\rlc_{D_i}(\sF)+\varepsilon\right)\\
&>\sum_{i=1}^m rd_i\cdot\rlc_{D_i}(\sF)+\sum_{i=1}^m d_i=\sum_{i=1}^m d_i(r\cdot\rlc_{D_i}(\sF)+1)\nonumber\\
&=\sum_{i=1}^m d_i(\rlc_{D'_i}(\sF|_{U'})+1)\geq \sum_{i=1}^m d_i\cdot\rc_{D'_i}(\sF|_{U'})\nonumber
\end{align}
However, applying Theorem \ref{Ctheorem} to the morphism $h':S'\to X'$ and the sheaf $\sF|_{U'}$, we have
\begin{equation*}
\rc_{x'}(\sF|_{W'})\leq \sum_{i=1}^m \rc_{D'_i}(\sF|_{U'})\cdot m_{x'}(h'^*D'_i)=\sum_{i=1}^m d_i\cdot\rc_{D'_i}(\sF|_{U'}),
\end{equation*}
which contradicts to \eqref{computeclc}. Hence, our assumption \eqref{assumption lc>LC} is not valid. We obtain the proposition.
\end{proof}

 \begin{lemma}[{\cite[Proposition 6.3]{Hu19IMRN}}]\label{huprop6.3}
 Assume that $X$ is affine, connected with $\dim_k X=n\geq 2$ and that $D$ is irreducible and smooth over $\spec(k)$, which is defined by an element $f$ of $\Gamma(X,\sO_X)$. Let $r$ be a positive integer co-prime to $p$,
\begin{equation*}
X'=\spec(\Gamma(X,\sO_X)[T]/(T^r-f))
\end{equation*}
a cyclic cover of $X$ of degree $r$ tamely ramified along $D$, $[r]:X'\rightarrow X$ the canonical projection and $D'$ the smooth divisor $(T)=(D\times_XX')_{\mathrm{red}}$ of $X'$. Let $C$ be a closed conical subset of  $\bT^*X'\times_{X'}D'$ of equidimension $n$ such that $B(C)=D$. Then, we can find a smooth $k$-curve $S$ and an immersion $h':S\rightarrow X'$ such that
\begin{itemize}
\item[(i)] $x'=S\bigcap D'$ is a closed point of $D'$ and $m_{x'}(g'^*D')=1$;
\item[(ii)] $h':S\rightarrow X'$ is $C$-transversal at $x'$;
\item[(iii)] the composition $[r]\circ h':S\rightarrow X$ is also an immersion.
\end{itemize}
\end{lemma}

\begin{proposition}\label{lcx+epsilon}
Assume that $\dim_kX\geq 2$ and that $D$ is irreducible and smooth over $\spec(k)$. Then, for any real number $\varepsilon>0$, we can find a smooth $k$-curve $S$ and an immersion $h:S\to X$ such that $x=S\bigcap D$ is a closed point of $X$ and that
\begin{equation*}
\frac{\rlc_x(\sF|_{S-\{x\}})}{m_x(h^*D)}+\varepsilon> \rlc_D(\sF).
\end{equation*}
\end{proposition}
\begin{proof}
We take an integer $r>1/\varepsilon$ which is co-prime to $p$. Replace $X$ by an affine Zariski neighborhood of the generic point of $D$ such that $D$ is defined by an element $f$ of $\Gamma(X,\sO_X)$. Let \begin{equation*}
X'=\spec(\Gamma(X,\sO_X)[T]/(T^r-f))
\end{equation*}
be a cyclic cover of $X$ of degree $r$ tamely ramified along $D$. We denote by  $[r]:X'\rightarrow X$ the canonical projection, by $D'$ the smooth divisor $(T)=(D\times_XX')_{\mathrm{red}}$ of $X'$, by $U'$ the complement of $D'$ in $X'$ and by $j':U'\to X'$ the caonical injection.
Notice that $D'$ is isomorphic to $D$ as schemes. Hence, after shrinking $X$ again to an affine Zariski neighborhood of the generic point of $D$, we may assume that the ramification of $\sF|_{U'}$ along $D'$ is non-degenerate. By Lemma \ref{huprop6.3}, we can find a closed point $x'$ of $D'$, a smooth $k$-curve $S$ and an immersion $h':S\rightarrow X'$ such that
\begin{itemize}
\item[(i)] $S\bigcap D'=x'$ and $m_{x'}(g'^*D')=1$;
\item[(ii)] $h':S\rightarrow X'$ is $SS(j'_!\sF)$-transversal at $x'$;
\item[(iii)] the composition $[r]\circ h':S\rightarrow X$ is an immersion.
\end{itemize}
Applying Theorem \ref{DTtheorem} (2) to the immersion $h':S\to X'$ and the sheaf $\sF|_{U'}$, we have
\begin{equation}\label{rcx'=rcD'}
\rc_{x'}(\sF|_{S-\{x'\}})=\rc_{D'}(\sF|_{U'})
\end{equation}
Using fundamental properties of Abbes and Saito's ramification theory (subsections \ref{slopedecom} and \ref{K'/K}), we have
\begin{equation}\label{AS3inequalities}
\rlc_{x'}(\sF|_{S-\{x'\}})+1\geq\rc_{x'}(\sF|_{S-\{x'\}}),\ \ \ \rc_{D'}(\sF|_{U'})\geq \rlc_{D'}(\sF|_{U'}),\ \ \ \rlc_{D'}(\sF|_{U'})=r\cdot \rlc_{D}(\sF|_{U}).
\end{equation}
Combining \eqref{rcx'=rcD'} and \eqref{AS3inequalities}, we obtain that
\begin{equation*}
\rlc_{x'}(\sF|_{S-\{x'\}})+1\geq r\cdot \rlc_{D}(\sF|_{U}).
\end{equation*}
Observe that $r=m_{x'}(h'^*(rD'))=m_{x'}(([r]\circ h')^*D)$ and $r>1/\varepsilon$. We get
\begin{equation*}
\frac{\rlc_{x'}(\sF|_{S-\{x'\}})}{m_{x'}(([r]\circ h')^*D)}+\varepsilon> \rlc_{D}(\sF|_{U}).
\end{equation*}
Hence, the curve $S$ and the immersion $[r]\circ h':S\rightarrow X$ satisfies the condition of the proposition.
\end{proof}

\begin{theorem}\label{LCtheorem}
We assume that $D$ is a divisor with simple normal crossings. Let $f:Y\to X$ be a morphism of smooth $k$-schemes. We assume that $f^*D=D\times_XY$ is a Cartier divisor of $Y$.
\begin{itemize}
\item[(1)]
Then, we have
\begin{equation}\label{fLC>LCf}
f^*(\rLC_X(j_!\sF))\geq \rLC_Y(f^*j_!\sF).
\end{equation}
\item[(2)]
We further assume that $D$ is irreducible. Let $\mathcal I(X, D)$ be the set of triples $(S,h:S\to X,x)$ where $h:S\to X$ is an immersion from a smooth $k$-curve $S$ to $X$ such that $x=S\bigcap D$ is a closed point of $X$. Then, we have
\begin{equation}\label{equality lc sup}
\rlc_D(\sF)=\sup_{\mathcal I(X, D)}\frac{\rlc_x(\sF|_{S-\{x\}})}{m_x(h^*D)}.
\end{equation}
\end{itemize}
\end{theorem}

\begin{proof}
(2) We only need to treat the case where $\dim_kX\geq 2$. On one hand, for any element $(S,h:S\to X,x)$ of $\cI(X,D)$, applying Proposition \ref{LCtheorem curve} to $h:S\to X$ and the sheaf $\sF$, we have
$$\rlc_D(\sF) \geq \frac{\rlc_x(\sF|_{S-\{x\}})}{m_x(h^*D)}.$$
Hence we have
\begin{equation}\label{lcDgeqsup}
\rlc_D(\sF)\geq\sup_{\mathcal I(X, D)}\frac{\rlc_x(\sF|_{S-\{x\}})}{m_x(h^*D)}.
\end{equation}
On the other hand, for any positive real number $\varepsilon>0$, we have (Proposition \ref{lcx+epsilon})
\begin{equation*}
\sup_{\mathcal I(X, D)}\frac{\rlc_x(\sF|_{S-\{x\}})}{m_x(h^*D)}+\varepsilon> \rlc_D(\sF).
\end{equation*}
Thus,
\begin{equation}\label{supgeqlcD}
\sup_{\mathcal I(X, D)}\frac{\rlc_x(\sF|_{S-\{x\}})}{m_x(h^*D)}\geq \rlc_D(\sF).
\end{equation}
Combining \eqref{lcDgeqsup} and \eqref{supgeqlcD}, we get part (2) of Theorem \ref{LCtheorem}.

(1) By Proposition \ref{LCtheorem curve}, we only need to verify part (1) in the case where $\dim_k Y\geq 2$. The formula \eqref{fLC>LCf} is Zariski local. Hence, after replacing $Y$ by Zariski open neighborhoods of generic points of $f^*D$,  we may assume that $E=(f^*D)_{\mathrm{red}}$ is an irreducible and smooth divisor of $Y$. We denote by $V$ the complement of $E$ in $Y$. For any real number $\varepsilon>0$, we find a smooth $k$-curve $S$ and an immersion $h:S\to Y$ such that $y=S\bigcap E$ is a closed point of $Y$ and (Proposition \ref{lcx+epsilon})
\begin{equation}\label{lcy/m>lcE}
\frac{\rlc_y(\sF|_{S-\{y\}})}{m_y(h^*E)}+\varepsilon> \rlc_E(\sF|_V).
\end{equation}
Applying Proposition \ref{LCtheorem curve} to the composition $f\circ h:S\to X$ and the sheaf $\sF$, we obtain that
\begin{equation}\label{lcDgeqlcy}
\sum^m_{i=1} \rlc_{D_i}(\sF)\cdot m_y((f\circ h)^*D_i)\geq \rlc_y(\sF|_{S-\{y\}}),
\end{equation}
where $D_1,\ldots, D_m$ are irreducible components of $D$ with $E=(f^*D_i)_{\mathrm{red}}$. Let $\xi$ be the generic point of $E$. Note that
$m_y((f\circ h)^*D_i)=m_y(h^*E)\cdot m_{\xi}(f^*D_i)$ for each $1\leq i\leq m$.
Then, \eqref{lcy/m>lcE} and \eqref{lcDgeqlcy} imply that
\begin{equation*}
\sum^m_{i=1} \rlc_{D_i}(\sF)\cdot m_{\xi}(f^*D_i)\geq\frac{\rlc_y(\sF|_{S-\{y\}})}{m_y(h^*E)}> \rlc_E(\sF|_V)-\varepsilon,
\end{equation*}
for any real number $\varepsilon>0$. Hence, we have
\begin{equation*}
\sum^m_{i=1} \rlc_{D_i}(\sF)\cdot m_{\xi}(f^*D_i)\geq \rlc_E(\sF|_V),
\end{equation*}
i.e.,
\begin{equation*}
f^*(\rLC_X(j_!\sF))\geq \rLC_Y(f^*j_!\sF).
\end{equation*}
\end{proof}

%\begin{remark}
%Theorem \ref{Ctheorem} and Theorem \ref{LCtheorem} are analogue results of Theorem \ref{DTtheorem} (1) and Theorem \ref{SWtheorem} for conductor divisors and logarithmic conductor divisors, respectively. The common special case where $\rk_{\Lambda}\sF=1$ of Theorem \ref{SWtheorem} and Theorem \ref{LCtheorem} is done in \cite{bar}, aiming at a conjecture of Esnault and Kerz (\cite[Conjectures A and B]{bar}). Theorem \ref{SWtheorem} answers the conjecture when the base scheme is smooth and Theorem \ref{LCtheorem} enriches the content of the conjecture. When the divisor $D$ is smooth, Theorem \ref{Ctheorem} (1) and Theorem \ref{LCtheorem} (1) was partially proved in \cite{HT18}.
%\end{remark}

\section{Semi-continuity of conductors}\label{semi-contcondsection}

\subsection{}\label{semicontnotation}
In this section, let $S$ denote an excellent Noetherian scheme, $f:X\to S$ a separated and smooth morphism of finite type of relative dimension $1$, $D$ a reduced closed subscheme of $X$ such that the restriction $f|_D:D\to S$ is quasi-finite and flat, $U$ the complement of $D$ in $X$ and $j:U\to X$ the canonical injection. For any morphism $g:T\to S$, we put $X_{T}=X\times_ST$, put $U_{T}=U\times_ST$ and put $D_{T}=D\times_ST$. Let $\Lambda$ be a finite field of characteristic $\ell$, where $\ell$ is invertible in $S$ and $\sF$ a locally constant and constructible sheaf of $\Lambda$-modules on $U$.

Let $\ol s\to S$ be an algebraic geometric point above a point $s\in S$. For each point $x$ of $D_{\ol s}$, the total dimension $\dt_x(\sF|_{U_{\ol s}})$ and the conductor $\rc_x(\sF|_{U_{\ol s}})$ are independent of the choice of $\ol s\to S$ above $s$. We have the following two functions
\begin{align*}
\varphi: S\to \mathbb Z,&\ \ \ s\mapsto \sum_{x\in D_{\ol s}}\dt_x(\sF|_{U_{\ol s}})\\
\chi: S\to \mathbb Q,&\ \ \ s\mapsto \sum_{x\in D_{\ol s}}\rc_x(\sF|_{U_{\ol s}})
\end{align*}

In a main result of \cite{lau}, Deligne and Laumon proved that $\varphi:S\to \mathbb Z$ is constructible and lower semi-continuous. In this section, we prove that the same property hold for $\chi: S\to \mathbb Q$ in the equal characteristic case.

\begin{theorem}\label{theoremlowsemicontchi}
We take the notation and assumptions of subsection \ref{semicontnotation}. Then, the function $\chi: S\to \mathbb Q$ is constructible.  If we further assume that $S$ is an $\bF_p$-scheme $(p\neq \ell)$, then the function $\chi: S\to \mathbb Q$ is lower semi-continuous.
\end{theorem}
\begin{proof}
Following the strategy  in \cite[\S 3]{lau}, the constructibility of $\chi:S\to \mathbb Q$ is reduced to proving the following case:

{\it $(\ast)$ The scheme $X=\spec(A)$ is affine and integral, $D$ is defined by a principal prime ideal $\fp=(t)\subset A$ and the inclusion $\iota:D\to X$ is a section of $f:X\to S$. The sheaf $\sF$ is trivialized by a Galois \'etale covering $\wt U$ of $U$ of the group $G$ and the normalization $\wt X=\spec(\wt A)$ of $X$ in $\wt U$  is smooth over $S$ and is defined by an Eisenstein polynomial
\begin{equation*}
\wt A=A[\varpi]/(\varpi^e+a_{e-1}\varpi^{e-1}+\cdots + a_1\varpi+a_0),
\end{equation*}
where $a_1,\ldots, a_{e-1}\in \fp$ and $a_0\in \fp-\fp^2$. Then, the function $\chi:S\to \mathbb Q$ is constant in an open dense subset of $S$.}

We mimic \cite[Lemme 3.9.1]{lau} to prove $(\ast)$.
The inclusion $A_{\fp}\to A_{\fp}\otimes_A\wt A$ is a totally ramified extension of discrete valuation rings of the Galois group $G$ and $\varpi$ is a uniformizer of $ A_{\fp}\otimes_A\wt A$. For any $g\in G$, we have $$g(\varpi)-\varpi=\frac{\alpha_g}{\beta_g} \varpi^{i_{g}}\in A_{\fp}\otimes_A\wt A,$$ where $\alpha_g\in \wt A-\varpi \wt A$, $\beta_g\in A-\fp$ and $i_{g}\in \mathbb Z_{\geq 1}$. The lower numbering filtration of the Galois group $G=\gal((A_{\fp}\otimes_A\wt A)/A_{\fp})$  is
\begin{equation*}
G_i=\{g\in G\;|\; i_g\geq i+1\}
\end{equation*}
Let $\eta$ be the generic point of $S$. By the Herbrand function, we have
\begin{equation*}
\chi(\eta)=1+\int^n_0\frac{dt}{[G:G_t]}, \ \ \ n=\max\{i\in \mathbb Z\;|\; \sF(\wt U)^{G_i}\neq \sF(\wt U)\}.
\end{equation*}
Let $V$ be an open dense subset of $D=S$ defined by $a_0/t\neq 0$, $\alpha_g\neq 0$ and $\beta_g\neq 0$ for any $g\in G$. Let $v\in V$ be a point and $\ol v\to V$ an algebraic geometric point above $v$. Then we have
$$\wt X_{\ol v}=\spec(A_{\ol v}[\varpi]/(\varpi^e+\ol{a_{e-1}}\varpi^{e-1}+\cdots + \ol{a_1}\varpi+\ol{a_0})),$$
where $A_{\ol v}=A\otimes_{\Gamma(S,\sO_S)}k(\ol v)$.
Since $a_0/t\neq 0$ at $v$, the fiber $\wt X_{\ol v}$ is also defined by an Eisenstein polynomial on $X_{\ol v}$ ramified at $\ol v=D_{\ol v}$. Hence $\varpi$ is a uniformizer of $\sO_{\wt X_{\ol v},\ol v}$ and the extension $\sO_{X_{\ol v},\ol v}\to \sO_{\wt X_{\ol v},\ol v}$ of discrete valuation rings is also Galois of group $G$. For any $g\in G$ we have
$$g(\varpi)-\varpi=\frac{\alpha_g}{\beta_g} \varpi^{i_{g}}\in \sO_{\wt X_{\ol v},\ol v}.$$
Since $\alpha_g\not\in\varpi\sO_{\wt X_{\ol v},\ol v}$ and $\beta_g\neq 0$, the Galois group $G=\gal(\sO_{\wt X_{\ol v},\ol v}/\sO_{X_{\ol v},\ol v})$ has the same lower numbering filtration as $G=\gal((A_{\fp}\otimes_A\wt A)/A_{\fp})$. Moreover, the representation of $G$ associated to $\sF|_{U_{\ol v}}$ coincides with that of $G$ associated to $\sF$. Thus, $\chi(v)=\chi(\eta)$. The assertion $(\ast)$ is proved and we obtain the constructibility of $\chi:S\to\mathbb Q$.

Now we  assume that $S$ is an $\bF_p$-scheme. We proceed the proof of the lower semi-continuity of $\chi:S\to\mathbb Q$ by the following d\'evissage:

(i) It is sufficient to show that, for any $t\in S$ and any $s\in \ol{\{t\}}$, we have $\chi(t)\geq \chi(s)$. Hence, after replacing $S$ by an affine open neighborhood of $s$ in $\ol{\{t\}}$, we may assume that $S$ is integral and affine. Since $\Gamma(S,\sO_S)=\varinjlim_{\lambda} A_{\lambda}$ for finitely generated $\bF_p$-algebras $\{A_{\lambda}\}_{\lambda}$, the morphism $f:X\to S$ is of finite presentation and $\sF$ is constructible, by a passage to limit argument, we may replace $S$ by $S_\lambda=\spec(A_{\lambda})$ for some $\lambda$. Hence, we are reduced to the case where $S$ is an affine $\mathbb F_p$-scheme of finite type.

(ii) Then, we can reduced to the case where $S$ is an affine $\ol{\mathbb F}_p$-scheme of finite type by base-changing. Combining with the constructibility of $\chi:S\to\mathbb Q$, we may further reduce to the case where $S$ is a connected and smooth $\ol{\bF}_p$-curve by base-changing. In this case, it is sufficient to show that, after fixing a closed point $s$ of $S$, there exists a Zariski open neighborhood $V$ of $s\in S$ such that, for every closed point $t\in V$, we have $\chi(t)\geq \chi(s)$.

(iii) By \cite[IV, 18.12.1]{EGA4}, we can find an \'etale neighborhood $S'$ of $s\in S$ such that $D_{S'}$ is a disjoint union $(\coprod^d_{i=1} D'_i)\coprod E'$, where each $D'_i$ is finite and flat over $S'$ with the fiber $(D'_i)_s$ is a single point in  $X_s$ and $E'$ is quasi-finite and flat over $S'$ with $E'_s=\emptyset$. Hence, after replacing $S$ by $S'$ and replacing $X$ by each $X_{S',r}=X_{S'}-E'-\coprod_{i\neq r} D'_i$ for ${r=1,\ldots, d}$, we may assume that the restriction $f|_D:D\to S$ is finite and flat and that $z=(D_s)_{\mathrm{red}}$ is a single point of $X_s$.

 (iv) By the generically universally locally acyclic theorem, there is an open dense subscheme $V$ of $S$ such that $f:X_V\to V$ is universally locally acyclic with respect to $(j_!\sF)|_{X_V}$. By \cite[Corollary 1.5.7]{SaitoDirectimage}, we can find a connected and smooth $k$-curve $\wt V$ and a finite and faithfully flat morphism $\theta:\wt V\to V$ such that $\wt f:X_{\wt V}\to \wt V$ is $SS((j_!\sF)|_{X_{\wt V}})$-transversal. After replacing $S$ by the normalization of $S$ in $\wt V$, we may further assume that, taking $W=S-\{s\}$, the smooth morphism $f:X_W\to W$ is $SS((j_!\sF)|_{X_W})$-transversal.

(v) Let $\{D_i\}_{i\in I}$ be the set of irreducible component of $D$. Note that the singular locus $Z_1$ of $D$ and the degenerate locus $Z_2$ of the ramification of $\sF$ along $D$ are finite closed points of $D$ and that the image $f(Z_1\bigcup Z_2)$ is a finite set of closed points of $S$.  Hence, after replacing $S$ by a Zariski neighborhood of $s$, we may assume that each $D_i$ is finite and flat over $S$, that $D_W$ is smooth over $\spec(k)$ and that the ramification of $j_!\sF$ along $D_W$ is non-degenerate, where $W=S-\{s\}$. Note that, for any $t\in W$ and any $x\in D_t$, there is one and only one component $D_i$ contains $x$.

(vi)
Let $t$ be a closed point of $W$. Applying Theorem \ref{DTtheorem} (2) to the closed immersion $h:X_t\to X$ and the sheaf $\sF$, we have
\begin{align}
\chi(t)=\sum_{x\in D_t} \rc_x(\sF|_{U_t})=\sum_{x\in D_t} m_x(h^*\rC_X(j_!\sF))
=\sum_{i\in I} \rc_{D_i}(\sF)\cdot \bigg(\sum_{x\in D_{i,t}} m_x(h^*D_i)\bigg)\label{formulachit}
\end{align}
We put $z=(D_s)_{\mathrm{red}}$.
Applying Proposition \ref{Ctheorem curve} or Theorem \ref{Ctheorem} to the closed immersion $g:X_s\to X$ and the sheaf $\sF$, we have
\begin{align}\label{formulachis}
\chi(s)&=\rc_z(\sF|_{U_t})\leq m_z(g^*\rC_X(j_!\sF))=\sum_{i\in I} \rc_{D_i}(\sF)\cdot m_z(g^*D_i)
\end{align}
Since each $D_i$ is finite and flat over $S$, we have $m_z(g^*D_i)=\sum_{x\in D_{i,t}} m_x(h^*D_i)$. By \eqref{formulachit} and \eqref{formulachis}, we obtain $\chi(t)\geq\chi(s)$ for any $t\in W$. The lower semi-continuity of $\chi:S\to \mathbb Q$ is proved.
\end{proof}

\begin{remark}
When $S$ is a henselian trait, the defect of the lower semi-continuity function $\varphi:S\to\bZ$ is computed in terms of the dimension of the stalk of nearby cycles complex of $j_!\sF$ in \cite{lau}. We expect cohomological descriptions on that of $\chi:S\to\bQ$.
\end{remark}

\section{\'Etale sheaves with pure ramifications (II)}\label{sectionpureram2}

\subsection{}\label{all of affine line}
In this section, let $k$ denote a perfect field of characteristic $p>0$. We denote by $\infty$ the complement of $\bA^1_k=\spec(k[t])$ in $\bP^1_k$, by $0$ the origin of $\bA^1_k$, by $\bG_{m,k}=\spec(k[t, t^{-1}])$ the complement of $0$ in $\bA^1_k$ and by $\jmath:\bG_{m,k}\to\bA^1_k$ the canonical injection. Let $\xi$ be the generic point of  $\spec(\wh \sO_{\bA^1_k,0})=\spec(k[[t]])$, $\ol K$ a separable closure of the fraction field $K=k((t))$ of $\wh \sO_{\bA^1_k,0}$, $G_K$ the Galois group of $\ol K$ over $K$, $P_K$ the wild inertia subgroup of $G_K$ and $\ol{\xi}=\spec(\ol K)$. Let $\Lambda$ be a finite field of characteristic $\ell$ $(\ell\neq p)$.

\begin{lemma}\label{nondegD-00}
Let $n\geq 2$ and $m\geq 1$ be positive integers satisfying $n\geq m$. Let $f:\bA^n_k\to\bA^1_k$ be a $k$-morphism of affine spaces given by
\begin{equation*}
f^{\sharp}:k[t]\to k[x_1,\ldots, x_n],\ \ \ t\mapsto \prod_{i=1}^m x_i^{a_i},
\end{equation*}
where each $a_i$ is a positive integer. We denote by $D_i$  the Carter divisor of $\bA^n_k$ associated to the ideal $(x_i)$ $(1\leq i\leq m)$, by $D$ the sum of all $D_i$, by $U$ the complement of $D$ in $\bA^n_k$ and by $f_0:U\to \bG_{m,k}$ the restriction of $f:\bA^n_k\to\bA^1_k$ on $U$. Let $\sG$ be a locally constant and constructible sheaf of $\Lambda$-modules on $\bG_{m,k}$. Then,
the ramification of $f^*_0\sG$ along the smooth locus of $D$ is non-degenerate.
\end{lemma}
\begin{proof}
We firstly consider the case where $m=1$. In this case, $f:\bA^n_k\to\bA^1_k$ is the composition of $f':\bA^n_k\to \bA^1_k$ and $f'':\bA^1_k\to \bA^1_k$ given by
\begin{align*}
f'^{\sharp}:k[t]&\to k[x_1,\ldots, x_n],\ \ \ t\mapsto x_1,\\
f''^{\sharp}:k[t]&\to k[t],\ \ \ t\mapsto t^{a_1}.
\end{align*}
We denote by $f'_0:U\to \bG_{m,k}$ the restriction of $f':\bA^n_k\to \bA^1_k$ on $U$ and by $f''_0:\bG_{m,k}\to\bG_{m,k}$ the restriction of $f'':\bA^1_k\to \bA^1_k$ on $\bG_{m,k}$. Since $f':\bA^n_k\to \bA^1_k$ is smooth and the ramification of $\sF=f''^*_0\sG$ at $0\in\bA^1_k$ is non-degenerate, the ramification of $f_0^*\sG=f'^*_0\sF$ along the smooth divisor $D$ is non-degenerate.

Now we consider the case where $m\geq 2$. We may assume that $k$ is algebraically closed. Without loss of generality,  we show that the ramification of $f^*_0\sG$ along the smooth locus of $D_1$ is non-degenerate.
Let $\ul u=(u_1,\ldots,u_{m})$ be an $m$-tuple of elements in $k^{\times}$ satisfying $\prod^m_{i=1}u_i^{a_i}=1$ and  $\phi_{\ul u}:\bA^n_k\to\bA^n_k$ an isomorphism given by
\begin{align*}
\phi_{\ul u}^\sharp:k[x_1,\ldots, x_n]&\to k[x_1,\ldots, x_n],\\
x_i&\mapsto u_ix_i,\ \ (1\leq i\leq m),\\
x_i&\mapsto x_i,\ \ (m+1\leq i\leq n).
\end{align*}
and $\psi_{\ul u}:U\to U$ the restriction of $\phi_{\ul u}:\bA^n_k\to\bA^n_k$ on $U$.
We have the following commutative diagrams
\begin{equation*}
\xymatrix{\relax
U\ar[r]^-(0.5){\psi_{\ul u}}\ar[rd]_{f_0}&U\ar[d]^{f_0}&&\bA^n_k\ar[r]^-(0.5){\phi_{\ul u}}\ar[rd]_{f}&\bA^n_k\ar[d]^{f}\\
&\bG_{m,k}&&&\bA^1_k}
\end{equation*}
 Hence, we have
 \begin{equation}\label{psifG=fG}
 f^*_0\sG\cong\psi_{\ul u}^*(f_0^*\sG).
 \end{equation}
The isomorphism $\phi_{\ul u}:\bA^n_k\to\bA^n_k$ fixes each $D_i$.
Let $x$ be a closed point of the smooth locus of $D_1$ such that the ramification of $f_0^*\sG$ along $D$ is non-degenerate at $x$. Since $\phi_{\ul u}:\bA^n_k\to\bA^n_k$ is an isomorphism, the ramification of $ \psi_{\ul u}^*(f_0^*\sG)$ along $D$ is non-degenerate at $\phi_{\ul u}^{-1}(x)$. By the isomorphism \eqref{psifG=fG}, we see that the ramification of $f_0^*\sG$ along $D$ is non-degenerate at $\phi_{\ul u}^{-1}(x)$ for any $\ul u$. For any closed point $x'$ in the smooth locus of $D_1$, we can find $\ul u'$ such that $x'=\phi_{\ul u}^{-1}(x)$. Hence the ramification of $f_0^*\sG$ is non-degenerate along the smooth locus of $D_1$.
  \end{proof}

\begin{proposition}\label{NramD2}
Let $q$ be a positive integer and $b$ a positive integer co-prime to $p$. Let $Y$ be the affine scheme $\spec( k[y^{\pm 1}_1,y_2])$ and $g:Y\to\bA^1_k$ a $k$-morphism given by
\begin{equation*}
g^{\sharp}:k[t]\to k[y^{\pm 1}_1,y_2],\ \ \ t\mapsto y_1^{q}y_2^b.
\end{equation*}
We denote by $E$ the Carter divisor of $Y$ associated to the ideal $(y_2)$, by $W$ the complement of $E$ in $Y$, by $g_0:W\to \bG_{m,k}$ the restriction of $g:Y\to\bA^1_k$ on $W$ and $j:W\to Y$ the canonical injection. Let $\sN$ be a locally constant and constructible sheaf of $\Lambda$-modules on $\bG_{m,k}$ with isoclinic ramification at $0\in\bA^1_k$ of the conductor $a\geq 1$. Then, the ramification of  $g^*_0\sN$ at the generic point of $E$ is isoclinic and logarithmic isoclinic with (Definition \ref{defCDTNP})
\begin{equation*}
\rc_E(g_0^*\sN)=\rlc_E(g_0^*\sN)+1=b(a-1)+1.
\end{equation*}
Moreover, the ramification of  $g^*_0\sN$ along $E$ is non-degenerate and we have
\begin{align*}
SS(j_!g^*_0\sN)&=\bT^*_{Y}Y\bigcup E\cdot\langle dy_2\rangle\\
CC(j_!g^*_0\sN)&=\rk_{\Lambda}(\sN)([\bT^*_YY]+(b(a-1)+1)[E\cdot\langle dy_2\rangle])
\end{align*}
\end{proposition}
\begin{proof}
Let $\beta:\bA^1_k\to\bA^1_k$, $\gamma:Y\to Y$ and $h:Y\to\bA^1_k$ be $k$-morphisms given by
\begin{align*}
&\beta^\sharp:k[t]\to k[t],\ \ \ t\mapsto t^b,\\
&\gamma^\sharp: k[y^{\pm 1}_1,y_2]\to k[y^{\pm 1}_1,y_2],\ \ \ y_1\mapsto y^b_1,\ \ y_2\mapsto y_2,\\
&h^\sharp: k[t]\to k[y^{\pm 1}_1,y_2],\ \ \ t\mapsto y_1^qy_2.
\end{align*}
They give rise to a commutative diagram
\begin{equation*}
\xymatrix{\relax
Y\ar[r]^-(0.5){h}\ar[d]_{\gamma}&\bA^1_k\ar[d]^\beta\\
Y\ar[r]_-(0.5){g}&\bA^1_k}
\end{equation*}
We denote by $\beta_0:\bG_{m,k}\to\bG_{m,k}$ (resp. $h_0:W\to\bG_{m,k}$) the restriction of $\beta: \bA^1_k\to\bA^1_k$ (resp. $h:Y\to\bA^1_k$). Since $\beta_0:\bG_{m,k}\to\bG_{m,k}$ is tamely ramified at $0$ of the ramification index $b$, the ramification of $\sG=\beta_0^*\sN$ at $0$ is isoclinic with the conductor $b(a-1)+1$. Observe that $\gamma:Y\to Y$ is an \'etale map and, for any closed point $y\in E$, the canonical map $d\gamma:\bT^*_{\gamma(y)}Y\to \bT^*_{y}Y$ maps $dy_2$ to $dy_2$. Hence, we are reduced to prove that the ramification of $h_0^*\sG$ at the generic point of $E$ is isoclinic and logarithmic isoclinic with
\begin{equation}\label{cEGlcEG}
\rc_E(h_0^*\sG)=\rlc_E(h_0^*\sG)+1=b(a-1)+1,
\end{equation}
and that
\begin{align}
SS(j_!h^*_0\sG)&=\bT^*_{Y}Y\bigcup E\cdot\langle dy_2\rangle,\label{SSjhG}\\
CC(j_!h^*_0\sG)&=\rk_{\Lambda}(\sG)([\bT^*_YY]+(b(a-1)+1)[E\cdot\langle dy_2\rangle]).\label{CCjhG}
\end{align}
Notice that $h:Y\to\bA^1_k$ is smooth with $E=h^{-1}(0)$. By \cite[Proposition 2.22 and 3.8]{wr} and \cite[Lemma 1.22]{saito cc}, we obtain that the ramification of $h_0^*\sG$ at the generic point of $E$  is isoclinic and logarithmic isoclinic with \eqref{cEGlcEG}. By Lemma \ref{nondegD-00}, we obtain that the ramification of $h_0^*\sG$ along $E$ is non-degenerate. We have $CC(\jmath_!\sG)=-\rk_{\Lambda}(\sG)([\bT^*_{\bA^1_k}\bA^1_k]+(b(a-1)+1)[\{0\}\cdot\langle dt\rangle])$. Observe that for any closed point $y=(a,0)\in E$, the canonical map $dh:\bT^*_0\bA^1_k\to\bT^*_yY$ maps $dt$ to $a^qdy_2$.
Applying \cite[Theorem 7.6]{cc} to $h:Y\to\bA^1_k$ and the sheaf $\jmath_!\sG$, we obtain \eqref{SSjhG} and \eqref{CCjhG}.
\end{proof}

\begin{definition}\label{GKextensiondef}
A finite Galois \'etale cover $V\to\bG_{m,k}$ is called {\it special} if it is tame at $\infty$ and its geometric monodromy group has only one $p$-Sylow subgroup. We denote by $\pi_1^{\mathrm{sp}}(\bG_{m,k},\ol\xi)$ the quotient of $\pi_1(\bG_{m,k},\ol\xi)$ classifying all special covers. The composition of the canonical map $\gal(\ol K/K)=\pi_1(\xi,\ol\xi)\to \pi_1(\bG_{m,k},\ol\xi)$ and the surjection $\pi_1(\bG_{m,k},\ol\xi)\to \pi_1^{\mathrm{sp}}(\bG_{m,k},\ol\xi)$ is an isomorphism of pro-finite groups (\cite[1.4.10]{Katz1}). Hence, for a finitely generated $\Lambda$-module $N$ with a continuous $G_K$-action, we have a locally constant and constructible sheaf of $\Lambda$-modules $\sN$ on $\bG_{m,k}$ which is trivialized by a special cover of $\bG_{m,k}$ such that $\mathscr N|_{\xi}\cong N$ as $G_K$-representations by the isomorphism $\pi_1(\xi,\ol\xi)\xrightarrow{\sim} \pi_1^{\mathrm{sp}}(\bG_{m,k},\ol\xi)$ above. We call $\mathscr N$ a {\it Gabber-Katz extension} of $N$ on $\bG_{m,k}$.
\end{definition}

\subsection{}\label{GKextM}
In the following of this section, $M$ denotes a finitely generated $\Lambda$-module with a continuous $G_K$-action, which is irreducible and factors through a non-trivial finite $p$-group $P_0$ with the induced surjection $P_K\twoheadrightarrow P_0$.  Observe that $M$ is totally wild, isoclinic, logarithmic isoclinic, and the central character decomposition of $M$ is also isotypic (subsection \ref{centerchar}). Let  $\mathscr M$ be a Gabber-Katz extension of $M$ on $\bG_{m,k}$.

\begin{remark}
Galois modules satisfying conditions in subsection \ref{GKextM} are not rare. Assume that $k$ is algebraically closed. Let $N$ be a finitely generated $\Lambda$-module with a continuous $G_K$-action, which is isoclinic with $\rc_K(N)>1$. Then, we can find a finite positive integer $e$ co-prime to $p$ such that, after considering $\sN$ as an representation of an absolute Galois group $G_{K'}$ of $K'=K[t']/(t'^e-t)$ contained in $G_K$, this new action factors through a non-trivial $p$-group $P'$ which is also a quotient of $P_{K'}$. Hence, as a $G_{K'}$-representation, $N$ is a direct sum of finitely many $\Lambda$-modules $\{N_i\}_{i\in J}$ with continuous $G_{K'}$ actions. The $G_{K'}$-action on each $N_i$ satisfying conditions in the subsection \ref{GKextM} with the conductor $\rc_{K'}(N_i)=e(\rc_{K}(N)-1)+1$.
\end{remark}

\begin{proposition}\label{MramD1}
We take a sheaf $\sM$ of $\Lambda$-modules on $\bG_{m,k}$ satisfying conditions in subsection \ref{GKextM}. Let $c$ be the rational number $\rc_0(\sM)$, let $q$ be a power of $p$ with $q>1+\frac{1}{c-1}$ and let $b$ be a positive integer co-prime to $p$. Let $X$ be the affine scheme $\spec( k[x_1,x^{\pm 1}_2])$ and $f:X\to\bA^1_k$ a $k$-morphism given by
\begin{equation*}
f^{\sharp}:k[t]\to k[x_1,x^{\pm 1}_2],\ \ \ t\mapsto x_1^{q}x_2^b.
\end{equation*}
We denote by $D$ the Carter divisor of $X$ associating to the ideal $(x_1)$, by $U$ the complement of $D$ in $X$, by $f_0:U\to \bG_{m,k}$ the restriction of $f:X\to\bA^1_k$ on $U$ and $j:U\to X$ the canonical injection. Then, the ramification of  $f^*_0\sM$ at the generic point of $D$ is isoclinic and logarithmic isoclinic with (Definition \ref{defCDTNP})
\begin{equation*}
\rc_D(f_0^*\sM)=\rlc_D(f_0^*\sM)=q(c-1).
\end{equation*}
Moreover, 
 Furthermore, the ramification of  $f^*_0\sM$ along $D$ is non-degenerate and we have 
\begin{align}
SS(j_!f^*_0\sM)&=\bT^*_{X}X\bigcup D\cdot\langle dx_2\rangle,\label{SSjf0M}\\
CC(j_!f^*_0\sM)&=\rk_{\Lambda}(\sM)([\bT^*_{X}X]+q(c-1)[D\cdot\langle dx_2\rangle]).\label{CCjf0M}
\end{align}
\end{proposition}
\begin{proof}
We may assume that $k$ is algebraically closed.
The morphism $f:X\to\bA^1_k$ is the composition of the following $k$-morphisms $f':X\to X$ and $f'':X\to \bA^1_k$ given by
\begin{align*}
&f'^\sharp: k[x_1,x^{\pm 1}_2]\to k[x_1,x^{\pm 1}_2],\ \ \ x_1\mapsto x_1,\ \ x_2\mapsto x_2^b;\\
&f''^{\sharp}:k[t]\to k[x_1,x^{\pm 1}_2],\ \ \ t\mapsto x_1^{q}x_2.
\end{align*}
Observe that $f':X\to X$ is an \'etale morphism and that, for any closed point $x=(0,a)\in D$, the canonical map $df':\bT^*_{f'(x)}X\to \bT^*_{x}X$ maps $dx_1$ to $dx_1$ and maps $dx_2$ to $ba^{b-1}dx_2$. Hence, to verify the proposition, it is sufficient to treat the morphism $f'':X\to \bA^1_k$ and the sheaf $f'^*_0\sM$ instead, where $f'_0:U\to U$ is the restriction of $f':X\to X$ on $U$. Hence, we only need to consider the case where $b=1$ in the following.

  Let $\eta$ be the generic point of $D$. The fraction field of  $\wh\sO_{X,\eta}$ is $L=k(x_2)((x_1))$. The morphism $f:X\to \bA^1_k$ induces the following $k$-fields extension (subsection \ref{all of affine line})
\begin{equation*}
f^{\diamond}:K\to L,\ \ \ t\mapsto x_1^qx_2.
\end{equation*}
Let $\ol L$ be a separable closure of $L$ containing $\ol K$.  We denote by $G_L$ the Galois group of $\ol L$ over $L$. The field extension $f^\diamond:K\to L$ induces a group morphism $\psi:G_L\to G_K$. We assert that $\psi:G_L\to G_K$ is surjective with $\psi(P_L)=P_K$. Indeed, we have two finite purely inseparable extensions of $k$-fields
\begin{equation*}
F_K:K\to K,\ \ \ t\mapsto t^q\ \ \ \textrm{and}\ \ \ F_L: L\to L,\ \ \ x_1\mapsto x_1,\ \ x_2\mapsto x_2^q,
\end{equation*}
and a $k$-field extension
\begin{equation*}
g^\diamond: K\to L,\ \ \ t\mapsto x_1x_2,
\end{equation*}
that makes the following diagram of field extensions commutative
\begin{equation*}
\xymatrix{\relax
K\ar[r]^{f^\diamond}\ar[d]_{F_K}&L\ar[d]^{F_L}\\
K\ar[r]_{g^{\diamond}}&L}
\end{equation*}
It induces a commutative diagram of group homomorphisms
\begin{equation}\label{GKGLdiag}
\xymatrix{\relax
G_K&G_L\ar[l]_{\psi}\\
G_K\ar[u]^{\theta_K}&G_L\ar[l]^{\varphi}\ar[u]_{\theta_L}}
\end{equation}
Since $g^{\diamond}:K\mapsto L$ maps the uniformizer $t$ of $K$ to a uniformizer $x_1x_2$ of $L$ and the residue field extension $k(x_2)/k$ of $L/K$ is adding a transcendental element $x_2$ on the algebraically closed field $k$, the induced group homomorphism $\varphi:G_L\to G_K$ is surjective with $\varphi(P_L)=P_K$. Since $F_K$ and $F_L$ are finite and purely inseparable, the induced $\theta_K:G_K\to G_K$ and $\theta_L:G_L\to G_L$ are isomorphisms with $\theta_K(P_K)=P_K$ and $\theta_L(P_L)=P_L$. By chasing the diagram \eqref{GKGLdiag}, we see that $\psi:G_L\to G_K$  and its restriction $\psi:P_L\to P_K$ are surjective.

The ramification of $\sM$ at $0$ associates to the $G_K$-module $M$. Hence,
the ramification of $f_0^*\sM$ at $\eta$ associates to $M$ with a continuous $G_L$-action induced by the surjection $\psi:G_L\to G_K$. This $G_L$ action on $M$ also factors through the non-trivial finite $p$-group $P_0$ in subsection \ref{GKextM} and irreducible. Hence the ramification of $f_0^*\sM$ at $\eta$ is isoclinic and logarithmic isoclinic. Let $\{G^r_{K,\log}\}_{r\in \mathbb Q_{\geq 0}}$ and  $\{G^r_{L,\log}\}_{r\in \mathbb Q_{\geq 0}}$ be logarithmic ramification filtrations of $G_K$ and $G_L$, respectively.  By \cite[Lemma 1.22]{saito cc}, we have $\psi(G_{L,\log}^{qr})\subseteq G^r_{K,\log}$ for any $r\in \mathbb Q_{>0}$. Moreover, the restriction map $\psi:G^{qr}_{L,\log}\to G^r_{K,\log}$ induces a surjection
$\ol{\psi}^r:\Gr^{qr}_{\log} G_L\to \Gr^r_{\log}G_K$ for any $r\in \mathbb Q_{>0}$.
 Since $M$ under the continuous $G_K$-action is logarithmic isoclinic with the logarithmic conductor $c-1$, the action of $G^{q(c-1)+}_{L,\log}$ on $M$ through $G^{(c-1)+}_{K,\log}$ is trivial and the action of $\Gr^{q(c-1)}_{\log}G_L$ on $M$ through the surjection $\ol\psi^{(c-1)}:\Gr^{q(c-1)}_{\log} G_L\to \Gr^{(c-1)}_{\log}G_K$ is non-trivial. Hence the ramification of $f^*_0\sM$ at $\eta$ has a logarithmic slope $q(c-1)$. In summary, the ramification of $f_0^*\sM$ at $\eta$ is isoclinic and logarithmic isoclinic with
\begin{equation*}
\rc_D(f^*_0\sM)\geq \rlc_D(f^*_0\sM)=q(c-1).
\end{equation*}
For any $u\in k^{\times}$, let $h(u):\bA^1_k\to X$ be a closed immersion defined by
\begin{equation*}
k[x_1,x_2^{\pm 1}]\to k[t],\ \ \ x_1\mapsto t,\ \ x_2\mapsto u.
\end{equation*}
The composition $f\circ h(u):\bA^1_k\to\bA^1_k$ is given by $k[t]\to k[t],\ \ \ t\mapsto ut^q$. Since $q$ is a power of $p$, the morphism $f\circ h(u):\bA^1_k\to\bA^1_k$ is a Frobenius relative to $\spec(k)$. Hence the ramification of $h(u)_0^*f^*_0\sM$ at $0\in\bA^1_k$ is isoclinic with the conductor $c$ (cf. \eqref{purinsNP}), where $h(u)_0:\bG_{m,k}\to U$ denotes the restriction of $h(u):\bA^1_k\to Y$ on $\bG_{m,k}$. Since $q> 1+\frac{1}{c-1}$ and $c>1$,  we have
\begin{equation}\label{dthfM<dtfM}
\dt_0(h(u)_0^*f^*_0\sM)=\rk_{\Lambda}\sM\cdot c<\rk_{\Lambda}M\cdot q(c-1)\leq \rk_{\Lambda}\sM\cdot \rc_D(f^*_0\sM)=\dt_D(f^*_0\sM).
\end{equation}
By Lemma \ref{nondegD-00}, the ramification of $f^*_0\sM$ along $D$ is non-degenerate. By \cite[Corollary 3.9]{wr} and \eqref{dthfM<dtfM}, we obtain that the morphism $h(u):\bA^1_k\to Y$ is not $SS(j_!f^*_0\sM)$-transversal for each $u\in k^{\times}$. Hence we have
\begin{equation}\label{dx2inSS}
\bT^*_XX\bigcup D\cdot\langle dx_2\rangle\subseteq SS(j_!f^*_0\sM).
\end{equation}
Since the restriction $\psi:P_L\to P_K$ is surjective and the $P_L$-action on $M$ factors through $P_0$ and is irreducible, the central character decomposition of $M$ is isotypic as a $P_L$-representation (subsection \ref{centerchar}). By \cite[Theorem 7.14]{cc} and the fact that the ramification of $f^*_0\sM$ along $D$ is non-degenerate, we obtain that
\begin{equation}\label{SS=C}
SS(j_!f^*_0\sM)=\bT^*_XX\bigcup C,
\end{equation}
where $C$ is an irreducible closed conical subscheme of $\bT^*X$ with $B(C)=D$. By \eqref{dx2inSS} and \eqref{SS=C}, we obtain \eqref{SSjf0M}. Since $\bT^*_DX=D\cdot\langle dx_1\rangle\neq D\cdot\langle dx_2\rangle=SS(j_!f^*_0\sM)\times_XD$, the Newton polygon $\rNP_D(f^*_0\sM)$ coincides with the logarithmic Newton polygon $\rLNP_D(f^*_0\sM)$ (\cite[Proposition 7.2]{Hu19IMRN}). Hence the ramification of $f^*_0\sM$ at $\eta$ is isoclinic and logarithmic isoclinic with $$\rc_D(f^*_0\sM)=\rlc_D(f^*_0\sM)=q(c-1).$$
By \cite[Theorem 7.14]{cc} again and the fact that $\dt_D(f^*_0\sM)=\rk_{\Lambda}(\sM)\cdot q(c-1)$, we get \eqref{CCjf0M}.
\end{proof}

\begin{proposition}\label{all of fM in dim 2}
We take a sheaf $\sM$ of $\Lambda$-modules on $\bG_{m,k}$ satisfying conditions in subsection \ref{GKextM}. Let $c$ be the rational number $\rc_0(\sM)$, $q$ a power of $p$ with $q>1+\frac{1}{c-1}$ and $b$ a positive integer co-prime to $p$. Let $g:\bA^2_k\to\bA^1_k$ be a $k$-morphism of affine spaces given by
\begin{equation*}
g^{\sharp}:k[t]\to k[y_1,y_2],\ \ \ t\mapsto y_1^{q}y_2^b.
\end{equation*}
 We denote by $E_i$ the Carter divisor of $\bA^2_k=\spec(k[y_1,y_2])$ associated to the ideal $(y_i)$ $(i=1,2)$, by $E$ the sum of  $E_1$ and $E_2$, by $V$ the complement of $E$ in $\bA^2_k$ by $j:V\to\bA^2_k$ the canonical injection and by $g_0:V\to \bG_{m,k}$ the restriction of $g:\bA^2_k\to\bA^1_k$ on $U$. Then, 
 \begin{itemize}
 \item[1)]
The ramification of  $g^*_0\sM$ at the generic point of $E_i$ $(i=1,2)$ is isoclinic and logarithmic isoclinic with
 \begin{align*}
 \rc_{E_1}(g^*_0\sM)&=\rlc_{E_1}(g^*_0\sM)=q(c-1),\\
 \rc_{E_2}(g^*_0\sM)&=\rlc_{E_2}(g^*_0\sM)+1=b(c-1)+1.
 \end{align*}
 \item[2)]
The ramification of $g^*_0\sM$ along $E$ is $(E\cdot\langle dy_2\rangle)$-isoclinic by restricting to curves (Definition \ref{defCisoclinic}).
\end{itemize}
\end{proposition}
\begin{proof}
1) is deduced from Proposition \ref{NramD2} and Proposition \ref{MramD1}. 

2). By Proposition \ref{curiso=geniso+pure}, Proposition \ref{NramD2} and Proposition \ref{MramD1}, we see that the ramification of $g^*_0\sM$ along $E-\{(0,0)\}$ is $(E\cdot\langle dy_2\rangle)$-isoclinic by restricting to curves. To show 2), we are left to show that, for any quasi-finite $(E\cdot\langle dy_2\rangle)$-transversal morphism $h:S\to \bA^2_k$ from a smooth $k$-curve $S$ with $s=h^{-1}(E)$ a closed point and $h(s)=(0,0)$, we have 
\begin{equation}\label{csg0MS-s=mshCXj!g0M}
\rc_s((g^*_0\sM)|_{S-\{s\}})=m_s(h^*(\rC_X(j_!g^*_0\sM)))
\end{equation}
 We may assume that $k$ is algebraically closed. Let $g^\sharp:\sO_{\bA^1_{k},0}\to \sO_{\bA^2_{k},(0,0)}$ and $h^\sharp:\sO_{\bA^2_{k},(0,0)}\to \sO_{S,s}$ be induced homomorphisms of local rings. Since $h:S\to \bA^2_k$ is $\langle dy_2\rangle$-transversal at $s$, the image $h^\sharp(y_2)$ in $\sO_{S,s}$ is a uniformizer and we have $m_s(h^*E_2)=1$. We put $\varpi=h^\sharp(y_2)\in \sO_{S,s}$. Since $h(S)$ is not contained in $E_1$, the image $h^\sharp(y_1)$ is a non-zero element of $\sO_{S,s}$. We put $h^\sharp(y_1)=u\varpi^m$, where $u$ is a unit of $\sO_{S,s}$ and $m=m_s(h^*E_1)$ is a positive integer. Hence, the composition $(g\circ h)^\sharp:\sO_{\bA^1_{k},0}\to \sO_{S,s}$ satisfies
$$(g\circ h)^\sharp(t)=h^\sharp(g^\sharp(t))=h^\sharp(y_1^qy_2^b)=u^q\varpi^{mq+b}.$$
Since $mq+b$ is co-prime to $p$, the composition $g\circ h:S\to\bA^1_k$ is tamely ramified at $0$ with the ramification index $mq+b$. Hence $(g^*_0\sM)|_{S-\{s\}}$ is isoclinic with
\begin{align*}
\rc_s((g^*_0\sM)|_{S-\{s\}})&=(mq+b)(\rc_0(\sM)-1)+1=mq(c-1)+(b(c-1)+1)\\
&=q(c-1)\cdot m_s(h^*E_1)+(b(c-1)+1)\cdot m_s(h^*E_2)=m_s(h^*(\rC_X(j_!g^*_0\sM))).
\end{align*}
We obtain \eqref{csg0MS-s=mshCXj!g0M} and 2) is proved.
\end{proof}
%(iii) The non-degeneracy of the ramification of $g^*_0\sM$ along $E-\{(0,0)\}$ is deduced by Lemma \ref{nondegD-00}.
%By part (ii),  for a connected and smooth $k$-curve $S$ and a quasi-finite morphism $h:S\to \bA^2_k$ such that $s=h^{-1}(E)$ is a closed point of $S$ and that $h:S\to X$ is $\langle dy_2\rangle$-transversal at $s$, we have
%\begin{align*}
%\dt_s((g^*_0\sM)|_{S_0})&=\rk_{\Lambda}(M)\cdot \rc_s((g^*_0\sM)|_{S_0})=\rk_{\Lambda}(M)\cdot m_s(h^*(\rC_{\bA^2_k}(j_!g^*_0\sM)))\\
%&=m_s(h^*(\rk_{\Lambda}(M)\cdot\rC_{\bA^2_k}(j_!g^*_0\sM)))=m_s(h^*(\DT_{\bA^2_k}(j_!g^*_0\sM)))
%\end{align*}
%Hence, the ramification of $g^*_0\sM$ along $E$ is $(E\cdot\langle dy_2\rangle)$-pure by restricting to curves. By part (i) and Proposition \ref{criteriaCP}, we obtain \eqref{ThSSjg0M} and \eqref{ThCCjg0M}. Hence, part (iii) holds.

 \begin{theorem}\label{pureMtheorem}
 We take a sheaf $\sM$ of $\Lambda$-modules on $\bG_{m,k}$ satisfying conditions in subsection \ref{GKextM}. Let $c$ be the rational number $\rc_0(\sM)$, let $q$ be a power of $p$ with $q>1+\frac{1}{c-1}$, let $b$ be a positive integer co-prime to $p$, and let $n\geq 2$ and $m\geq 1$ be integers satisfying $n\geq m$. Let $f:\bA^n_k\to\bA^1_k$ be a $k$-morphism of affine spaces given by
 \begin{equation*}
f^{\sharp}:k[t]\to k[x_1,\ldots, x_n],\ \ \ t\mapsto x_1^b,
 \end{equation*}
 when $m=1$ and given by
 \begin{equation*}
f^{\sharp}:k[t]\to k[x_1,\ldots, x_n],\ \ \ t\mapsto x_1^q\cdots x_{m-1}^qx_m^b,
 \end{equation*}
 when $m\geq 2$.
We denote by $D_i$ the Carter divisor of $\bA^n_k=\spec(k[x_1,\ldots,x_n])$ associated to the ideal $(x_i)$ $(1\leq i\leq m)$, by $D$ the sum of  all $D_i$, by $U$ the complement of $D$ in $\bA^n_k$ and by $f_0:U\to \bG_{m,k}$ the restriction of $f:\bA^n_k\to\bA^1_k$ on $U$. Then, 
\begin{itemize}
 \item[(i)]
 The ramification of  $f^*_0\sM$ at the generic point of $D_i$ $(1\leq i\leq m)$ is isoclinic and logarithmic isoclinic with
 \begin{align*}
 \rc_{D_i}(f^*_0\sM)&=\rlc_{D_1}(f^*_0\sM)=q(c-1),\ \ \ \textrm{for}\ \ 1\leq i\leq m-1, \\
 \rc_{D_m}(f^*_0\sM)&=\rlc_{D_m}(f^*_0\sM)+1=b(c-1)+1.
 \end{align*}
  \item[(ii)]
The ramification of $f^*_0\sM$ along $D$ is $(D\cdot\langle dx_m\rangle)$-isoclinic by restricting to curves (Definition \ref{defCisoclinic}). 
%  \item[(ii)]
% For a connected and smooth $k$-curve $S$ and a quasi-finite morphism $h:S\to \bA^n_k$ such that $s=h^{-1}(D)$ is a closed point of $S$ and that $h:S\to X$ is $\langle dx_m\rangle$-transversal at $s$, the ramification of $(f^*_0\sM)|_{S_0}$ at $s$ is isoclinic with
% \begin{equation*}
 %\rc_s((f^*_0\sM)|_{S_0})=m_s(h^*(\rC_{\bA^n_k}(j_!f^*_0\sM)))=\sum_{i=1}^{m-1} q(c-1)\cdot m_s(h^*D_i)+(b(c-1)+1)\cdot m_s(h^*D_m),
 %\end{equation*}
 %where $S_0=S-\{s\}$.
%  \item[(iii)]
% The ramification of $f^*_0\sM$ along the smooth locus of $D$ is non-degenerate. The ramification of $f^*_0\sM$ along $D$ is $(D\cdot\langle dx_m\rangle)$-pure by restricting to curves. We have
%\begin{align}
%SS(j_!f^*_0\sM)&=\bT^*_{\bA^n_k}\bA^n_k\bigcup D\cdot\langle dx_m\rangle\label{ThSSjf0M}\\
%CC(j_!f^*_0\sM)&=(-1)^n\rk_{\Lambda}(\sM)\bigg([\bT^*_{\bA^n}\bA^n_k]+\sum^{m-1}_{i=1}q(c-1)[D_i\cdot\langle dx_m\rangle]\label{ThCCjf0M}\\
 %&\ \ \ \ \ \ \ \ \ \ \ \ \ \ \ \ \ \ \ \ \ +(b(c-1)+1) [D_m\cdot\langle dx_m\rangle]\bigg)\nonumber
%\end{align}
\end{itemize}
 \end{theorem}

\begin{proof}
Firstly, we consider the special case where $m=1$. The morphism $f:\bA^n_k\to\bA^1_k$ is the composition of $f':\bA^n_k\to \bA^1_k$ and $f'':\bA^1_k\to\bA^1_k$ given by
\begin{align*}
f'^{\sharp}:k[t]&\to k[x_1,\ldots, x_n],\ \ \ t\mapsto x_1,\\
f''^{\sharp}:k[t]&\to k[t],\ \ \  t\mapsto t^b.
\end{align*}
We denote by $f'_0:U\to \bG_{m,k}$ the restriction of $f':\bA^n_k\to \bA^1_k$ on $U$ and by $f''_0:\bG_{m,k}\to\bG_{m,k}$ the restriction of $f'':\bA^1_k\to \bA^1_k$ on $\bG_{m,k}$. Since $f'':\bA^1_k\to \bA^1_k$ is tamely ramified at $0$ with the ramification index $b$, the ramification of $f''^*_0\sM$ at $0\in\bA^1_k$ is isoclinic with the conductor $b(c-1)+1$. Since $f':\bA^1_k\to \bA^1$ is smooth with $D=f^{-1}(0)$,  part (i) is obtained by \cite[Proposition 2.22 and 3.8]{wr} and \cite[Lemma 1.22]{saito cc} and part (ii) is obtained from part (i), \cite[Theorem 7.6]{cc} and Proposition \ref{curiso=geniso+pure}.

In the following, we consider the case where $m\geq 2$. We may assume that $k$ is algebraically closed. The morphism  $f:\bA^n_k\to\bA^1_k$ is the composition of $g':\bA^n_k\to \bA^2_k$ and $g:\bA^2_k\to\bA^1_k$ given by
\begin{align*}
g'^{\sharp}:k[y_1,y_2]&\to k[x_1,\ldots, x_n],\ \ \ y_1\mapsto x_1\cdots x_{m-1},\ \ y_2\mapsto x_m,\\
g^{\sharp}:k[t]&\to k[y_1,y_2],\ \ \ t\mapsto y_1^qy_2^b.
\end{align*}
We denote by $E_i$ the divisor of $\bA^2_k=\spec(k[y_1,y_2])$ associates to the ideal $(y_i)$ $(i=1,2)$, by $E$ the sum of $E_1$ and $E_2$, and by $V$ the complement of $E$ in $\bA^2_k$. Put $D'=\sum^{m-1}_{i=1}D_i$. We have $g'^{-1}(E_1)=D'$ and $g'^{-1}(E_2)=D_m$. We denote by $g'_0:U\to V$ the restriction of $g':\bA^n_k\to \bA^2_k$ on $U$ and by $g_0:V\to \bG_{m,k}$ the restriction of $g:\bA^2_k\to\bA^1_k$ on $V$.
For each $1\leq i\leq m$, the morphism $g':\bA^n_k\to \bA^2_k$ is smooth at the generic point of $D_i$. Moreover, for each $1\leq i\leq m-1$, we have $g'(D_i)=E_1$ and $g'(D_m)=E_2$.

By \cite[Proposition 2.22 and 3.8]{wr}, \cite[Lemma 1.22]{saito cc} and Proposition \ref{all of fM in dim 2} 1), we obtain that the ramification of $f_0^*\sM$ at each generic point of $D$ is isoclinic and logarithmic isoclinic with
\begin{align*}
\rc_{D_i}(f^*_0\sM)=\rc_{E_1}(g^*_0\sM)=q(c-1),\ \ \ \rlc_{D_i}(f^*_0\sM)=\rlc_{E_1}(g^*_0\sM)=q(c-1),
\end{align*}
for each $1\leq i\leq m-1$, and
\begin{align*}
\rc_{D_m}(f^*_0\sM)=\rc_{E_2}(g^*_0\sM)=b(c-1)+1,\ \ \ \rlc_{D_m}(f^*_0\sM)=\rlc_{E_2}(g^*_0\sM)=b(c-1).
\end{align*}
We obtain part (i).

Let $S$ be a connected and smooth $k$-curve and $h:S\to \bA^n_k$ a quasi-finite morphism such that $s=h^{-1}(D)$ is a closed point of $S$ and that $h:S\to X$ is $\langle dx_m\rangle$-transversal at $s$. We denote by $h':S\to \bA^2_k$ the composition of $h:S\to \bA^n_k$ and $g':\bA^n_k\to\bA^2_k$. It is quasi-finite with $h'^{-1}(E)=s$ and is $\langle dy_2\rangle$-transversal at $s$. Applying Proposition \ref{all of fM in dim 2} 2) to $g^*_0\sM$, the ramification of $(f^*_0\sM)|_{S_0}=(g^*_0\sM)|_{S_0}$ at $s$ is isoclinic with the conductor
\begin{align*}
\rc_s((f^*_0\sM)|_{S_0})&=q(c-1)\cdot m_s(h'^*E_1)+(b(c-1)+1)\cdot m_s(h'^*E_2)\\
  &=q(c-1)\cdot m_s(h^*D')+(b(c-1)+1)\cdot m_s(h^*D_m))\\
 &=\sum_{i=1}^{m-1} q(c-1)\cdot m_s(h^*D_i)+(b(c-1)+1)\cdot m_s(h^*D_m)\\
 &=m_s(h^*(\rC_{\bA^n_k}(j_!f^*_0\sM))).
\end{align*}
Hence, we obtain part (ii).
\end{proof}

%The non-degeneracy of the ramification of $f^*_0\sM$ along the smooth locus of $D$ is deduced by Lemma \ref{nondegD-00}. By part (ii),  for a connected and smooth $k$-curve $S$ and a quasi-finite morphism $h:S\to \bA^n_k$ such that $s=h^{-1}(D)$ is a closed point of $S$ and that $h:S\to \bA^n_k$ is $\langle dx_m\rangle$-transversal at $s$, we have
%\dt_s((f^*_0\sM)|_{S_0})&=\rk_{\Lambda}(\sM)\cdot \rc_s((f^*_0\sM)|_{S_0})=\rk_{\Lambda}(\sM)\cdot m_s(h^*(\rC_{\bA^n_k}(j_!f^*_0\sM)))\\
%&=m_s(h^*(\rk_{\Lambda}(\sM)\cdot\rC_{\bA^n_k}(j_!f^*_0\sM)))=m_s(h^*(\DT_{\bA^n_k}(j_!f^*_0\sM)))
%\end{align*}
%Hence, the ramification of $f^*_0\sM$ along $D$ is $(D\cdot\langle dx_m\rangle)$-pure by restricting to curves. By part (i) and Proposition \ref{criteriaCP}, we obtain \eqref{ThSSjf0M} and \eqref{ThCCjf0M}. In summary, part (iii) holds.

\begin{corollary}\label{generalKatzLaumon}
We take a sheaf $\sM$ of $\Lambda$-modules on $\bG_{m,k}$ satisfying conditions in subsection \ref{GKextM}. Let $n$ and $m$ be integers satisfying $n\geq m\geq 2$. Let $g:\bA^n_k\to \bA^1_k$ and $\pr:\bA^n_k\to \bA^{n-1}_k$ be $k$-morphisms given by
\begin{align}
&g^\sharp:k[x]\to k[x_1,\ldots, x_n],\ \ \ x\mapsto x_1\cdots x_m; \nonumber\\
&\pr^{\sharp}:k[x_2,\ldots, x_n]\to k[x_1,\ldots,x_n],\ \ \ x_i\mapsto x_i\ \ (2\leq i\leq n).\nonumber
\end{align}
Then, $\pr:\bA^n_k\to \bA^{n-1}_k$ is universally locally acyclic with respect to $g^*\jmath_! \sM$.
\end{corollary}

\begin{proof}
Let $q$ be a power of $p$ with $q>1+\frac{1}{c-1}$, where $c=\rc_0(\sM)$. Let $\mathrm{F}_n:\bA^n_k\to\bA^n_k$ and $\mathrm{F}_{n-1}:\bA^{n-1}_k\to\bA^{n-1}_k$ be two $k$-morphisms given by
\begin{align*}
\mathrm{F}_n^\sharp:k[x_1,\ldots,x_n]&\to k[x_1,\ldots,x_n],\\
 x_i&\mapsto x_i^q, \ \ \ (2\leq i\leq m),\\
 x_r&\mapsto x_r, \ \ \ (r=1\ \ \textrm{or}\ \ m+1\leq r\leq n);\\
 \mathrm{F}_{n-1}^\sharp:k[x_2,\ldots,x_{n}]&\to k[x_2,\ldots,x_{n}], \\
 x_i&\mapsto x_i^q, \ \ \ (2\leq i\leq m),\\
 x_r&\mapsto x_r, \ \ \ (m+1\leq r\leq n).\\
\end{align*}
We have the following diagram with a Cartesian square
\begin{equation}\label{diagramprFn}
\xymatrix{\relax
\bA^{n}_k\ar[r]^{\mathrm{F}_{n}}\ar[d]_-(0.5){\pr}\ar@{}|-{\Box}[rd]&\bA^n_k\ar[d]^-(0.5){\pr}\ar[r]^-(0.5){g}&\bA^1_k\\
\bA^{n-1}_k\ar[r]_{\mathrm{F}_{n-1}}&\bA^{n-1}_k&}
\end{equation}
By Lemma \ref{lemmaisocliniccurve} and Theorem \ref{pureMtheorem}, we have
\begin{equation*}
SS(\mathrm{F}_{n}^*g^*\jmath_!\sM)=\bT^*_{\bA^n_k}\bA^n_k\bigcup D\cdot \langle dx_1\rangle,
\end{equation*}
where $D$ denotes the divisor of $\bA^n_k$ associating to the ideal $(x_1\cdots x_m)$. Hence $\pr:\bA^n_k\to \bA^{n-1}_k$ is $SS(\mathrm{F}_{n}^*g^*\jmath_!\sM)$-transversal, and thus, is universally locally acyclic with respect to $\mathrm{F}_{n}^*g^*\jmath_! \sM$. Since $\mathrm{F}_{n-1}$ is finite, surjective and radicial and the square in \eqref{diagramprFn} is Cartesian, $\pr:\bA^n_k\to \bA^{n-1}_k$ is also universally locally acyclic with respect to $g^*\jmath_! \sM$.
\end{proof}

Corollary \ref{generalKatzLaumon} shares a similarity with  \cite[Theorem 2.4.4]{KL}. In the rest of the section, we prove the following acyclicity property of higher direct image operator for \'etale sheaves with nice ramification behaviors. It plays a crucial role in \S \ref{ramificationboundsection}.

\begin{theorem}\label{keysteprambound}
 Let $f:X\to S$ be a flat morphism of connected and smooth $k$-schemes. We assume that $\dim_k S=1$. Let $s$ be a closed point of $S$, $V$ the complement of $s$ in $S$, $U$ the complement of $D=f^*(s)$ in $X$ and $j:U\to X$ the canonical injection. We assume that the restriction $f_0:U\to V$ of $f:X\to S$ on $U$ is smooth and $D$ is divisor with simple normal crossings on $X$. Let $\{D_i\}_{i\in I}$ be the set of irreducible components of $D$. We have the following diagram
 \begin{equation*}
 \xymatrix{\relax
 U\ar[d]_-(0.5){f_0}\ar[r]^j\ar@{}|-{\Box}[rd]&X\ar[d]^-(0.5)f\ar@{}|-{\Box}[rd]&D\ar[l]\ar[d]\\
 V\ar[r]&S&\ar[l]s
 }
 \end{equation*}
 Let $\sF$ be a locally constant and constructible sheaf of $\Lambda$-modules on $U$ and $\sN$ a locally constant and constructible sheaf of $\Lambda$-modules on $V$. We assume that the ramification of $\sN$ at $s$ is logarithmic isoclinic with
 \begin{equation*}
 \rlc_s(\sN)>\max_{i\in I}\{\rlc_{D_i}(\sF)\}.
 \end{equation*}
 Then, we have
 \begin{equation*}
Rj_*(\sF\otimes_{\Lambda}f^*_0\sN)=j_!(\sF\otimes_{\Lambda}f^*_0\sN).
 \end{equation*}
 \end{theorem}

We first show the following less general proposition below. The proof of the theorem above will be postponed in subsection \ref{proofofpurity}.

\begin{proposition}\label{propRj*=j!}
We take the notation and assumptions of Theorem \ref{keysteprambound} and we further assume that $S=\bA^1_k$ and $s$ is the origin $0\in\bA^1_k$. Let $\sF$ be a locally constant and constructible sheaf of $\Lambda$-modules on $U$ and $\sM$ a locally constant and constructible sheaf on $\bG_{m,k}$ satisfying assumptions in subsection \ref{GKextM} with
  \begin{equation*}
\rc_0(\sM)>\max_{i\in I}\{\rc_{D_i}(\sF)\}+1.
 \end{equation*}
 Then, we have
 \begin{equation}\label{Rj*M=j!M}
Rj_*(\sF\otimes_{\Lambda}f^*_0\sM)=j_!(\sF\otimes_{\Lambda}f^*_0\sM).
 \end{equation}
 %Let $f:X\to \bA^1_k$ be a flat morphism of connected and smooth $k$-schemes, $U$ the complement of $D=f^*(0)$ in $X$ and $j:U\to X$ the canonical injection. We assume that the restriction $f_0:U\to \bG_{m,k}$ of $f:X\to \bA^1_k$ on $U$ is smooth and $D$ is divisor with simple normal crossings on $X$. Let $\{D_i\}_{i\in I}$ be the set of irreducible components of $D$.
 %\begin{equation}
 %\xymatrix{\relax
 %U\ar[d]_-(0.5){f_0}\ar[r]^j\ar@{}|-{\Box}[rd]&X\ar[d]^-(0.5)f\ar@{}|-{\Box}[rd]&D\ar[l]\ar[d]\\
 %\bG_{m,k}\ar[r]_-(0.5){\jmath}&\bA^1_k&\ar[l]\{0\}}
 %\end{equation}
\end{proposition}

 \begin{proof}
This is an \'etale local question. We may assume that $k$ is algebraically closed.  To prove the proposition, it is sufficient to verify \eqref{Rj*M=j!M} in a Zariski neighborhood of a closed point $x\in D$. We fix a closed point $x$ of $D$. We replace $X$ by an affine Zariski neighborhood of $x$ and we denote by $v_1,\ldots, v_n$ a local coordinate of $X$ at $x$ such that the ring homomorphism $f^{\sharp}:k[t]\to\Gamma(X,\sO_X)$ associating $f:X\to \bA^1_k$ sends $t$ to $v_1\cdots v_m$ $(1\leq m\leq n)$. Let $D_i$  be the divisor of $X$ associating to the ideal $(v_i)\subset \Gamma(X,\sO_X)$ $(1\leq i\leq m)$ and we note that $D=\sum^m_{i=1}D_i$. The morphism $f:X\to \mathbb A^1_k$ is a composition of \'etale morphism $g:X\to\bA^n_k$ and $h:\bA^n_k\to \bA^1_k$ which are given by
 \begin{align*}
g^{\sharp}: k[t_1,\ldots, t_n]&\to\Gamma(X,\sO_X),\ \ \  t_i\to v_i\ \  (1\leq i\leq n),\\
h^{\sharp}:k[t]&\to k[t_1,\ldots, t_n],\ \ \  t\mapsto t_1\cdots t_m.
 \end{align*}
We put $c=\rc_0(\sM)$. Let $q$ be a power of $p$ with $q>1+\frac{1}{c-1}$. Let $\varphi:\bA^n_k\to\bA^n_k$ be the $k$-morphism, which  is an identity when $m=1$ and is given by
\begin{align*}
\varphi^\sharp:k[t_1,\ldots, t_n]&\to k[t_1,\ldots, t_n]\\
t_i&\mapsto t_i^q,\ \ (1\leq i\leq m-1),\\
t_i&\mapsto t_i,\ \ (i\geq m),
\end{align*}
when $m\geq 2$. We put $X'=X\times_{\bA^n_k,\varphi}\bA^n_k$ and put  $h'=h\circ\varphi$. We have the following commutative diagram
\begin{equation}\label{diagfrob}
\xymatrix{\relax
X'\ar[d]_-(0.5){\varphi'}\ar[r]^-(0.5){g'}\ar@{}[rd]|-{\Box}&\bA^n_k\ar[d]^{\varphi}\ar[rd]^{h'}&\\
X\ar[r]_-(0.5){g}&\bA^n_k\ar[r]_-(0.5){h}&\bA^1_k}
\end{equation}
We denote by $E_i$ the Cartier divisor of $\bA^n_k$ corresponding to the ideal $(t_i)\subset k[t_1,\ldots, t_n]$ $(1\leq i\leq m)$, by $E$ the sum of $E_i$ $(1\leq i\leq m)$, by $V$ the complement of $E$ in $\bA^n_k$, by $\jmath_n:V\to\bA^n_k$ the canonical injection, by $D'_i$ the pull-back divisor $g'^*{E_i}$ and by $D'$ the sum of $D_i$ $(1\leq i\leq m)$. The base change of \eqref{diagfrob} by $\jmath:\bG_{m,k}\to \bA^1_k$ is following commutative diagram
\begin{equation*}
\xymatrix{\relax
U'\ar[d]_-(0.5){\varphi'_0}\ar[r]^-(0.5){g_0'}\ar@{}[rd]|-{\Box}&V\ar[d]^{\varphi_0}\ar[rd]^{h'_0}&\\
U\ar[r]_-(0.5){g_0}&V\ar[r]_-(0.5){h_0}&\bG_{m,k}}
\end{equation*}
We denote by $j':U'\to X'$ the canonical injection. Note that $\varphi:X'\to X$ is finite, surjective and radicial. Hence, to show \eqref{Rj*M=j!M}, we are left to show
\begin{equation}\label{Rj'*FM=j'!FM}
Rj'_*(\varphi'^*_0\sF\otimes_{\Lambda}\varphi'^*_0f^*_0\sM)=j'_!(\varphi'^*_0\sF\otimes_{\Lambda}\varphi'^*_0f^*_0\sM).
\end{equation}
By Theorem \ref{pureMtheorem}, the ramification of $h'^*_0\sM$ at each generic point of $E$ is isoclinic with 
\begin{align}
&\rc_{D_i}(h'^*_0\sM)=q(c-1)>c\ \ \ 1\leq i\leq m-1,\\
&\rc_{D_m}(h'^*_0\sM)=c,
\end{align}
and that of $h'^*_0\sM$ along $E$ is $(E\cdot\langle dy_m\rangle)$-isoclinic by restricting to curves. Since $g_0':X'\to \bA^n_k$ is \'etale, we obtain that the ramification of $\varphi'^*_0f^*_0\sM=g'^*_0(h'^*_0\sM)$ at each generic point of $D'$ is isoclinic with
\begin{align*}
&\rc_{D'_i}(\varphi'^*_0f^*_0\sM)=q(c-1)>c,\ \ \ (1\leq i\leq m-1),\\
&\rc_{D'_m}(\varphi'^*_0f^*_0\sM)=c, 
\end{align*}
and that of $\varphi'^*_0f^*_0\sM$ along $D'$ is $(D'\cdot\langle dg'(dy_m)\rangle)$-isoclinic by restricting to curves. By \eqref{logcondineq}, we have
 \begin{align*}
q\cdot \rc_{D_i}(\sF)&\geq \rc_{D'_i}(\varphi'^*_0\sF),\ \ \ (1\leq i\leq m-1)\\
\rc_{D_m}(\sF)&\geq \rc_{D'_m}(\varphi'^*_0\sF).
 \end{align*}
  Hence
 \begin{align*}
\rc_{D'_i}(\varphi'^*_0f^*_0\sM)=q(c-1)> q\cdot \rc_{D_i}(\sF)&\geq \rc_{D'_i}(\varphi'^*_0\sF),\ \ \ (1\leq i\leq m-1),\\
\rc_{D'_m}(\varphi'^*_0f^*_0\sM)=c> \rc_{D_m}(\sF)&\geq \rc_{D'_m}(\varphi'^*_0\sF).
 \end{align*}
Applying Proposition \ref{sGCisoclinic+CXjG>CXjF=>Rj*sGsF=j!sGsF} to the sheaf $\varphi'^*_0f^*_0\sM$ and $\varphi'^*_0\sF$ on $U'$, we obtain \eqref{Rj'*FM=j'!FM}. We finishes the proof.
\end{proof}

\begin{lemma}\label{RfF=0RfF0=0}
 Let $f:X\to Y$ be a separated morphism of connected $k$-schemes of finite type and $\sF$ a constructible sheaf of $\Lambda$-modules on $X$. Assume that there is an increasing filtration
 \begin{equation*}
 0=\sF_0\subseteq\sF_1\subseteq\cdots\subseteq \sF_n=\sF
 \end{equation*}
 of $\sF$ consisting of constructible sheaves such that $\sF_i/\sF_{i-1}\cong \sF_1$ for each $1\leq i\leq n$. Let $\ol y\to Y$ be a geometric point. If $(Rf_*\sF)_{\ol y}=0$, then $(Rf_*\sF_i)_{\ol y}=0$ for each $1\leq i\leq n$.
 \end{lemma}

 \begin{proof}
 We proceed by induction on the integer $m$ satisfying the condition:
 \begin{itemize}
 \item[]
 $(\ast)$\ \ \ \ \   $(R^af_*\sF_i)_{\ol y}=0$ for all $a\leq m$ and all $1\leq i\leq n$.
  \end{itemize}
Obviously, $m=-1$ satisfies the condition. Suppose that $m=r-1$ satisfies the condition for an integer $r\geq 0$. Applying the functor $(Rf_*(-))|_{\ol y}$ to the short exact sequence $0\to \sF_{i-1}\to\sF_i\to \sF_1\to 0$, we obtain an exact sequence  $(R^{r-1}f_*\sF_1)|_{\ol y}\to (R^{r}f_*\sF_{i-1})|_{\ol y}\to (R^{r}f_*\sF_i)|_{\ol y}$ for each $1\leq i\leq n$. Since ${r-1}$ satisfies the condition, we have $(R^{r-1}f_*\sF_1)|_{\ol y}=0$. Hence $(R^{r}f_*\sF_{i-1})|_{\ol y}\to (R^{r}f_*\sF_i)|_{\ol y}$ is an injection for each $1\leq i\leq n$. Since $(R^{r}f_*\sF_n)|_{\ol y}=(R^{r}f_*\sF)|_{\ol y}=0$, we obtain that $(R^{r}f_*\sF_i)|_{\ol y}=0$ for all $1\leq i\leq n$. Therefore, $m=r$ satisfies condition $(\ast)$ and the lemma holds by induction.
 \end{proof}

\subsection{Proof of Theorem \ref{keysteprambound}}\label{proofofpurity}
It is an \'etale local question. We may assume that $k$ is algebraically closed. To prove the theorem, it is sufficient to verify
\begin{equation}\label{RjSNx=0}
(Rj_*(\sF\otimes_{\Lambda}f^*_0\sN))|_{x}=0
\end{equation}
for an arbitrary closed point $x$ of $D$. We fix a closed point $x$ of $D$. After replacing $X$ and $S$ by affine Zariski neighborhoods of $x$ and $s$, respectively, we choose a local coordinate $v_1,\ldots, v_n$  of $X$ at $x$ and choose a local coordinate $v$ of $S$ at $s$ such that each irreducible component $D_i$ containing $x$ associates to the ideal $(v_i)\subseteq  \Gamma(X,\sO_X)$ $(i=1,\ldots, m)$, and that $f^{\sharp}(v)=v_1\cdots v_m$. Let $\eta=\spec(L)$ be the generic point of $\spec(\wh \sO_{S,s})$, $\ol L$ a  separable closure of $L$ and $G_L$ the Galois group of $\ol L/L$. The restriction $\sN|_{\eta}$ associates to a finite generated $\Lambda$-module $N$ with a continuous $G_L$-action which factors through a finite quotient $G$. Let $P$ be the image of the wild inertia $P_L$ in $G$. Since $\sN|_{\eta}$ is logarithmic isoclinic and $\rlc_s(\sN)>0$, $P$ is a subgroup of $G$ with $P\neq \{e\}$. The quotient $G/P$ is isomorphic to a cyclic group $\mathbb Z/d_0\mathbb Z$ with $(d_0,p)=1$. Let $d$ be a positive integer satisfying:
 \begin{itemize}
 \item[]
($\star$)\ \ \  $d$ is divisible by $d_0$, co-prime to $p$ and
 $\displaystyle d>\frac{1}{\rlc_s(\sN)-\max_{i\in I}\{\rlc_{D_i}(\sF)\}}.$
 \end{itemize}
  We put
 \begin{equation*}
 S'=\spec(\Gamma(S,\sO_S)[T]/(T^d-v))\ \ \ \textrm{and}\ \ \ X'=\spec(\Gamma(X,\sO_X)[T_1,\ldots, T_m]/(T_1^d-v_1,\ldots, T_r^d-v_m)).
 \end{equation*}
 We denote by $\gamma_S:S'\to S$ and $\gamma_X:X'\to X$ the canonical projections, by $s'\in S'$ the pre-image of $s\in S$, by $V'$ complement of $x'$ in $S'$, by $D'$ (resp. $D'_i$) the reduced Cartier divisor $(\gamma_X^{-1}(D))_{\mathrm{red}}$ (resp. smooth Cartier divisor $(\gamma_X^{-1}(D_i))_{\mathrm{red}}$) of $X'$ and by $U'$ the complement of $D'$ in $X'$. We have the following commutative diagram
 \begin{equation*}
 \xymatrix{\relax
 V'\ar[d]_{\gamma_V}& U'\ar[l]_-(0.5){f'_0}\ar[r]^-(0.5){j'}\ar[d]^{\gamma_U}\ar@{}|-{\Box}[rd]& X'\ar[r]^-(0.5){f'}\ar[d]^{\gamma_X}& S'\ar[d]^{\gamma_S}\\
 V& U\ar[l]^-(0.5){f_0}\ar[r]_-(0.5){j}& X\ar[r]_f& S}
 \end{equation*}
where $f':X'\to S'$ is induced by
\begin{equation*}
f'^{\sharp}:\Gamma(S,\sO_S)\to \Gamma(X,\sO_X),\ \ \  T\mapsto T_1\cdots T_m,
\end{equation*}
and $\gamma_V$,  $\gamma_U$ and $f'_0$ are restrictions of $\gamma_S$, $\gamma_X$ and  $f'$, respectively.
We have
\begin{align}\label{finitebcRj}
\gamma_{X*}Rj'_*(\gamma_U^*\sF\otimes_\Lambda f'^*_0\gamma_V^*\sN)&\xrightarrow{\sim}\gamma_{X*}Rj'_*\gamma_U^*(\sF\otimes_\Lambda f^*_0\sN)\\
&\xrightarrow{\sim} Rj_*\gamma_{U*}\gamma_U^*(\sF\otimes_\Lambda f^*_0\sN)\nonumber\\
&\xrightarrow{\sim}Rj_*(\sF\otimes_\Lambda f^*_0\sN\otimes_{\Lambda} \gamma_{U*}\Lambda),\nonumber
\end{align}
where the third morphism is from the projection formula. Let $x'\in D'$ be the only pre-image of a closed point $x$ of $D$. By the isomorphism \eqref{finitebcRj}, we have
\begin{align*}
(Rj'_*(\gamma_U^*\sF\otimes_\Lambda f'^*_0\gamma_V^*\sN))_{x'}&\cong (\gamma_{X*}Rj'_*(\gamma_U^*\sF\otimes_\Lambda f'^*_0\gamma_V^*\sN))_x\\
&\cong(Rj_*(\sF\otimes_\Lambda f^*_0\sN\otimes_{\Lambda} \gamma_{U*}\Lambda))_x
\end{align*}
The locally constant sheaf $\gamma_{U*}\Lambda$ on $U$ corresponds to the induced representation  $\rho=\mathrm{Ind}^{\gal(U'/U)}_{\{e\}}\Lambda$. Note that $\gal(U'/U)\cong(\mathbb Z/d\mathbb Z)^{m}$. When $d$ is co-prime to $\ell$, the trivial representation of $\gal(U'/U)$ is a direct summand of $\rho$. If $d$ is divisible by $\ell$, the representation $\rho$ has a direct summand $\rho_0$ which is a finite successive extensions of the trivial representation of $\gal(U'/U)$. In summary, $\gamma_{U*}\Lambda$ has a direct summand $\mathscr W_0$, which is a successive extension of constant sheaf $\Lambda$ on $U$. If $(Rj'_*(\gamma_U^*\sF\otimes_\Lambda f'^*_0\gamma_V^*\sN))_{x'}=0$, we obtain $(Rj_*(\sF\otimes_\Lambda f^*_0\sN\otimes_{\Lambda} \mathscr W_0))_x=0$, and further $(Rj_*(\sF\otimes_\Lambda f^*_0\sN))_x=0$ by Lemma \ref{RfF=0RfF0=0}. Thus, to show $(Rj_*(\sF\otimes_\Lambda f^*_0\sN))_x=0$ it suffices to show $(Rj'_*(\gamma_U^*\sF\otimes_\Lambda f'^*_0\gamma_V^*\sN))_{x'}=0$. Since $d$ satisfies condition $(\star)$ above, $\gamma_S:S'\to S$ is tamely ramified at $s$ of the degree $d$ and $\gamma_X:X'\to X$ is tamely ramified along $D$ of the degree $d$ at the generic point of each $D_i$ $(1\leq i\leq m)$, the two sheaves $\gamma_V^*\sN$ and $\gamma_U^*\sF$ satisfy following two conditions:
\begin{itemize}
\item[(1)]
The ramification of $\gamma_V^*\sN$ at $s'$ is logarithmic isoclinic with
\begin{align*}
\rc_{s'}(\gamma_V^*\sN)&=\rlc_{s'}(\gamma_V^*\sN)+1=d\cdot \rlc_{s}(\sN)+1>d\cdot\max_{1\leq i\leq m}\{\rlc_{D_i}(\sF)\}+2\\
&=\max_{1\leq i\leq m}\{d\cdot\rlc_{D_i}(\sF)\}+2=\max_{1\leq i\leq m}\{\rlc_{D'_i}(\gamma_U^*\sF)\}+2\geq \max_{1\leq i\leq m}\{\rc_{D'_i}(\gamma_U^*\sF)\}+1.
\end{align*}
\item[(2)]
Let $\eta'=\spec(L')$ be the generic point of $\spec(\wh\sO_{S',s'})$. We have $L'=L[T]/(T^d-v)$. Embedding $L'$ into $\ol L$, we denote by $G_{L'}$ the Galois group of $\ol L/L'$. The restriction $(\gamma_V^*\sN)|_{\eta'}=\sN|_{\eta'}$ associates to the $\Lambda$-module $N$ with the continuous action of $G_{L'}$ induced by the inclusion $G_{L'}\subseteq G_L$. This $G_{L'}$ action factors through the finite $p$-group $P$ above.
 \end{itemize}
Hence, after changing the notation, we are reduced to proving \eqref{RjSNx=0} under the following refined condition: The ramification of $\sN$ at $s$ is logarithmic isoclinic with
\begin{align}
\rc_{s}(\sN)> \max_{1\leq i\leq m}\{\rc_{D_i}(\sF)\}+1,\nonumber
\end{align}
and moreover, the $G_L$ representation $N$ associated to $\sN|_{\eta}$ factors through a non-trivial finite $p$-group $P$ such that $P_L\twoheadrightarrow P$ is also surjective.

It implies that the $G_L$ representation $N$ is semi-simple. Then, after replacing $S$ by an \'etale neighborhood $T$ of $s$ and replacing $X$ by $X\times_ST$, we have an \'etale map $\pi:S\to\bA^1_k$ mapping $s$ to the origin $0$, and sheaves $\sM_1,\ldots, \sM_r$ on $\bG_{m,k}$ satisfying conditions in subsection \ref{GKextM} with $\bigoplus_{i=1}^r\pi^*_0\sM_i\cong \sN$, where $\pi_0:V\to\bG_{m,k}$ is the restriction of $\pi$ on $V$. We have
\begin{equation*}
\rc_0(\sM_i)=\rc_{s}(\sN)> \max_{1\leq i\leq m}\{\rc_{D_i}(\sF)\}+1,
\end{equation*}
for each $1\leq i\leq r$. By Proposition \ref{propRj*=j!}, we get
 \begin{align*}
 (Rj_*(\sF\otimes_{\Lambda}f^*_0\sN))|_{x}&\cong \bigoplus_{i=1}^r(Rj_*(\sF\otimes_{\Lambda}(\pi_0\circ f_0)^*\sM_i))|_{x}\\
 &=\bigoplus_{i=1}^r(j_!(\sF\otimes_{\Lambda}(\pi_0\circ f_0)^*\sM))|_{x}=0,
 \end{align*}
which finishes the proof. \hfill$\Box$

%we may further assume that the $G_L$-representation $N$ is irreducible. Let $\pi:S\to\bA^1_k$ be an \'etale map given by
%\begin{equation}
%\pi^\sharp: k[t]\to \Gamma(S,\sO_S),\ \ \ t\mapsto v.
%\end{equation}
%and $\pi_0:V\to \bG_{m,k}$ its restriction on $V$.
% It induces an isomorphism between the generic point $\eta=\spec(L)$ of $\spec(\wh \sO_{S,s})$ and $\xi=\spec(K)$ (subsection %\ref{all of affine line}). Let $\sM$ be a Gabber-Katz extension of the $G_L$-representation $N$ on $\bG_{m,k}$. Hence, $\sM$ satisfies assumptions in subsection \ref{GKextM} with
% \begin{equation}
% \rc_{0}(\sM)=\rc_s(\sN)> \max_{1\leq i\leq m}\{\rc_{D_i}(\sF)\}+1.
%  \end{equation}
%  Note that $(\pi_0^*\sM)|_{\eta}\cong \sN|_\eta$. Hence, after replacing again $S$ by an \'etale neighborhood of $s\in S$, we havep $\sN=\pi_0^*\sM$.

% \begin{equation}
% \xymatrix{\relax
 %& U'\ar[dl]_-(0.5){\gamma_U}\ar[dd]\ar[rr]& &X' \ar[dl]^-(0.5){\gamma_X}\ar[dd]\\
%U\ar[dd]\ar[rr] & &X\ar[dd] & \\
% &V' \ar[dl]^-(0.5){\gamma_V}\ar[rr]& &S'\ar[dl]^-(0.5){\gamma_S} \\
% V\ar[rr]& &S & }
% \end{equation}

\section{Ramification bound of nearby cycles}\label{ramificationboundsection}

\subsection{}
In this section, $R$ denotes a henselian discrete valuation ring, $K$ its fraction field, $\ol K$ a separable closure of $K$, $F$ its residue field, $G$ the Galois group of $\ol K/K$ and $P$ the wild inertia subgroup of $G$. We assume that $F$ is perfect of characteristic $p>0$. Let $\mathcal S$ be the spectrum of $R$, $s$ the closed point of $\cS$, $\bar s$ an algebraic geometric point above $s$, $\eta=\spec(K)$ the generic point of $\cS$, and $\eta^\rt$ the maximal tame cover of $\eta$ dominated by $\ol\eta=\spec(\ol K)$. Let $f:\cX\to \cS$ be a morphism of finite type, $f_{\eta}:\cX_\eta\to\eta$ (resp. $f_{\ol s}:\cX_{\ol s}\to \ol s$, resp. $\cX_{\eta^{\rt}}\to\eta^{\rt}$ and resp. $f_{\ol\eta}: \cX_{\ol\eta}\to \ol\eta$) the fiber of $f:\cX\to \cS$ at $\eta$ (resp. $\bar s$, resp. $\eta^{\rt}$ and resp. at $\ol\eta$). Let $\ol \jmath:\cX_{\ol\eta}\to \cX$ (resp. $\jmath^{\rt}:\cX_{\eta^{\rt}}\to \cX$ and resp. $\ol \iota:\cX_{\ol s}\to \cX$) be the base change of the canonical cover $\ol \eta\to \eta$ (resp. $\eta^{\rt}\to \eta$ and resp. $\ol s\to s$) by $f:\cX\to \cS$.

Let $\Lambda$ be a finite field of characteristic $\ell>0$ ($\ell\neq p$). For an object $\sK$ of $D^{+}(\cX_{\eta},\Lambda)$, we denote by $R\Psi(\sK,f)=\ol \iota^*R\ol \jmath_*\ol \jmath^*\sK$ (resp. $R\Psi^{\rt}(\sK,f)=\ol \iota^*R \jmath^{\rt}_{*} \jmath^{\rt*}\sK$) the nearby cycle complex (resp. tame nearby cycle complex) of $\sK$ with respect to $f:\cX\to\cS$. The nearby cycle complex is an object of $D^+(\cX_{\ol s}, \Lambda)$ with a $G$-action and the tame nearby cycle complex is an object of $D^+(\cX_{\ol s}, \Lambda)$ with a $G/P$-action.

\begin{definition}\label{deframificationbound}
 Let $\sK$ be an object of $D^b_c(\cX_{\eta},\Lambda)$.
\begin{itemize}
\item[a)] (\cite{Leal}) We say that a rational number $r$ is a {\it upper numbering slope} of $R\Gamma(\cX_{\ol\eta},\sK|_{\cX_{\ol\eta}})$ if there is an integer $i$ such that $r$ is a upper numbering slope of $H^i(\cX_{\ol\eta},\sK|_{\cX_{\ol\eta}})$ with respect to the continuous $G$-action. We say that {\it the upper numbering ramification of $R\Gamma(\cX_{\ol\eta},\sK|_{\cX_{\ol\eta}})$ is bounded by} a rational number $c$ if the action $G^{c+}_{\log}$ on each $H^i(\cX_{\ol\eta},\sK|_{\cX_{\ol\eta}})$ is trivial.
\item[b)](\cite{HT18})
We say that a rational number $r$ is a {\it upper numbering slope} of $R\Psi(\sK,f)$ if there exists a closed point $x\in \cX_{\ol s}$ and an integer $i$ such that $r$ is a upper numbering slope of $R^i\Psi_x(\sK,f)$ with respect to the continuous $G$-action. We say that {\it the upper numbering ramification of $R\Psi(\sK,f)$ is bounded by} a rational number $c$ if the action $G^{c+}_{\log}$ on each $R^i\Psi(\sK,f)$ is trivial.
\end{itemize}
\end{definition}

 In Definition \ref{deframificationbound}, if the upper numbering ramification of $R\Gamma(\cX_{\ol\eta},\sF|_{\cX_{\ol\eta}})$ (resp. $R\Psi(\sK,f)$) is bounded by $c$, then $c$ is greater than or equal to the supreme of the set of all upper numbering slopes of $R\Gamma(\cX_{\ol\eta},\sF|_{\cX_{\ol\eta}})$ (resp. $R\Psi(\sK,f)$).

\begin{proposition}[{\cite[Lemma 5.3, Corollary 5.4]{HT18}}]\label{HT18nearby}
Let $\sK$ be an object of $D^b_c(\cX_{\eta},\Lambda)$. Then, $r$ is an upper numbering slope of $R\Psi(\sK,f)$ if there exists a locally constant and constructible sheaf $\sN$ of $\Lambda$-modules on $\eta$ whose ramification is logarithmic isoclinic at $s\in \cS$ with $r=\rlc_s(\sN)$ such that
\begin{equation*}
R\Psi^{\rt}(\sF\otimes^L_{\Lambda}f_{\eta}^*\sN,f)\neq 0.
\end{equation*}
In particular, the upper numbering ramification of $R\Psi(\sK,f)$ is bounded by $c$ if and only if, for any locally constant and constructible sheaf $\sN$ of $\Lambda$-modules on $\eta$ whose ramification is logarithmic isoclinic at $s\in \cS$ with $\rlc_s(\sN)>c$, we have
\begin{equation*}
R\Psi^{\rt}(\sF\otimes^L_{\Lambda}f_{\eta}^*\sN,f)=0.
\end{equation*}
\end{proposition}

\begin{remark}
The complex $R\Psi^{\rt}(\sK\otimes^L_{\Lambda}f_{\eta}^*\sN,f)$ was firstly studied in \cite{T15}, motivated by $D$-modules theory (cf. \cite{T16}). In \cite{T15}, a number $b$ with a locally constant and constructible sheaf $\sN$ of $\Lambda$-modules on $\eta$ with isoclinic logarithmic slope $b$ at $s$
such that $R\Psi^{\rt}(\sK\otimes^L_{\Lambda}f_{\eta}^*\sN,f)\neq 0$ is called a {\it nearby slope} of $\sK$.
\end{remark}

\begin{definition}\label{semistablepair}
Let $\cZ$ be a reduced closed subscheme of $\cX$. We say $(\cX,\cZ)$ is a {\it semi-stable pair over } $\cS$ if, \'etale locally, $\cX$ is \'etale over a $\cS$-scheme
\begin{equation*}
\spec(R[t_1,\cdots, t_d]/(t_{r+1}\cdots t_d-\pi)),
\end{equation*}
where $r<d$ and $\pi$ is a uniformizer of $R$, and if $\cZ=\cZ_f\bigcup \cX_s$ with $\cZ_f$ defined by an ideal $(t_1\cdots t_m)\subset R[t_1,\cdots, t_d]/(t_{r+1}\cdots t_d-\pi)$ with $m\leq r$.
\end{definition}

In \cite{Leal}, Leal proposed the following conjecture concerning the ramification bound of \'etale cohomology groups.

\begin{conjecture}[{I. Leal, \cite[Conjecture 1]{Leal}}]\label{Lealconj}
Let $(\cX,\cZ)$ be a semi-stable pair over $\cS$ with $f:\cX\to\cS$ proper. Let $\sF$ be a locally constant and constructible sheaf of $\Lambda$-modules on $\cU=\cX-\cZ$ such that its ramification at generic points of $\cZ_f$ is tame and let $\rlc(\sF)$ be the maximum of the set of logarithmic conductors of $\sF$ at generic points of the special fiber $\cX_s$. Then, the upper numbering ramification of $R\Gamma_c(\cU_{\ol\eta},\sF|_{\cU_{\ol\eta}})$ is bounded by $\rlc(\sF)$. 
\end{conjecture}

The main result of this section is the following:

\begin{theorem}\label{maintheoremnearbycycle}
We assume that $\cS$ a henselization of a smooth curve over a perfect field $k$ of characteristic $p>0$ at a closed point. Let $(\cX,\cZ)$ be a semi-stable pair over $\cS$, $\cU$ the complement of $\cZ$ in $\cX$ and $j:\cU\to\cX$ the canonical injection. Let $\sF$ be a locally constant and constructible sheaf of $\Lambda$-modules on $\cU$ such that its ramification at generic points of $\cZ_f$ is tame and let $\rlc(\sF)$ be the maximum of the set of logarithmic conductors of $\sF$ at generic points of the special fiber $\cX_s$. Then, the upper numbering ramification of $R\Psi(j_!\sF,f)$ is bounded by $\rlc(\sF)$.
   \end{theorem}

\begin{proof}
It is sufficient to show that, for any locally constant and constructible sheaf of $\Lambda$-modules $\sN$ on $\eta$ whose ramification is logarithmic isoclinic  at $s\in \cS$ with $\rlc_s(\sN)>\rlc(\sF)$, we have (Proposition \ref{HT18nearby})
\begin{equation}\label{RPsitjFfN=0}
 R\Psi^{\rt}(j_!\sF\otimes^L_{\Lambda}f_{\eta}^*\sN,f)= 0.
\end{equation}

We first prove \eqref{RPsitjFfN=0} with an extra condition that $\cZ=\cX_{s}$. We may assume that $k$ is algebraically closed. Let $n$ be an integer co-prime to $p$ and $\pi$ a uniformizer of $R$. We put $\cS_n=\spec(R[T]/(T^n-\pi))$ and put $\cX_n=\cX\times_{\cS}\cS_n$. We have the following commutative diagram
\begin{equation*}
\xymatrix{\relax
\cU_n\ar[r]^-(0.5){j_n}\ar[d]_-(0.5){h_n}\ar@{}|-{\Box}[rd]&\cX_n\ar[d]^-(0.5){g_n}&\cX_s\ar[l]_-(0.5){\iota_n}\ar@{=}[d]\\
\cU\ar[r]_-(0.5){j}&\cX&\cX_s\ar[l]^-(0.5){\iota}}
\end{equation*}
By \cite[VII 5.11]{SGA4II}, we have
\begin{equation*}
R\Psi^{\rt}(\sF\otimes_{\Lambda}f^*_{\eta}\sN,f)=\varinjlim_{(n,p)=1} \iota_n^*Rj_{n*}h_n^*(\sF\otimes_{\Lambda}f^*_{\eta}\sN)
\end{equation*}
Hence, it is sufficient to show that, for any positive integer $n$ co-prime to $p$ and any closed point $x$ of $\cX_s$, we have
\begin{equation}\label{Rjn*hnFfNx=0}
(Rj_{n*}h_n^*(\sF\otimes_{\Lambda}f^*_{\eta}\sN))|_x=0.
\end{equation}
We fix a closed point $x$ of $\cX_s$. Since $\cS$ is a henselization of of a smooth $k$-curve at a closed point, $f:\cX\to \cS$ is of finite type and $\sF$ and $\sN$ are constructible, we may descend to the case where $\cS$ is an affine and smooth $k$-curve with a closed point $s$. We still denote by $\pi$ the local coordinate of $\cS$ at $s$ associating to the uniformizer. We put $\cV=\cS-\{s\}$ and we denote by $f_{\cV}:\cU\to \cV$ the restriction of $f:\cX\to\cS$ on $\cV$.
We replace $\cX$ by an affine Zariski open neighborhood of $x$ and we choose a local coordinate $t_1,\ldots, t_d$ of $\cX$ at $x$ such that $f^{\sharp}(\pi)=t_{r+1}\cdots t_d$.
We put
\begin{equation*}
\wt\cX_n=\spec(\Gamma(\cX,\sO_{\cX})[T_{r+1},\ldots,T_d]/(T_{r+1}^n-t_{r+1},\ldots, T_d^n-t_d)).
\end{equation*}
Note that $\wt \cX_n$ is a smooth $k$-scheme but $\cX_n$ is not in general. The canonical projection $\wt g_n:\wt\cX_n\to \cX$ is a composition of $\ol g_n:\wt \cX_n\to\cX_n$ and $g_n:\cX_n\to \cX$, where the former is characterized by $\ol g_n(T)=T_{r+1}\cdots T_d$. We have the following commutative diagram
\begin{equation*}
\xymatrix{\relax
\wt\cU_n\ar[d]_-(0.5){\ol h_n}\ar[r]^-(0.5){\wt j_n}\ar@{}|-{\Box}[rd]&\wt\cX_n\ar[d]^-(0.5){\ol g_n}&\cX_s\ar[l]_-(0.5){\wt \iota_n}\ar@{=}[d]\\
\cU_n\ar[r]^-(0.5){j_n}\ar[d]_-(0.5){h_n}\ar@{}|-{\Box}[rd]&\cX_n\ar[d]^-(0.5){g_n}&\cX_s\ar[l]_-(0.5){\iota_n}\ar@{=}[d]\\
\cU\ar[r]_-(0.5){j}&\cX&\cX_s\ar[l]^-(0.5){\iota}}
\end{equation*}
We put $\wt h_n=h_n\circ \ol h_n$. Since $\ol g_{n}:\wt\cX_n\to\cX_n$ is finite and $(\wt \cX_n\times_{\cX_n}\cX_s)_{\mathrm{red}}=\cX_s$, we have
\begin{align}
(R\wt j_{n*}\wt h_n^*(\sF\otimes_{\Lambda}f^*_{\cV}\sN))|_x&\cong  (R\ol g_{n*}R\wt j_{n*}\wt h_n^*(\sF\otimes_{\Lambda}f^*_{\cV}\sN))|_x\label{singtosmooth}\\
&\cong(R j_{n*}(\ol h_{n*}\ol h_n^*(h_n^*(\sF\otimes_{\Lambda}f^*_{\cV}\sN))))|_x\nonumber\\
&\cong (R j_{n*}(h_n^*(\sF\otimes_{\Lambda}f^*_{\cV}\sN)\otimes_{\Lambda}\ol h_{n*}\Lambda))|_x,\nonumber
\end{align}
where the third isomorphism is from the projection formula. Since $\ol h_n:\wt \cU_n\to\cU_n$ is a finite Galois \'etale cover, the  locally constant sheaf $\ol h_{n*}\Lambda$ corresponds to the induced representation  $\rho=\mathrm{Ind}^{\gal(\wt \cU_n/\cU_n)}_{\{e\}}\Lambda$. Note that $\gal(\wt \cU_n/\cU_n)\cong (\mathbb Z/n\mathbb Z)^{d-r-1}$. When $n$ is co-prime to $\ell$, the trivial representation of $\gal(\wt\cU_n/\cU_n)$ is a direct summand of $\rho$. If $d$ is divisible by $\ell$, the representation $\rho$ has a direct summand $\rho_0$ which is a finite successive extension of the trivial representation of $\gal(\wt\cU_n/\cU_n)$. In summary, $\ol h_{n*}\Lambda$ has a direct summand $\mathscr W$, which is a successive extension of constant sheaf $\Lambda$ on $\cU_n$. To show \eqref{Rjn*hnFfNx=0}, it is sufficient to show $(R j_{n*}(h_n^*(\sF\otimes_{\Lambda}f^*_{\cV}\sN)\otimes_{\Lambda}\mathscr W))|_x=0$ (Lemma \ref{RfF=0RfF0=0}). Hence, we are reduced to show  $(R\wt j_{n*}\wt h_n^*(\sF\otimes_{\Lambda}f^*_{\cV}\sN))|_x=0$ by \eqref{singtosmooth}. We denote by $s'\in \cS_n$ the pre-image of $s\in \cS$ from the canonical cover $\varphi:\cS_n\to \cS$ and we put $\cV_n=\cS_n-\{s'\}$. We have the following commutative diagram
\begin{equation*}
\xymatrix{\relax
\wt\cU_n\ar[r]^-(0.5){f_{\cV_n}}\ar[d]_{\wt h_n}&\cV_n\ar[d]^{\varphi_\cV}\\
\cU\ar[r]_-(0.5){f_{\cV}}&\cV}
\end{equation*}
 We denote by $\rlc(\wt h^*_n\sF)$ the maximum of the set of logarithmic conductors of $\wt h_n^*\sF$ at generic points of the divisor $\cX_n\subset \wt \cX_n$. Since the cover $\wt h_n:\wt\cX_n\to \cX$ is tamely ramified at each generic point of $\cX_s\subset \cX$ with the same ramification index $n$, we have $\rlc(\wt h^*_n\sF)=n\cdot\rlc(\sF)$.
 Since $\varphi:\cS_n\to \cS$ is tamely ramified at $s$ with the ramification index $n$, the ramification of $\varphi_\cV^*\sN$ at $s'$ is logarithmic isoclinic with $$\rlc_{s'}(\varphi_\cV^*\sN)=n\cdot \rlc_{s}(\sN)>n\cdot\rlc(\sF)=\rlc(\wt h^*_n\sF).$$ By Theorem \ref{keysteprambound}, we obtain that
\begin{equation}
(R\wt j_{n*}\wt h_n^*(\sF\otimes_{\Lambda}f^*_{\cV}\sN))|_x=(\wt j_{n!}(\wt h_n^*\sF\otimes_{\Lambda}f^*_{\cV_n}(\varphi_{\cV}^*\sN))))|_x=0.
\end{equation}
We finish the proof of the theorem under the extra condition $\cZ=\cX_s$.

 By a similar argument as the Step 2 of \cite[Theorem 5.7]{HT18}, we may reduce the proof of \eqref{RPsitjFfN=0} in the general situation to the special case $\cZ=\cX_s$ above. The key ingredients of this reduction are Abhyankar's lemma (\cite[XIII 5.2]{SGA1}) and Theorem \ref{LCtheorem}. We omit the detail of the d\'evissage.
\end{proof}

\begin{corollary}\label{Lealcoro}
Conjecture \ref{Lealconj} is true when $\cS$ is a henselization of a smooth curve over a perfect field $k$ of characteristic $p>0$ at a closed point.
\end{corollary}

\begin{proof}
Since $f:\cX\to\cS$ is proper, we have a $G$-equivariant spectral sequence
\begin{equation*}
E^{ab}_2=H^a(\cX_{\bar s}, R^b\Psi(j_!\sF,f))\Rightarrow H^{a+b}_c(\cU|_{\ol\eta},\sF|_{\cU_{\ol\eta}}).
\end{equation*}
By Theorem \ref{maintheoremnearbycycle}, $G^{\rlc(\sF)+}_{\log}$ acts trivially on each $E^{ab}_2$.  Since $G^{\rlc(\sF)+}_{\log}$ is a pro-$p$ group and all vector spaces $E_2^{ab}$ and $H^i_c(\cU|_{\ol\eta},\sF|_{\cU_{\ol\eta}})$ are $\ell$-torsion, the action of $G^{\rlc(\sF)+}_{\log}$ on each $H^i_c(\cU|_{\ol\eta},\sF|_{\cU_{\ol\eta}})$ is also trivial.
\end{proof}

%\begin{remark}
%In \cite{Leal}, Leal proved Conjecture \ref{Lealconj} when $S$ is equal characteristic, $f:\cX\to\cS$ has relative dimension $1$ and $\rk_{\Lambda}\sF=1$. Her adopted a global method using Kato and Saito's conductor formula (\cite{KS13}). In a previous joint work with Teyssier \cite{HT18}, we obtained Theoerem \ref{maintheoremnearbycycle} and Corollary \ref{Lealcoro} under an extra condition that $f:\cX\to\cS$ is smooth. Thanks to the study of \'etale sheaves with pure ramifications along general divisors in \S \ref{pureramificationsection1} and \S \ref{sectionpureram2}, Leal's conjecture in a geometric setting can be fully solved in Corollary \ref{Lealcoro}.
%\end{remark}

\subsection*{Acknowledgement}
The author would like to express his gratitude to T. Saito for inspiring comments and valuable suggestions. Particularly, his comments improve results in \S \ref{pureramificationsection1} and \S\ref{semi-contcondsection}. The author thank J.-B. Teyssier for helpful discussions on applications and on related topic in $D$-modules theory.
This research is partially supported by the National Natural Science Foundation of China (grant No. 11901287), the Natural Science Foundation of Jiangsu Province (grant No. BK20190288), the Fundamental Research Funds for the Central Universities and the Nanjing Science and Technology Innovation Project.

\end{document}